\documentclass[12pt]{article}

\usepackage{amssymb}
\usepackage{url}
\usepackage{amsthm}
\usepackage{amsmath}
\usepackage{graphicx}
\usepackage{subcaption}
\newcommand\newsubcap[1]{\phantomcaption%
       \caption*{\figurename~\thefigure(\thesubfigure): #1}}
\usepackage{fullpage}
\usepackage{color}
\usepackage{enumerate}
 \numberwithin{equation}{section}
\usepackage{booktabs}
\allowdisplaybreaks

 \usepackage{mathtools}
 \mathtoolsset{showonlyrefs}
 \usepackage[nocompress]{cite}

\theoremstyle{plain}
\newtheorem{thm}{Theorem}[section]
\newtheorem{cor}[thm]{Corollary}

\newtheorem{lem}[thm]{Lemma}
\newtheorem{prop}[thm]{Proposition}

\theoremstyle{definition}

\newtheorem{ex}[thm]{Example}

\theoremstyle{remark}

\newtheorem{rem}[thm]{Remark}

\newcommand{\N}{\mathbb{N}}
\newcommand{\R}{\mathbb{R}}

\newcommand{\bp}{\begin{proof}[\ensuremath{\mathbf{Proof}}]}
\newcommand{\bs}{\begin{proof}[\ensuremath{\mathbf{Solution}}]}
\newcommand{\ep}{\end{proof}}
\newcommand{\be}{\begin{equation}}
\newcommand{\ee}{\end{equation}}

\begin{document}

\title{The density of Meissner polyhedra}

\author{Ryan Hynd\footnote{Department of Mathematics, University of Pennsylvania. This work was supported in part by an American Mathematical Society Claytor--Gilmer fellowship. }}

\maketitle

\begin{abstract}
We consider Meissner polyhedra in $\R^3$. These are constant width bodies whose boundaries consist of pieces of spheres and spindle tori.  We define these shapes by taking appropriate intersections of congruent balls and show that they are dense within the space of constant width bodies in the Hausdorff topology. This density assertion was essentially established by Sallee. However, we offer a modern viewpoint taking into consideration the recent progress in understanding ball polyhedra and in constructing constant width bodies based on these shapes. 
\end{abstract}


\tableofcontents

\section{Introduction}
A convex body $K$ is a convex and compact subset of $ \R^n$. We'll say that $K$ has {\it constant width} if each pair of parallel supporting planes for $K$ are separated by unit distance. The simplest example is a closed ball of radius $1/2$.  As we shall see below, there are many other constant width shapes.    The purpose of this note is to discuss a family of constant width bodies in $\R^3$ which approximate all other three dimensional constant width shapes.  
\begin{figure}[h]
\centering
 \includegraphics[width=.4\textwidth]{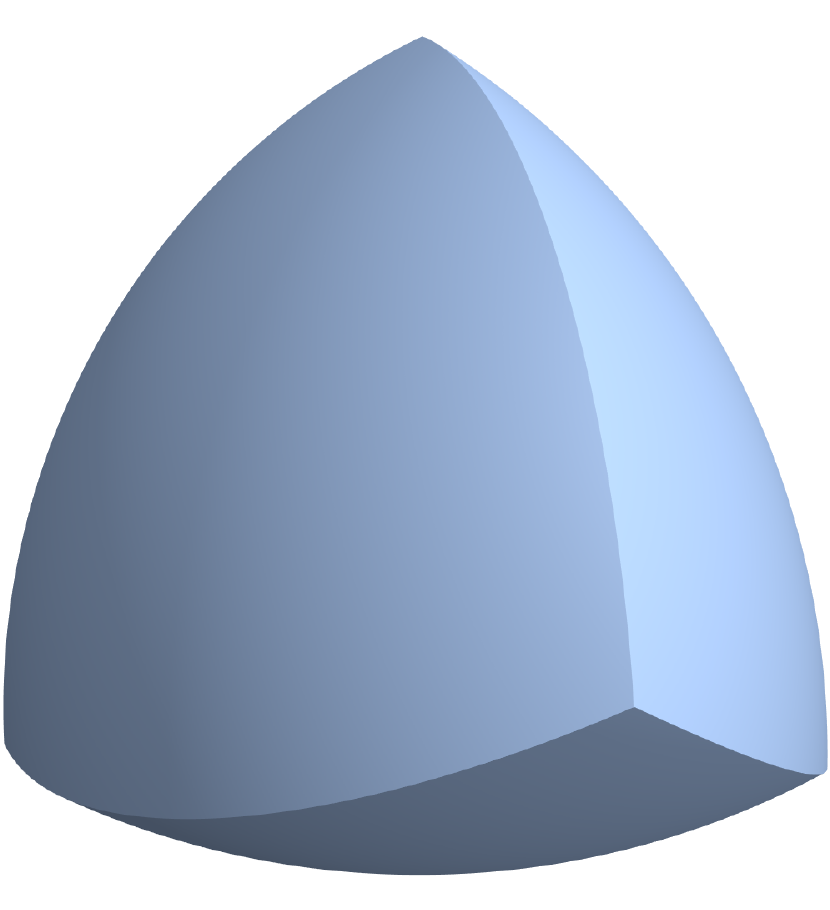}
 \hspace{.2in}
  \includegraphics[width=.4\textwidth]{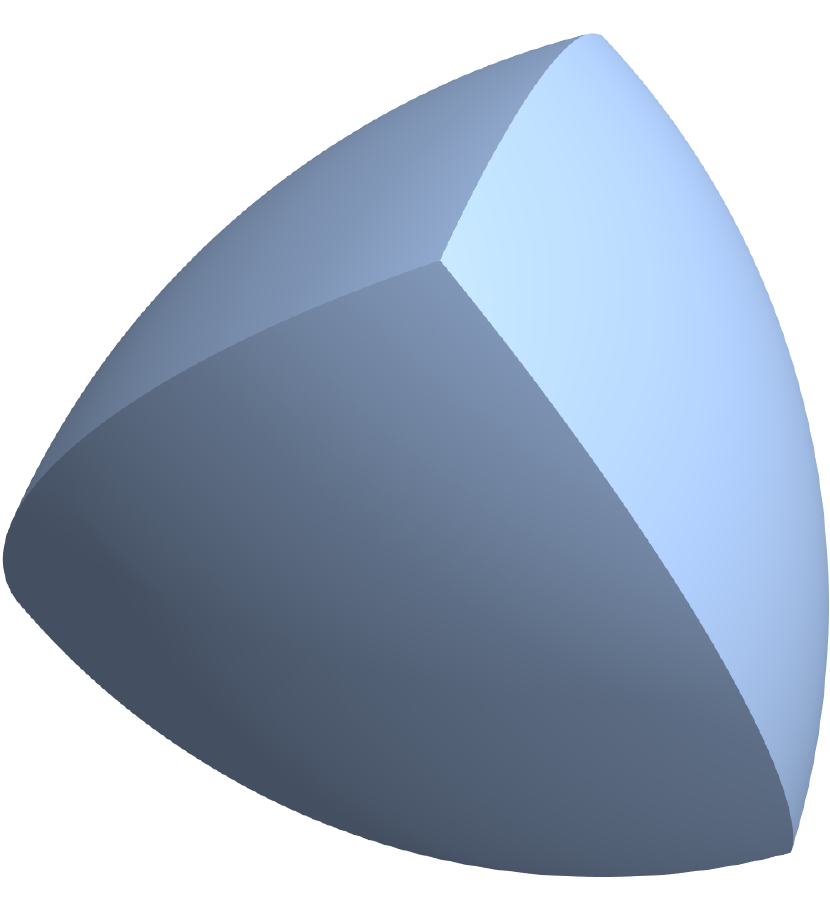}
 \caption{Here are two views of a Reuleaux tetrahedron, which is the intersection of four balls of radius one centered at the vertices of a regular tetrahedron.}\label{ReuleauxTetra}
\end{figure}
\par As it will motivate much of what follows, let us recall a geometric construction of two constant width bodies due to Meissner and Schilling \cite{Meissner}.  Consider the shape 
$$
R=B(a_1)\cap B(a_2)\cap B(a_3)\cap B(a_4),
$$
where $a_1,a_2,a_3,a_4\in \R^3$ and 
$$
|a_i-a_j|=1
$$
for $i\neq j$. Here $|x|$ denotes the Euclidean norm of $x\in \R^3$, and $B(x)\subset \R^3$ denotes a closed ball of radius one centered at $x$. It is evident that $R\subset \R^3$ is a convex body, and it is known as a {\it Reuleaux tetrahedron}.  

\par Similar to a regular tetrahedron in $\R^3$, $R$ has has four vertices, six edges, and  four faces. See Figure \ref{ReuleauxTetra}. The vertices of $R$ are the centers $a_1,a_2,a_3,a_4$. The faces of $R$ are each part of sphere centered at an opposing vertex. It turns out that each of these faces are geodesically convex in their respective spheres.  For example, the face $\partial B(a_1)\cap R$ opposite the vertex $a_1$ is a geodesically convex subset of  $\partial B(a_1)$. Each edge is the intersection of two faces and is a circular arc in both of the spheres that determine the faces. For instance, the edge that joins $a_2$ and $a_3$ is the intersection of the faces $\partial B(a_1)\cap R$ and $\partial B(a_4)\cap R$.

\par We may alter the boundary of $R$ near its edge joining $a_2$ and $a_3$ as follows.   Consider 
$$
\gamma_1\subset \partial B(a_1)\text{ the geodesic joining $a_2$ to $a_3$}
$$
and 
$$
\gamma_4\subset \partial B(a_4)\text{ the geodesic joining $a_2$ to $a_3$}.
$$
As we noted above, $\gamma_1$ and $\gamma_4$ are curves which are included in $\partial R$.  We can then  replace the region of $\partial R$ which contains the edge between $a_2$ and $a_3$ and is bounded by $\gamma_1$ and $\gamma_4$ with a piece of a spindle torus; that is, we can replace this portion of $\partial R$ by the surface obtained by rotating $\gamma_1$ into $\gamma_4$ about the line passing through $a_2$ and $a_3$. The resulting shape bounds a convex body which is a subset of $R$.   See Figure \ref{GeodesicFig}.  
\begin{figure}[h]
\centering
 \includegraphics[width=.49\textwidth]{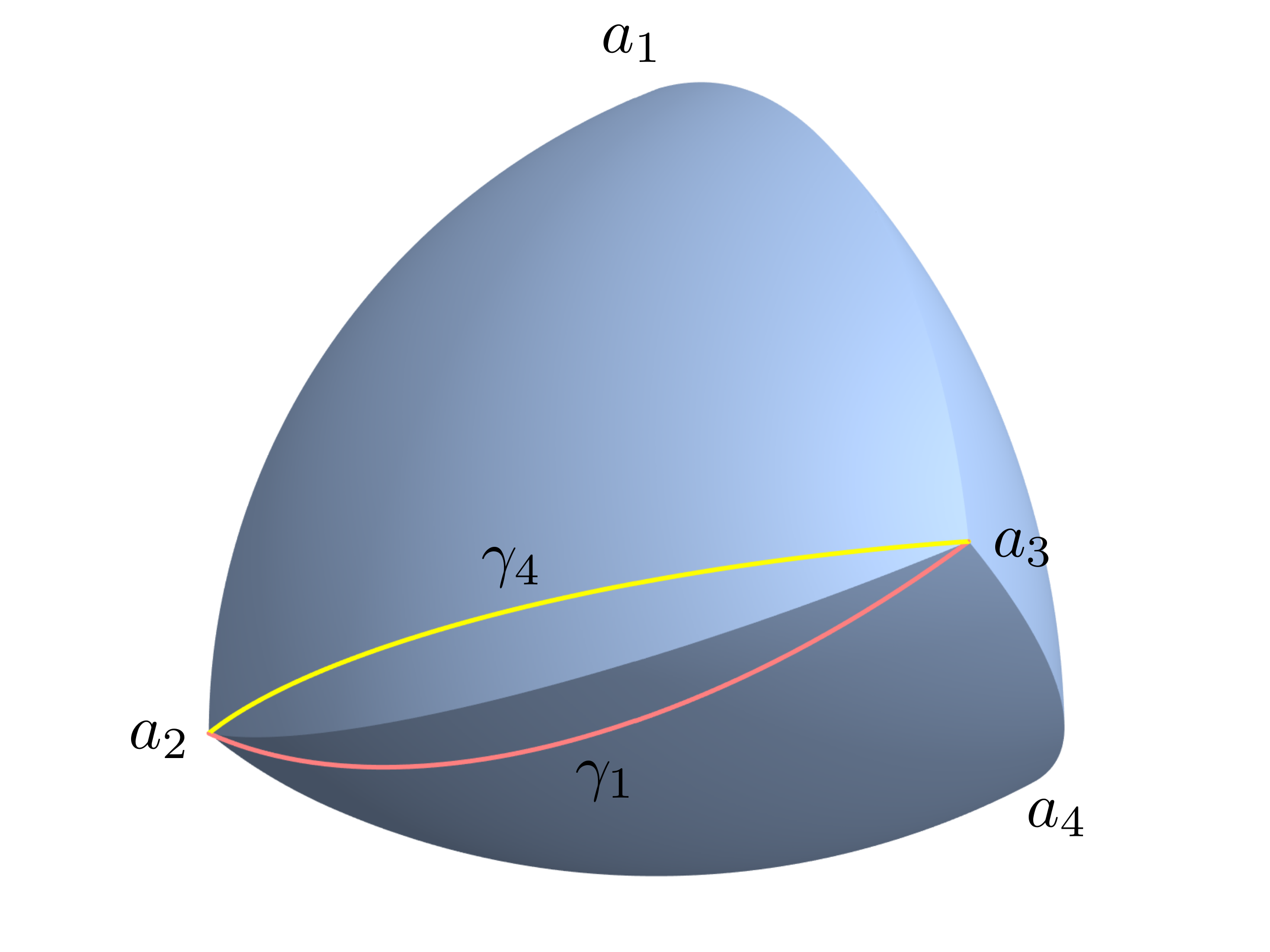}
  \includegraphics[width=.49\textwidth]{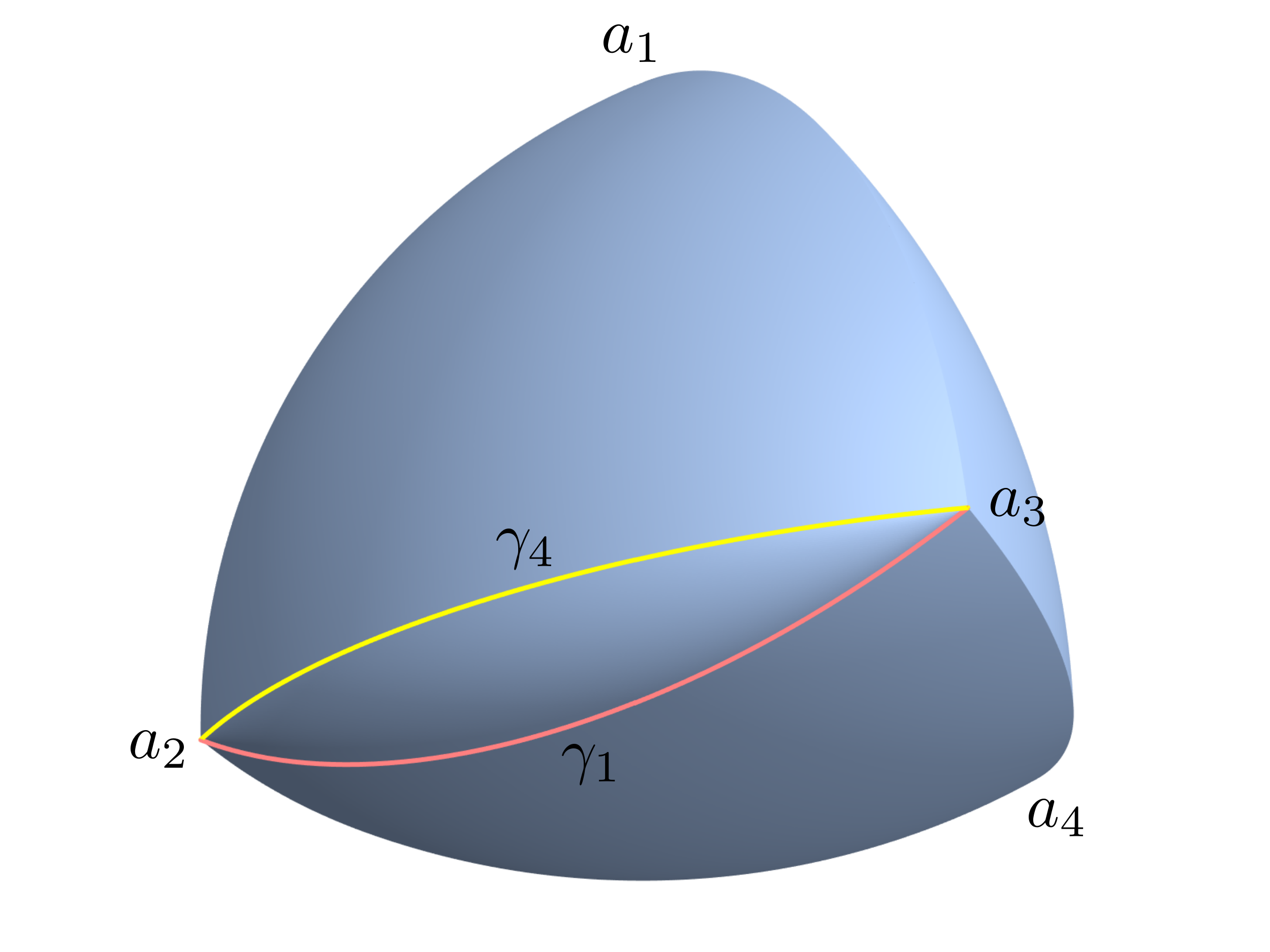}
 \caption{A Reuleaux tetrahedron $R$ on the left with geodesic curves $\gamma_1$ and $\gamma_4$. On the right is the figure obtained by replacing the region of $\partial R$ bounded by $\gamma_1$ and $\gamma_4$ with the shape we get when we rotate $\gamma_1$ into $\gamma_4$ about the line passing through $a_2$ and $a_3$.  This is a basic surgery operation.}\label{GeodesicFig}
\end{figure}

\par If we perform this type of surgery on the region of $\partial R$ near any three edges which either share a common vertex or share a common face, we obtain one of the two Meissner tetrahedra.  See Figures \ref{Meiss1} and \ref{Meiss2}; we also refer the reader to diagrams 106 and 107 in the classic text by Yaglom and Boltyanskii \cite{MR0123962} for images of these constructions.    It turns out that these shapes have constant width.  Moreover, these shapes have been of particular interest for a number of years as they have been conjectured to enclose the least volume among all constant width shapes.   See the article by Kawohl and  Weber \cite{MR2844102} for a recent survey of these shapes.  
\begin{figure}[h]
\centering
 \includegraphics[width=.4\textwidth]{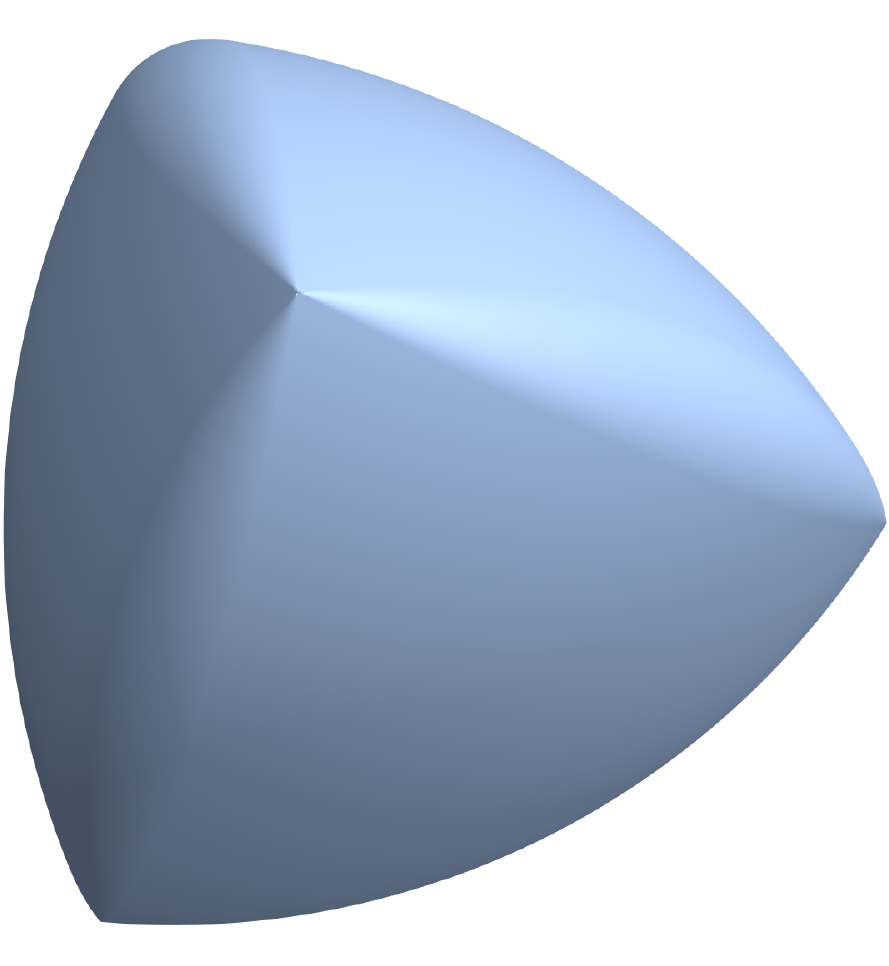}
 \hspace{.2in}
  \includegraphics[width=.4\textwidth]{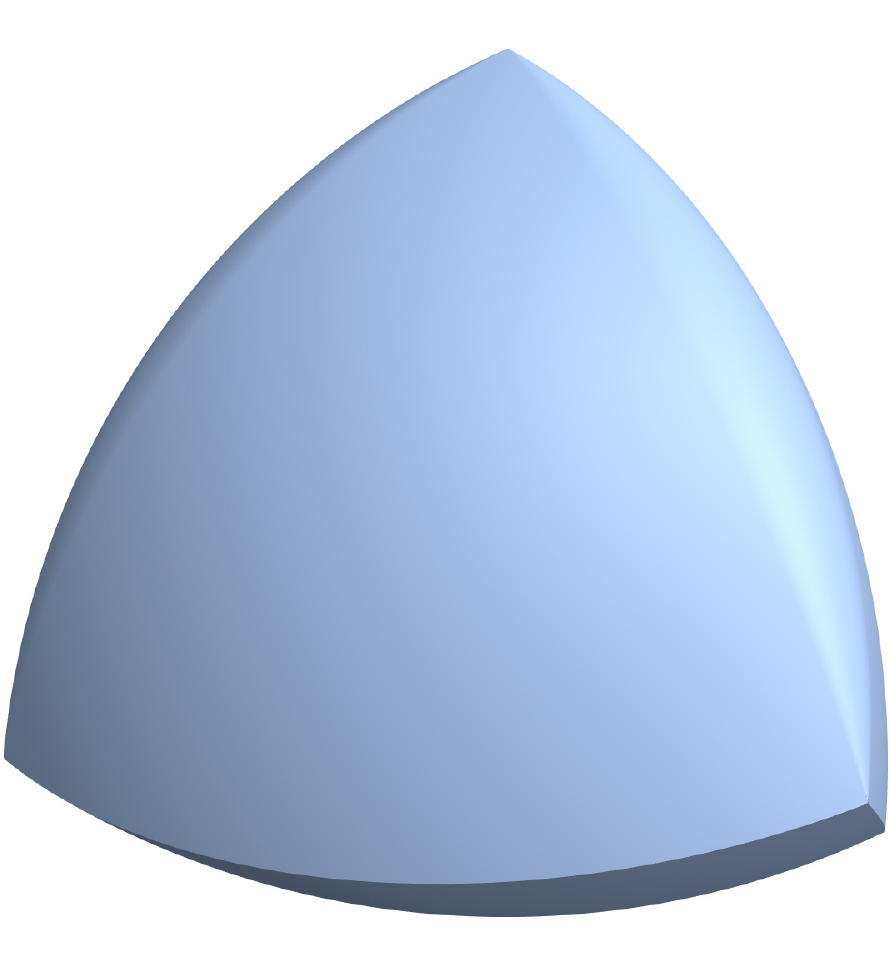}
 \caption{A Meissner tetrahedron in which smoothed edges share a common vertex.}\label{Meiss1}
\end{figure}
\par In what follows, we will discuss a family of constant width shapes in $\R^3$ which are designed analogously to the two Meissner's tetrahedra constructed above. They are known as Meissner polyhedra and are modeled on the class of shapes introduced by Montejano and Rold\'an-Pensado \cite{MR3620844} by the same name. However, we will define these shapes via intersections rather than by performing surgery on the boundary. For example if $e_{ij}$ represents the edge of $R$ joining $a_i$ and $a_j$,
$$
\bigcap\left\{B(x): x\in e_{12}\cup e_{13}\cup e_{14}\right\}
$$
is a Meissner tetrahedron in which the edges $e_{23},  e_{24},e_{34}$ that share the common face $\partial B(a_1)\cap R$ have been smoothed. Likewise 
$$
\bigcap\left\{B(x): x=a_1\;\text{or}\;x\in e_{23}\cup e_{24}\cup e_{34}\right\},
$$
is the Meissner tetrahedron in which the regions near the edges sharing vertex $a_1$ are smoothed.   
\begin{figure}[h]
\centering
 \includegraphics[width=.42\textwidth]{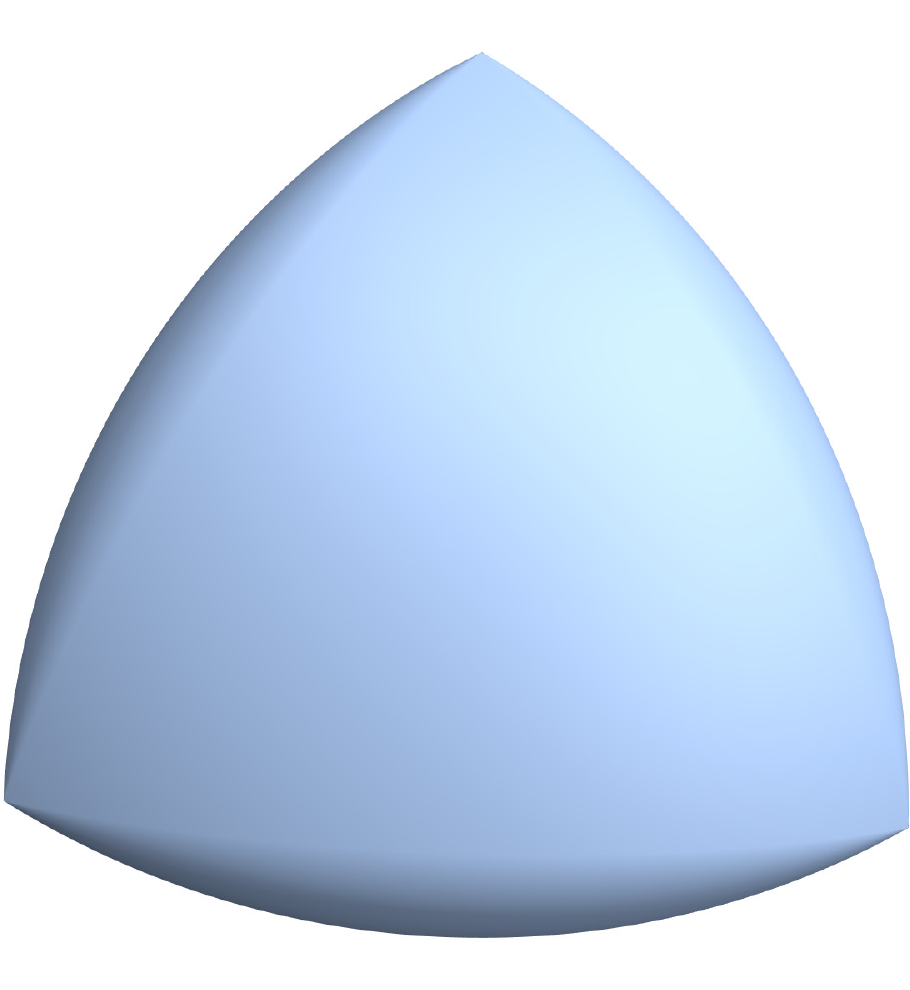}
 \hspace{.2in}
  \includegraphics[width=.4\textwidth]{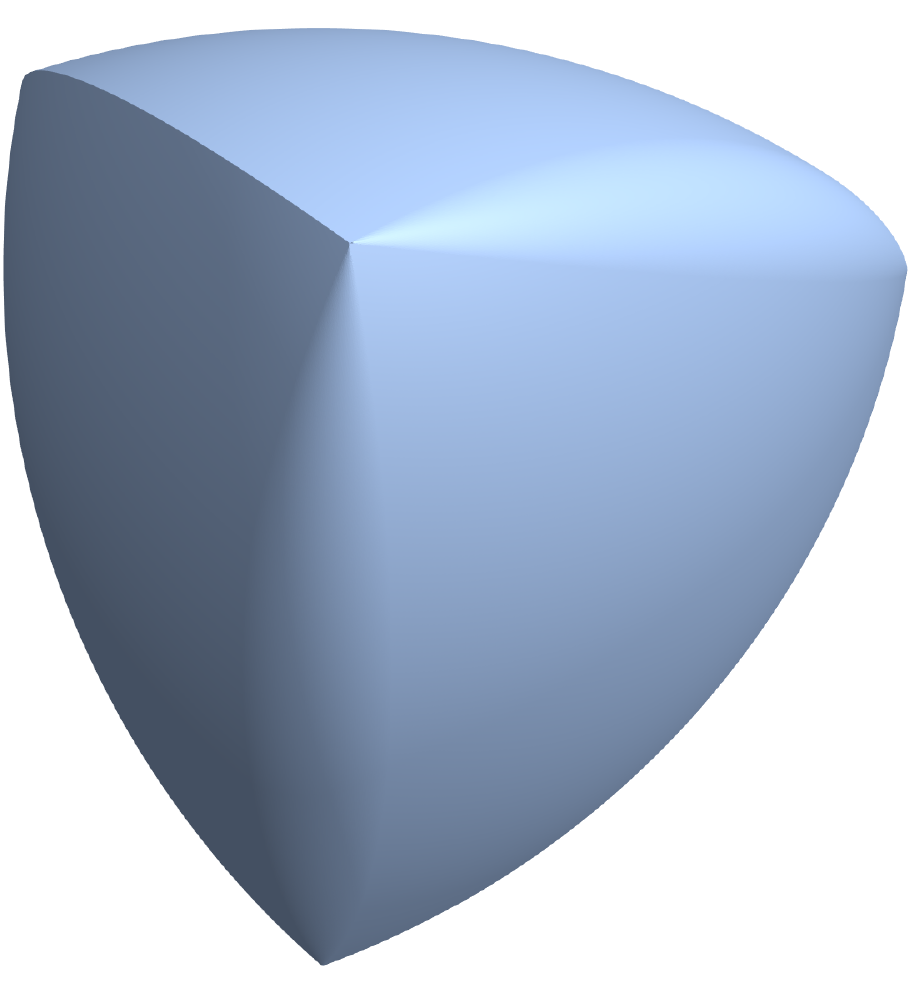}
 \caption{A Meissner tetrahedron in which smoothed edges share a common face.}\label{Meiss2}
\end{figure}

\par The reason we prefer to use intersections is that it allows to show that the family of Meissner polyhedra is dense in a sense specified below. The first result of this kind is due to Sallee \cite{MR296813}, who proved that a certain class of constant width shapes are also dense. It is unclear to us exactly how the family of shapes he considered compares to Meissner polyhedra. 
Nevertheless, our justification is largely based on Sallee's ideas along with recent developments in the understanding of spindle convex shapes.  

\par We recall that the Hausdorff distance between two convex bodies $K_1, K_2\subset \R^3$ is 
$$
d(K_1,K_2):=\inf\left\{r\ge 0: K_1\subset K_2+B_r(0), K_2\subset K_1+B_r(0) \right\}.
$$
Here $K+B_r(0):=\{x+z\in \R^3: x\in K, z\in B_r(0)\}$ is the Minkowski sum of $K$ and $B_r(0)$; $B_r(0)$ is the closed ball of radius $r$ centered at the origin. Moreover, we note that 
$d$ is a complete metric on the space of convex bodies.  The main result of this note is as follows. 
\\\\
\noindent {\bf Density Theorem}.  Assume $K\subset \R^3$ is a constant width body and $\epsilon>0$. There is a Meissner polyhedron $M\subset \R^3$ for which 
$$
d(K,M)\le\epsilon. 
$$

\par This paper is organized as follows.  In the next section, we consider spindle tori and the notion of spindle convexity. These ideas will assist us in describing some basic properties of ball polyhedra and constant width shapes. Then  in section \ref{RPolysect}, we study Reuleaux polyhedra, which is a family of ball polyhedra which includes the Reuleaux tetrahedra. These are the building blocks for Meissner tetrahedra, which are defined in section \ref{MeissSect}.  Finally, we verify the Density Theorem in section \ref{DensitySect}.  In addition, we compute the volume of the Meissner tetrahedra and show how to plot these figures using \texttt{Mathematica} in the appendix. The volume of these shapes was known, however, we include this computation as there does not appear to be a detailed calculation readily available in the literature.

\section{Spindles}\label{SpindleSect}
Throughout this section $n\ge 2$ will be a fixed natural number.   Let $x,y\in \R^n$ with $|x-y|\le 2.$
The {\it spindle} determined by $x$ and $y$ is the intersection of all closed balls of radius one which include both $x$ and $y$. We'll write
$$
\text{Sp}(x,y):=\bigcap_{x,y\in B(z)}B(z)
$$
for this set of points. Again we are using the notation $B(z)$ for the closed ball of radius one centered at $z\in \R^n$. It's clear that $\text{Sp}(x,y)$ is a convex body. We shall see that $\text{Sp}(x,x)=\{x\}$ and also that if $|x-y|=2$, then $\text{Sp}(x,y)=B((x+y)/2)$.
 
 \par For a given $X\subset\R^n$, it will also be convenient to use the notation 
 $$
 B(X):=\bigcap_{x\in X}B(x).
 $$
 For example, since $x,y\in B(z)$ if and only if $z\in B(x)\cap B(y)$, we may write 
 \be
  \text{Sp}(x,y)=B(B(x)\cap B(y) ).
 \ee
 In this section, we will derive some basic properties of spindles. We will also define spindle convex shapes and explain how they are related to constant width shapes.

\subsection{A defining inequality}
Recall that a spindle torus in $\R^3$ is a surface of revolution whose generating curve is a circle which intersects the axis of revolution. The ``inner" portions of such tori are the boundaries of the spindles considered in this note. See Figures \ref{SpindleFigure} and \ref{SpindleFigure2}.    
\begin{prop}
Suppose $0\le a\le 1$. Then
\be\label{aSpFormula}
\textup{Sp}(ae_n,-ae_n)=\left\{x\in \R^n: \left(\sqrt{x_1^2+\dots+x_{n-1}^2}+
\sqrt{1-a^2}\right)^2+x_n^2\le 1\right\}.
\ee
\end{prop}
\begin{rem}
We will write $x=(x_1,\dots, x_n)$ for the coordinates of a given $x\in \R^n$ and denote $e_1,\dots, e_n$ for the standard basis vectors in $\R^n$.  
\end{rem} 
\begin{figure}[h]
     \centering
     \begin{subfigure}[t]{0.13\textwidth}
         \centering
         \includegraphics[width=\textwidth]{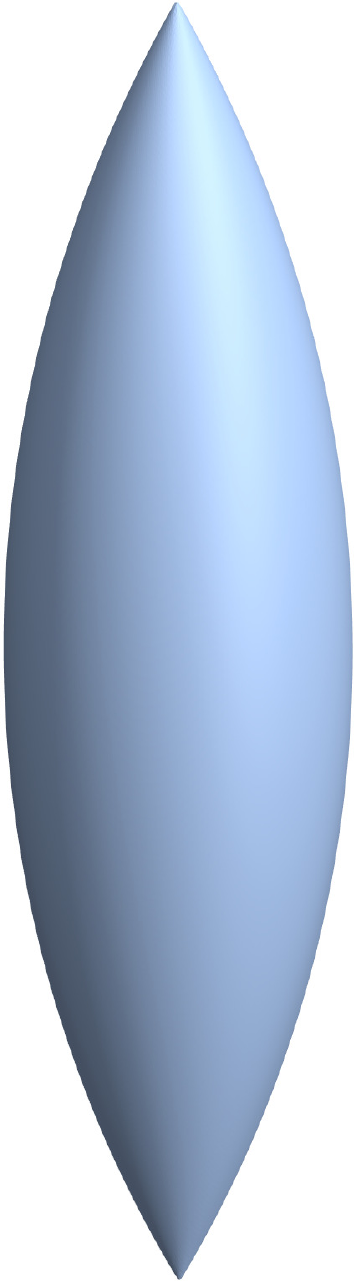}
         \caption{A spindle in $\R^3$.}\label{SpindleFigure}
     \end{subfigure}
    \hspace{.4in}
     \begin{subfigure}[t]{0.6\textwidth}
         \centering
         \includegraphics[width=\textwidth]{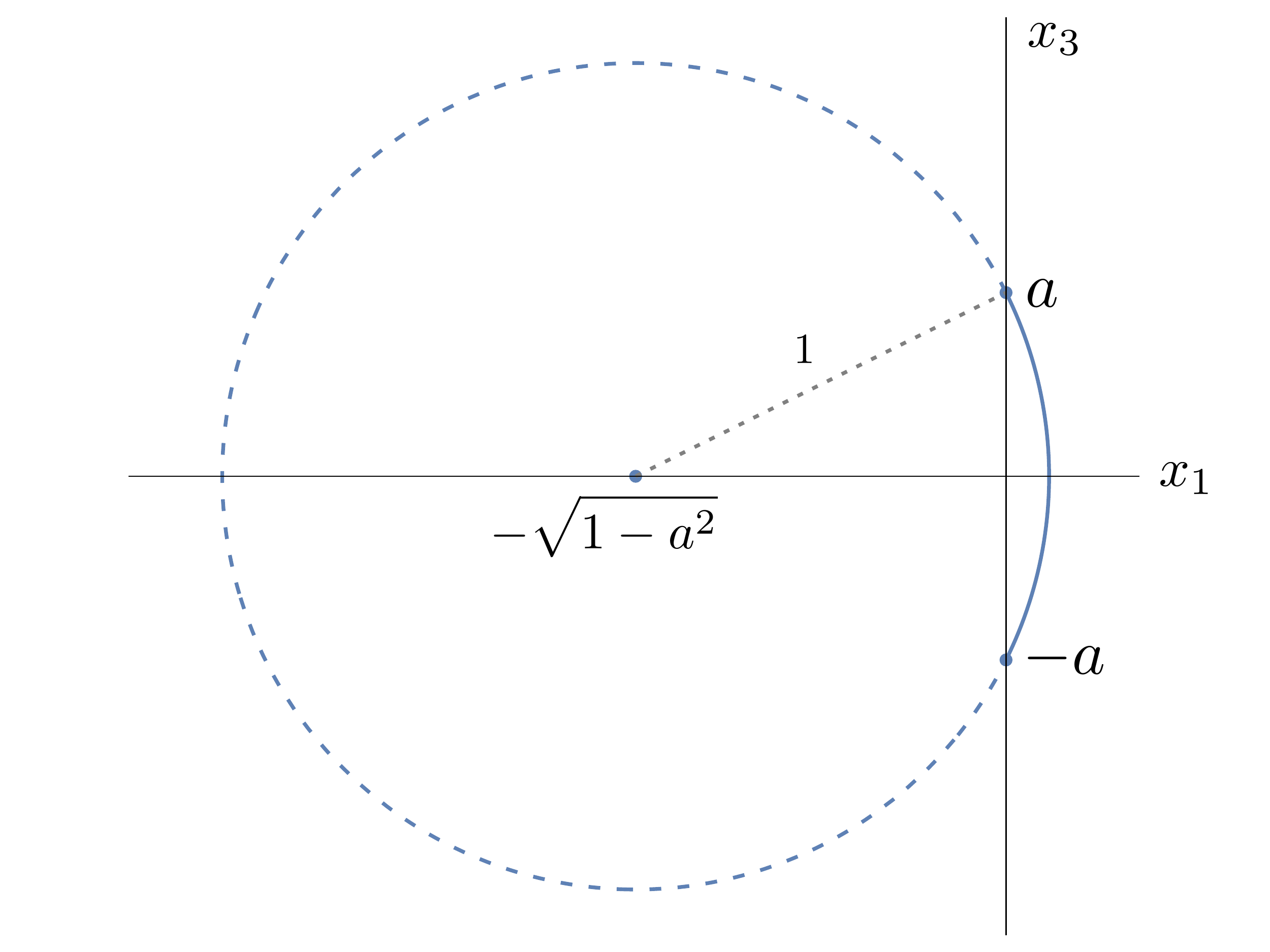}
               \caption{The curve in bold joining $ae_3$ to $-ae_3$ is the arc of a circle in the $x_1x_3$ plane which can be rotated about the $x_3$ axis to obtain $\partial \text{Sp}(ae_3,-ae_3)$.  }\label{SpindleFigure2}
     \end{subfigure}
\end{figure}
\begin{proof}[Proof of $\subset$ in \eqref{aSpFormula}]
Let $x\in \textup{Sp}(ae_n,-ae_n)$. First suppose $x_1^2+\dots+x_{n-1}^2>0$ and choose
$$
z_i=-\sqrt{1-a^2}\frac{x_i}{\sqrt{x_1^2+\dots+x_{n-1}^2}}
$$
for $i=1,\dots, n-1$ and $z_n=0$. As $|z\pm ae_n|=1$, 
\begin{align}
1&\ge |x-z|^2\\
&=\sum^{n-1}_{i=1}(x_i-z_i)^2+(x_n-z_n)^2\\
&=\sum^{n-1}_{i=1}\left(x_i+\sqrt{1-a^2}\frac{x_i}{\sqrt{x_1^2+\dots+x_{n-1}^2}}\right)^2+x_n^2\\
&=\sum^{n-1}_{i=1}\frac{x_i^2}{x_1^2+\dots+x_{n-1}^2}\left(\sqrt{x_1^2+\dots+x_{n-1}^2}+\sqrt{1-a^2}\right)^2+x_n^2\\
&=\left(\sqrt{x_1^2+\dots+x_{n-1}^2}+\sqrt{1-a^2}\right)^2+x_n^2.
\end{align}
Alternatively, if $x_1^2+\dots+x_{n-1}^2=0$, 
we choose any $z$ with $z_n=0$ and $|z|=\sqrt{1-a^2}$ and still find that $x$ belongs to the right hand side of \eqref{aSpFormula}.
\end{proof}
\begin{proof}[Proof of $\supset$ in \eqref{aSpFormula}]
\underline{Step 1}: We are to show that  
\be\label{KeyIneqXXZZ}
|x-z|\le1
\ee
for any $x\in \R^n$ belonging to the right hand side of \eqref{aSpFormula} and $z\in B(ae_n)\cap B(-ae_n)$. 
As the norm is a convex function and both the right hand side of  \eqref{aSpFormula} and $B(ae_n)\cap B(-ae_n)$ are convex, the largest $|x-z|$ can be occurs when $x$ belongs to the boundary of the right hand side of  \eqref{aSpFormula} and $z\in \partial (B(ae_n)\cap B(-ae_n))$.  As a result, it suffices to verify \eqref{KeyIneqXXZZ} when 
\be\label{xBoundaryKeyIneq}
\left(\sqrt{x_1^2+\dots+x_{n-1}^2}+\sqrt{1-a^2}\right)^2+x_n^2=1
\ee
and $z\in \partial (B(ae_n)\cap B(-ae_n))$. We will further reduce the complexity of deriving \eqref{KeyIneqXXZZ} with a series of observations. 

\par \underline{Step 2}: Observe that the right hand side of \eqref{aSpFormula} and $B(ae_n)\cap B(-ae_n)$ are both axially symmetric with respect to the $x_n$-axis and symmetric with respect to reflection about the $x_n=0$ hyperplane. It follows that  we only need to establish \eqref{KeyIneqXXZZ} for $x=x_1e_1+x_ne_n$ with 
\be\label{xSpinplaneConditions}
 \begin{cases}
\left(x_1+\sqrt{1-a^2}\right)^2+x_n^2=1\\
x_1\ge 0\\
x_n\ge 0.
\end{cases}
\ee
Indeed, if $x'$ belongs to the boundary of the right hand side of \eqref{aSpFormula} and $z'\in \partial (B(ae_n)\cap B(-ae_n))$, there is 
an orthogonal transformation $O:\R^n\rightarrow\R^n$ for which: $x=Ox'$ satisfies \eqref{xSpinplaneConditions},  $z=Oz'\in \partial (B(ae_n)\cap B(-ae_n))$, and $|x'-z'|=|x-z|$. Consequently, we will assume that $x$ satisfies these conditions for the remainder of this proof. 
 
 \par \par \underline{Step 3}: Next note that since $z\in \partial (B(ae_n)\cap B(-ae_n))$, then either $z\in \partial B(ae_n)\cap B(-ae_n)$ or $z\in B(ae_n)\cap \partial B(-ae_n)$. If $z\in B(ae_n)\cap\partial B(-ae_n)$, then $z_n\ge 0$. In this case, $w=z-2z_ne_n\in \partial B(ae_n)\cap B(-ae_n)$ and
 \begin{align}
|x-w|^2&=|x-z+2z_ne_n|^2\\
&=|x-z|^2+4z_n^2+2z_n(x_n-z_n)\\
&=|x-z|^2+2z_n^2+2z_nx_n\\
&\ge |x-z|^2.
\end{align}
As a result, we may consider \eqref{KeyIneqXXZZ} for $z\in \partial B(ae_n)\cap B(-ae_n)$ which implies that $z_n\le 0$.

\par Next choose  
$$
w=-\sqrt{z_1^2+\dots +z_{n-1}^2}e_1+z_ne_n. 
$$ 
It is easy to check that $w\in \partial B(ae_n)\cap B(-ae_n)$ and note 
\begin{align}
|x-w|^2&=\left(x_1+\sqrt{z_1^2+\dots +z_{n-1}^2}\right)^2+(x_n-z_n)^2\\
&=\left(\sqrt{x_1^2+\dots +x_{n-1}^2}+\sqrt{z_1^2+\dots +z_{n-1}^2}\right)^2+(x_n-z_n)^2\\
&\ge |x-z|^2.
\end{align}
The inequality follows from the triangle inequality applied to $(x_1,\dots, x_{n-1})$ and $(z_1,\dots, z_{n-1})$.
Therefore, we will derive \eqref{KeyIneqXXZZ} for $z\in \partial B(ae_n)\cap B(-ae_n)$ of the form $z=z_1e_n+z_ne_n$ with 
\be\label{zSpinplaneConditions}
\begin{cases}
z_1^2+(z_n-a)^2=1\\
z_1\le 0\\
z_n\le 0.
\end{cases}
\ee
\begin{figure}[h]
     \centering
         \includegraphics[width=.7\textwidth]{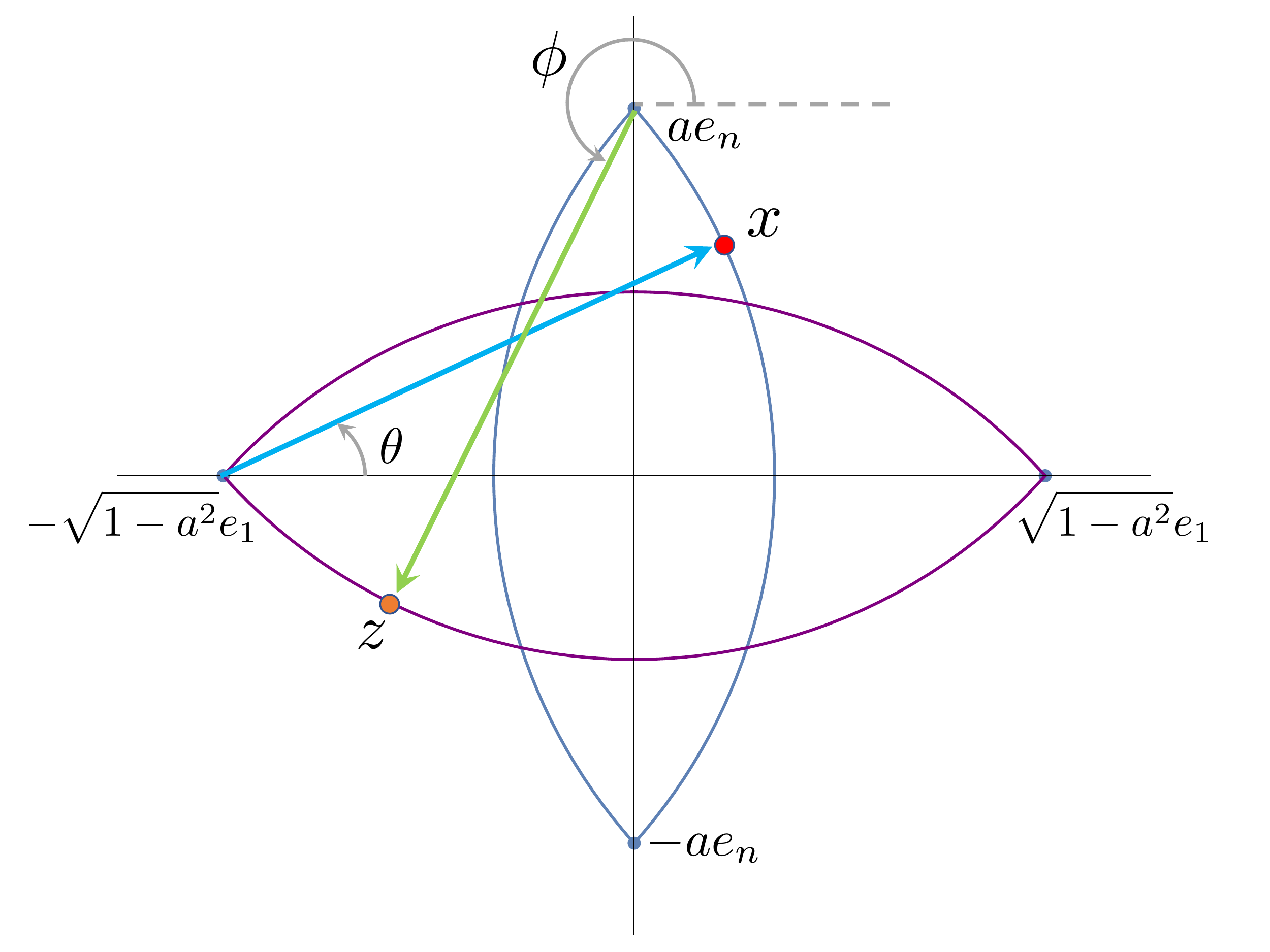}
 \caption{This figure illustrates the representations $x=u(\theta)-\sqrt{1-a^2}e_1$ and $z=u(\phi)+ae_n$ used in step 4 of our proof of Proposition \ref{aSpFormula}.  Here $\theta\in [0,\sin^{-1}(a)]$ and $\phi\in [\pi+\sin^{-1}(a),3\pi/2]$. These observations are crucial in showing $``\supset "$ holds in formula \eqref{aSpFormula}.}\label{SpinDistanceDiag}
\end{figure}
\par \underline{Step 4}:   It is routine to check that \eqref{xSpinplaneConditions} and \eqref{zSpinplaneConditions} are equivalent to 
$$
x=x(\theta):=u(\theta)-\sqrt{1-a^2}e_1\quad \text{and}\quad z=z(\phi):=u(\phi)+ae_n
$$
where $u(t)=\cos(t)e_1+\sin(t)e_n$, for some 
$$
\theta\in [0,\sin^{-1}(a)]\quad \text{and}\quad \phi\in [\pi+\sin^{-1}(a),3\pi/2]. 
$$
See Figure \ref{SpinDistanceDiag}.   If $\theta=\sin^{-1}(a)$, $x(\theta)=ae_n$. The inequality \eqref{KeyIneqXXZZ} holds in this case as $z(\phi)\in \partial B(ae_n)\cap B(-ae_n)$. Otherwise
\begin{align}
\frac{d}{d\theta}\frac{1}{2}|x(\theta)-z(\phi)|^2&=(x(\theta)-z(\phi))\cdot x'(\theta)\\
&=\left(u(\theta)-\sqrt{1-a^2}e_1-z(\phi)\right)\cdot u'(\theta)\\
&=\left(-\sqrt{1-a^2}e_1-z(\phi)\right)\cdot u'(\theta)\\
&=\left(-(\sqrt{1-a^2}e_1+ae_n)-u(\phi)\right)\cdot u'(\theta)\\
&=\left(u(\pi+\sin^{-1}(a))-u(\phi)\right)\cdot u'(\theta)\\
&=\sin(\pi+\sin^{-1}(a)-\theta)-\sin(\phi-\theta)\\
&\ge 0,
\end{align}
as 
$$
\pi\le\pi+\sin^{-1}(a)-\theta\le \phi-\theta\le  3\pi/2
$$
and $\sin$ is decreasing on $[\pi,3\pi/2]$. Therefore, $|x(\theta)-z(\phi)|^2\le |x(\sin^{-1}(a))-z(\phi)|^2=1$.
\end{proof}
\begin{rem}\label{IntersectionOverConvBody}
Our proof actually shows  $\text{Sp}(ae_n,-ae_n)=B(\partial B(ae_n)\cap \partial B(-ae_n))$.
\end{rem}

\par It also turns out that each spindle $\textup{Sp}(x,y)$ is simply related to $\textup{Sp}(ae_n,-ae_n)$ for an appropriate choice of $a$.  
\begin{prop}\label{SpInvarianceProp}
Suppose $x,y\in \R^n$ with $0< |x-y|\le 2$. Further assume $O: \R^n\rightarrow \R^n$ is an orthogonal transformation with
\be
Oe_n=\frac{x-y}{|x-y|}.
\ee
Then 
\be\label{SpxySpaminusu}
\textup{Sp}(x,y)=\frac{x+y}{2}+O\textup{Sp}\left(\frac{|x-y|}{2}e_n,-\frac{|x-y|}{2}e_n\right).
\ee
Moreover, if $|x-y|=2$, then $\textup{Sp}(x,y)=B((x+y)/2)$. 
\end{prop}
\begin{proof}
First we claim that 
\be\label{SpInvarianceProp}
\textup{Sp}(Ox+c,Oy+c)=c+O\textup{Sp}(x,y)
\ee
for any orthogonal transformation $O$ of $\R^n$ and fixed $c\in \R^n$. Let $w\in \textup{Sp}(Ox+c,Oy+c)$ and $z'\in B(x)\cap B(y)$. Set $z=Oz'+c$, and note $|z-(Ox+c)|\le 1$ and $|z-(Oy+c)|\le 1$. It follows that  
$$
1\ge |w-z|= |O^t(w-c)-z'|.
$$
As a result, $O^t(w-c)\in  \textup{Sp}(x,y)$. That is, 
$w\in c+O\textup{Sp}(x,y)$ and 
$$
\textup{Sp}(Ox+c,Oy+c)\subset c+O\textup{Sp}(x,y).
$$
The reverse inclusion holds similarly. 

\par Now assume $O$ is an orthogonal transformation as in the statement of this proposition. Observe that the mapping $$w\mapsto Ow+\frac{x+y}{2}$$ sends $\frac{|x-y|}{2}e_n$ to $x$ and $-\frac{|x-y|}{2}e_n$ to $y$. 
According to \eqref{SpInvarianceProp}, 
\begin{align}
\textup{Sp}(x,y)&=\textup{Sp}\left(O\left(\frac{|x-y|}{2}e_n\right)+\frac{x+y}{2},O\left(-\frac{|x-y|}{2}e_n\right)+\frac{x+y}{2}\right)\\
&=\frac{x+y}{2}+O\textup{Sp}\left(\frac{|x-y|}{2}e_n,-\frac{|x-y|}{2}e_n\right).
\end{align}
Finally, note that \eqref{aSpFormula} gives that $\textup{Sp}(e_n,-e_n)=B(0);$ so if $|x-y|=2$, then 
$$
\textup{Sp}(x,y)=\frac{x+y}{2}+O\textup{Sp}(e_n,-e_n)=\frac{x+y}{2}+B(0)=B((x+y)/2).
$$ 
\end{proof}
\begin{rem}
Formula \eqref{aSpFormula} also implies that $\textup{Sp}(0,0)=\{0\}$. Denoting $I_n$ as the identity mapping of $\R^n$, 
we also have 
$$
\textup{Sp}(x,x)=\textup{Sp}(I_n0+x,I_n0+x)=x+I_n\textup{Sp}(0,0)=\{x\}
$$
by \eqref{SpInvarianceProp}.  
\end{rem} 
\begin{rem} Formula \eqref{SpInvarianceProp} gives us another way to see that $\text{Sp}(ae_n,-ae_n)$ is cylindrically symmetric. Indeed if
$O^tO=I_n$ with $Oe_n=e_n$, then 
$$
O\text{Sp}(ae_n,-ae_n)=\text{Sp}(ae_n,-ae_n).
$$
\end{rem} 
The above proposition implies that $\text{Sp}(x,y)$ is cylindrically symmetric about the line passing through $x$ and $y$. In addition, we 
may write the general form of  \eqref{aSpFormula}.
\begin{cor}\label{GenSpFormulaCor}
Suppose $0<|x-y|\le 2$. Then $w\in \textup{Sp}(x,y)$ if and only if
\begin{align}
&\left|w-\frac{x+y}{2}-\left(\left(w-\frac{x+y}{2}\right)\cdot \frac{x-y}{|x-y|}\right)\frac{x-y}{|x-y|}\right| +\sqrt{1-\left|\frac{x-y}{2}\right|^2}\\
&\hspace{1.5in}\le \sqrt{1-\left(\left(w-\frac{x+y}{2}\right)\cdot \frac{x-y}{|x-y|}\right)^2}.
\end{align}
\end{cor}
\begin{proof}
Select an orthogonal transformation $O$ as in the statement of Proposition \ref{SpInvarianceProp}. Then $w\in \text{Sp}(x,y)$ if and only if 
there is $v\in   \text{Sp}\left(\frac{1}{2}|x-y|e_n,-\frac{1}{2}|x-y|e_n\right)$ with $w=(x+y)/2+Ov$.  The asserted inequality is equivalent to $v$ belonging to the right hand side of  \eqref{aSpFormula} for $a=\frac{1}{2}|x-y|$. 
\end{proof}
\begin{figure}[h]
     \centering
     \begin{subfigure}[t]{0.4\textwidth}
         \centering
         \includegraphics[width=\textwidth]{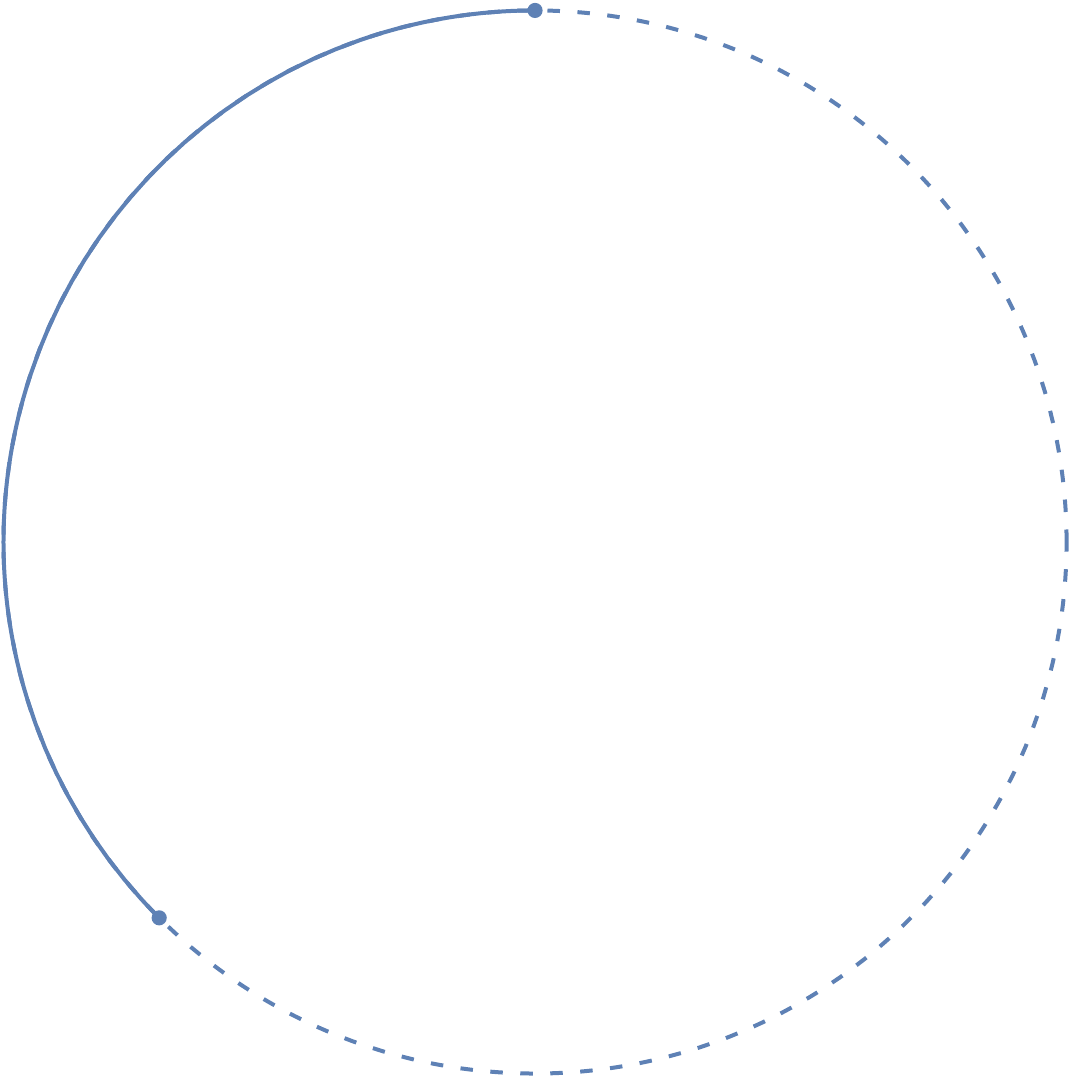}
               \caption{This figure shows an example of a short arc of a circle with the solid circular arc. The circle which includes this short arc is dashed.}
     \end{subfigure}
          \hspace{.1in}
          \begin{subfigure}[t]{0.5\textwidth}
         \centering
         \includegraphics[width=\textwidth]{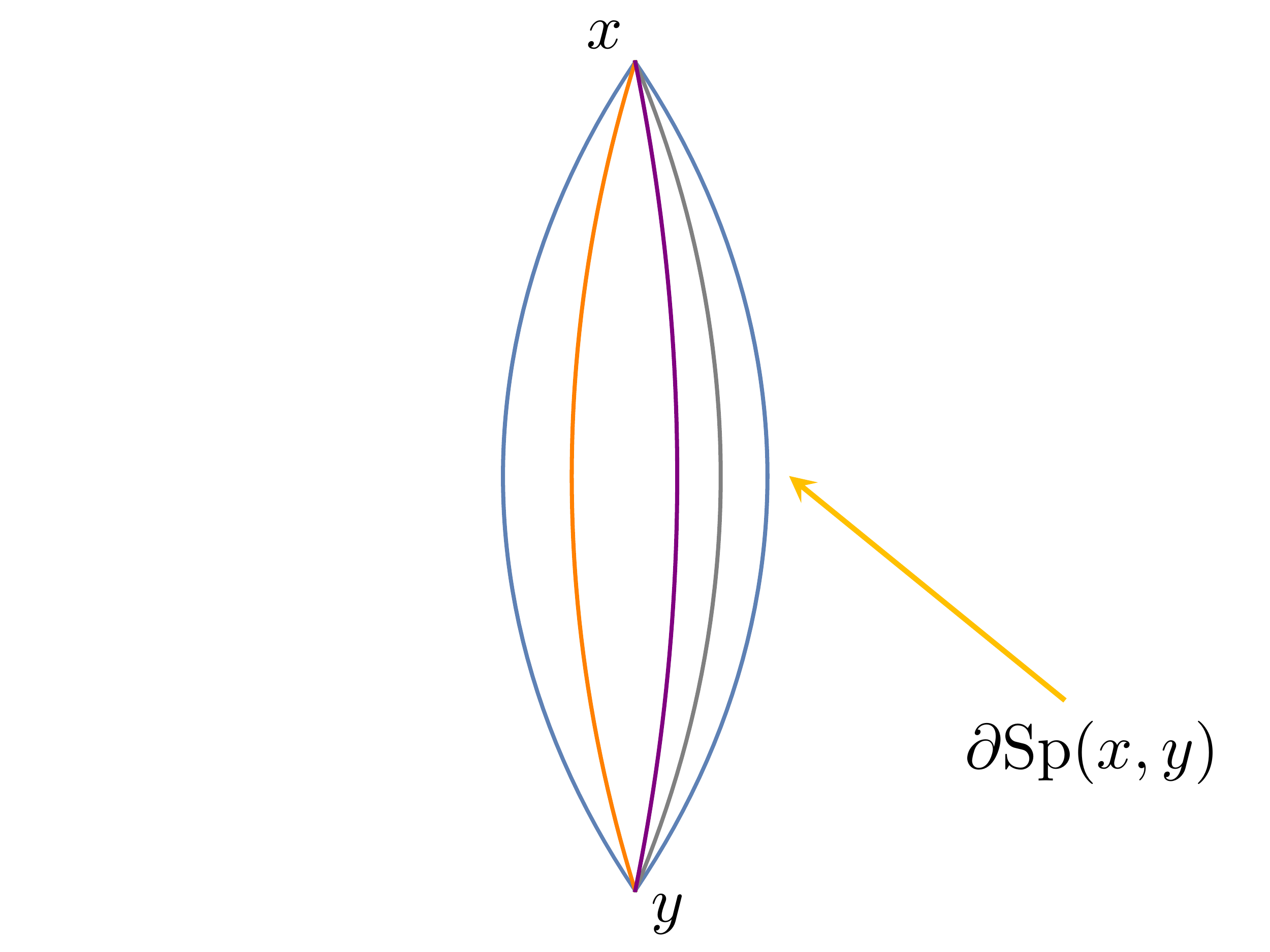}
 \caption{ This diagram displays a profile of a spindle $\text{Sp}(x,y)$ with three short arcs of radius at least one joining $x$ and $y$. Proposition \ref{ShortArcProp} asserts that this spindle is the union of all such short arcs. }
     \end{subfigure}
\end{figure}

\subsection{Short arcs}
 A {\it circle} $C\subset \R^n$ is a circle in a two-dimensional subspace of $\R^n$. Suppose the radius of $C$ is $r$ and $x,y\in C$. A {\it short arc} of $C$ joining $x$ and $y$ is a smaller of the two circular arcs within $C$ that joins these points.  Of course if $|x-y|<2r$ there will be a unique short arc within $C$ that joins these points; otherwise there will be two.  We also consider the line segment between $x$ and $y$ as the short arc of a circle with radius $r=\infty$. 
\begin{prop}\label{ShortArcProp}
Suppose $|x-y|\le 2$. Then $\textup{Sp}(x,y)$ is the union of all short arcs of circles with radius at least one which joins $x$ and $y$. 
\end{prop}
\begin{proof}
Without any loss of generality, we may assume $x=ae_n$ and $y=-ae_n$ for $0\le a\le 1$.   Suppose 
$w\in \text{Sp}(ae_n,-ae_n)$. If $w_1=\dots=w_{n-1}=0$, then $w=w_ne_n$ with $|w_n|\le a$; so $w$ is on the line segment between $ae_n$ to $-ae_n$. Alternatively, suppose $w\in \text{Sp}(ae_n,-ae_n)$ with $w$ not on the line segment between $ae_n$ to $-ae_n$.  There is an orthogonal mapping $O:\R^n\rightarrow \R^n$ which fixes the $e_n$ direction and such that $z=Ow$ satisfies  
\be\label{SpecialZeeCond}
z_2=\dots=z_{n-1}=0\; \text{ and }\; z_1> 0.
\ee
It is enough to show that this point $z$ is on a short arc of a circle with radius at least one and that joins $ae_n$ to $-ae_n$. Indeed, we can apply $O^{-1}$ to this arc to obtain the desired short arc for $w$. 

\par We now suppose $z\in \text{Sp}(ae_n,-ae_n)$ satisfies \eqref{SpecialZeeCond}. Then we have
$$
z_1^2+1-a^2+2z_1\sqrt{1-a^2}+z_n^2= \left(z_1+\sqrt{1-a^2}\right)^2+z_n^2\le 1.
$$
This inequality gives 
$$
\sqrt{1-a^2}\le \frac{a^2-z_1^2-z_n^2}{2z_1}=\sqrt{r^2-a^2}
$$
for some $r\ge 1$.  That is, 
$$
\left(z_1+\sqrt{r^2-a^2}\right)^2+z_n^2=r^2.
$$
Then $z$ belongs to the short arc of a circle of radius $r\ge 1$ which joins $ae_n$ to $-ae_n$.   

\par Conversely, assume that $z$ belongs to a short arc of a circle $C$ of radius $r\ge 1$ which joins $ae_n$ to $-ae_n$.  Without any loss of generality, we may suppose $C$ is a subset of the $x_1x_n$ plane and that $z_1\ge 0$.  Observe
$$
z_1^2+r^2-a^2+2z_1\sqrt{r^2-a^2}+z_n^2=\left(z_1+\sqrt{r^2-a^2}\right)^2+z_n^2=r^2,
$$
which can be expressed as
$$
z_1^2-a^2+2z_1\sqrt{r^2-a^2}+z_n^2=0
$$
Since $z_1\ge 0$ and $r\ge 1$, we have 
$$
z_1^2-a^2+2z_1\sqrt{1-a^2}+z_n^2\le 0.
$$
That is, 
$$
\left(z_1+\sqrt{1-a^2}\right)^2+z_n^2\le 1.
$$
We conclude that $z\in \text{Sp}(ae_n,-ae_n)$.
\end{proof}

\subsection{Spindle convexity}
We'll say that a subset $K\subset \R^n$ with diameter less than or equal to 2 is {\it spindle convex} if $\text{Sp}(x,y)\subset K$ whenever $x,y\in K$. 
Equivalently, $K$ is spindle convex if and only if for each $x,y\in K$ and short arc $\gamma$ of a circle of radius at least one joining $x$ and $y$, then $\gamma\subset K$.  Also note that $K$ is strictly convex since the interior of the line segment between $x$ and $y$ lies in the interior of $\text{Sp}(x,y)$. 
\par  Every closed ball $B$ of radius one is spindle convex. Indeed if $x,y\in B$, then by definition
$$
\text{Sp}(x,y)=\bigcap_{x,y\in B(z)}B(z)\subset B.
$$
It is also easy to check that the intersection of any collection of spindle convex sets is again spindle convex. Therefore, the intersection of any collection of closed balls of radius one is spindle convex. This observation in turn implies each spindle $\text{Sp}(x,y)$ itself is spindle convex as it is the intersection of closed balls of radius one.

\par We will verify the converse to our observation that the intersection of closed balls of radius one is spindle convex. That is, we will show that any spindle convex $K$ is the intersection of closed balls of radius one. To this end, we'll say that $\partial B$ is a {\it supporting sphere} through $x$ if $B$ is a ball of radius one, $x\in \partial B\cap \partial K$, and $K\subset B$. The following proposition is proved in Lemma 3.1 and Corollary 3.4 of \cite{MR2343304}, Theorem 3.1 of \cite{MR2593321}, and Theorem 6.1.5 of \cite{MR3930585}. Nevertheless, we will include a proof for completeness. 

\begin{prop}
Assume $K\subset\R^n$ is convex body with diameter at most two. The following statements are equivalent. \\
(i) $K$ is spindle convex.\\
(ii) For each $x\in \partial K$ and supporting plane $L$ for $K$ at $x$, there is a supporting sphere through $x$ which is tangent to $L$ and lies on the same side of $L$ as $K$ does.\\
(iii) $K$ is the intersection of closed balls of radius one.
\end{prop}
\begin{proof}
$(i)\Longrightarrow (ii)$ Suppose $x\in \partial K$, $H$ is a half-space such that $\partial H$ is a supporting plane for $\partial K$ at $x$, and $K\subset H$.  Let $B$ be the ball of radius one for which $\partial B$ is tangent to $\partial H$ at $x$ and $B\subset H$.  We claim that $\partial B$ is a supporting sphere for $K$ at $x$. If not, there is $y\in K$ with $y\not\in B$. Consider the two dimensional plane $\Pi$ determined by the line through the center of $B$ and $x$ and the line through the center of $B$ and $y$.  Since $y\not\in B$ and the diameter of $K$ is less than or equal to two, there is a short arc $\gamma\subset \Pi$ of radius one which joins $x$ and $y$ and is not included $H$; see Figure \ref{SpindConvFigFig}. In particular, there is $z\in \gamma$ which does not belong to $K$. However, this would contradict our assumption that $K$ is spindle convex. 
\begin{figure}[h]
\centering
 \includegraphics[width=.6\textwidth]{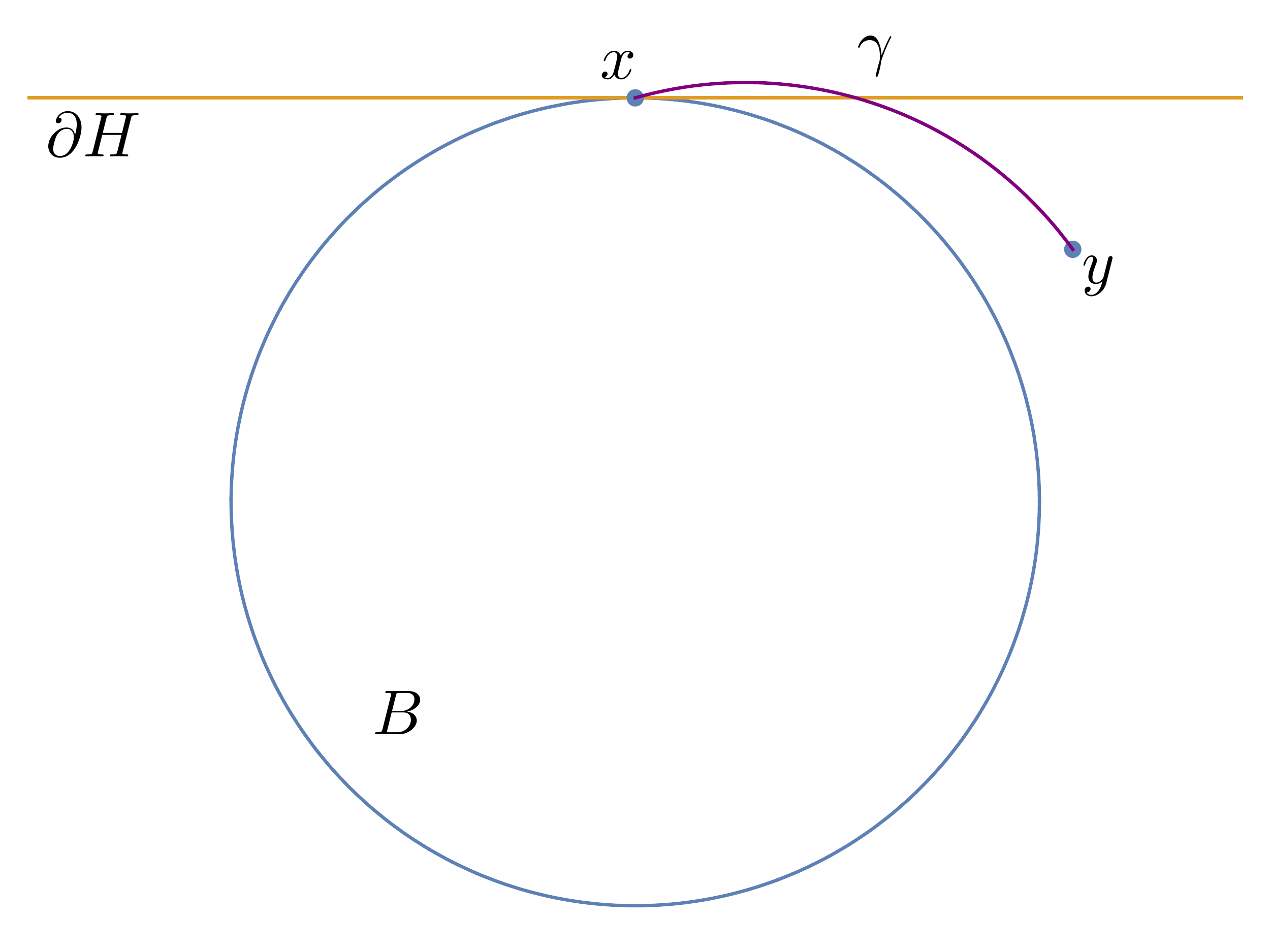}
 \caption{ Here $B$ is a ball of radius one which is included in a half-space $H$ for which $x\in \partial H\cap \partial B$. If $y\not\in B$, $y\in H$, and $|x-y|\le 2$, there is a short arc $\gamma$ of a circle of radius one which joins $x$ and $y$ and is not entirely contained in $H$.}\label{SpindConvFigFig}
\end{figure}
\par $(ii)\Longrightarrow (iii)$ Clearly 
$$
K\subset \bigcap\left\{B: \partial B \text{ is a supporting sphere for $K$ at some $x\in \partial K$}\right\}.
$$
Now suppose $y\not \in K$. Since $K$ is a convex body, there is supporting plane $L$ of $K$ at some $x_0\in \partial K$ which separates $K$ and $y$. By hypothesis, there is a supporting sphere $\partial B_{x_0}$ for $K$ which lies on the same side of $L$ as $K$ does. Thus, $y\not\in B_{x_0}$. In particular, $y$ does not belong to the intersection of balls whose boundaries are supporting spheres for $K$. We conclude 
$$
K^c\subset \left( \bigcap\left\{B: \partial B \text{ is a supporting sphere for $K$ at some $x\in \partial K$}\right\}\right)^c.
$$
\par \par $(iii)\Longrightarrow (i)$ As already noted, the intersection of a collection balls of radius one is necessarily spindle convex. 
\end{proof}

\par Another basic fact about spindle convex shapes, which was discussed in section 5 of \cite{MR2343304} and section 4 of \cite{MR2593321}, is as follows. 
\begin{lem}\label{SphereConvSubsetLem}
Suppose $B$ is a closed ball of radius one and $K\subset \R^n$ is spindle convex. Then $K\cap \partial B$ is geodesically convex in $\partial B$. And if $K\neq B$, $K\cap \partial B$ is a subset of a hemisphere of $\partial B$. 
\end{lem}
\begin{proof}
Suppose that $x,y\in K\cap \partial B$.  There is a length minimizing geodesic $\gamma\subset \partial B$ which joins $x$ and $y$.  As $\gamma$ is a short arc of a circle of radius one, $\gamma\subset K$.  Therefore, $\gamma\subset K\cap \partial B$. 

\par Let us assume that $B$ is the unit ball and that $K\neq B$. Since $K$ is spindle convex, it is a subset of another ball of radius one centered at  $a\in \R^n$ different from the origin. If $x\in K\cap \partial B$, then $|x-a|\le 1$ and $|x|=1$. Therefore, 
$$
|x|^2-2a\cdot x+|a|^2\le 1=|x|^2. 
$$ 
That is, $|x|=1$ and $x\cdot a>0$. As a result, $x$ belongs to a hemisphere of $\partial B$.  
\end{proof}
\par We can also identify which spindle convex shapes have constant width. The following theorem is usually credited to Eggleston \cite{MR200695}, who in turn gives credit to Jessen \cite{MR3108700}.  In the statement below, we'll use the notion of the normal cone at $z\in K$, which is defined 
$$
N_K(z)=\left\{w\in \R^n: w\cdot (v-z)\le 0 \text{ for all $v\in K$}\right\}.
$$ 
\begin{thm}\label{ConstantWidthCharacterization}
Suppose $K\subset \R^n$ is a convex body.  Then $K$ has constant width if and only if 
$$
K=B(\partial K).
$$
\end{thm}
 \begin{proof}
 First suppose that $K$ has constant width. It is not hard to see that $K$ has diameter one. As $|x-y|\le1$ for any $x\in K$ and $y\in \partial K$, it follows that $K\subset B(\partial K)$.  Next assume $y\not\in  K$. Then there is a supporting plane $L$ of $K$ which separates $K$ and $y$. Let $u$ be the outward unit normal to $L$ and choose $x\in \partial K$ for which $u\in N_K(x)$. As $K$ has constant width, there is also $w\in \partial K$ with $|x-w|=1$ and $-u\in N_K(w)$. Notice that $B(w)$ does not contain $y$ as $B(w)$ is on the same side of $L$ as $K$ is. See Figure \ref{BKequalKpic}. As a result,  $y\not\in B(\partial K)$. We conclude $K^c\subset B(\partial K)^c$. 
  \begin{figure}[h]
\centering
 \includegraphics[width=.7\textwidth]{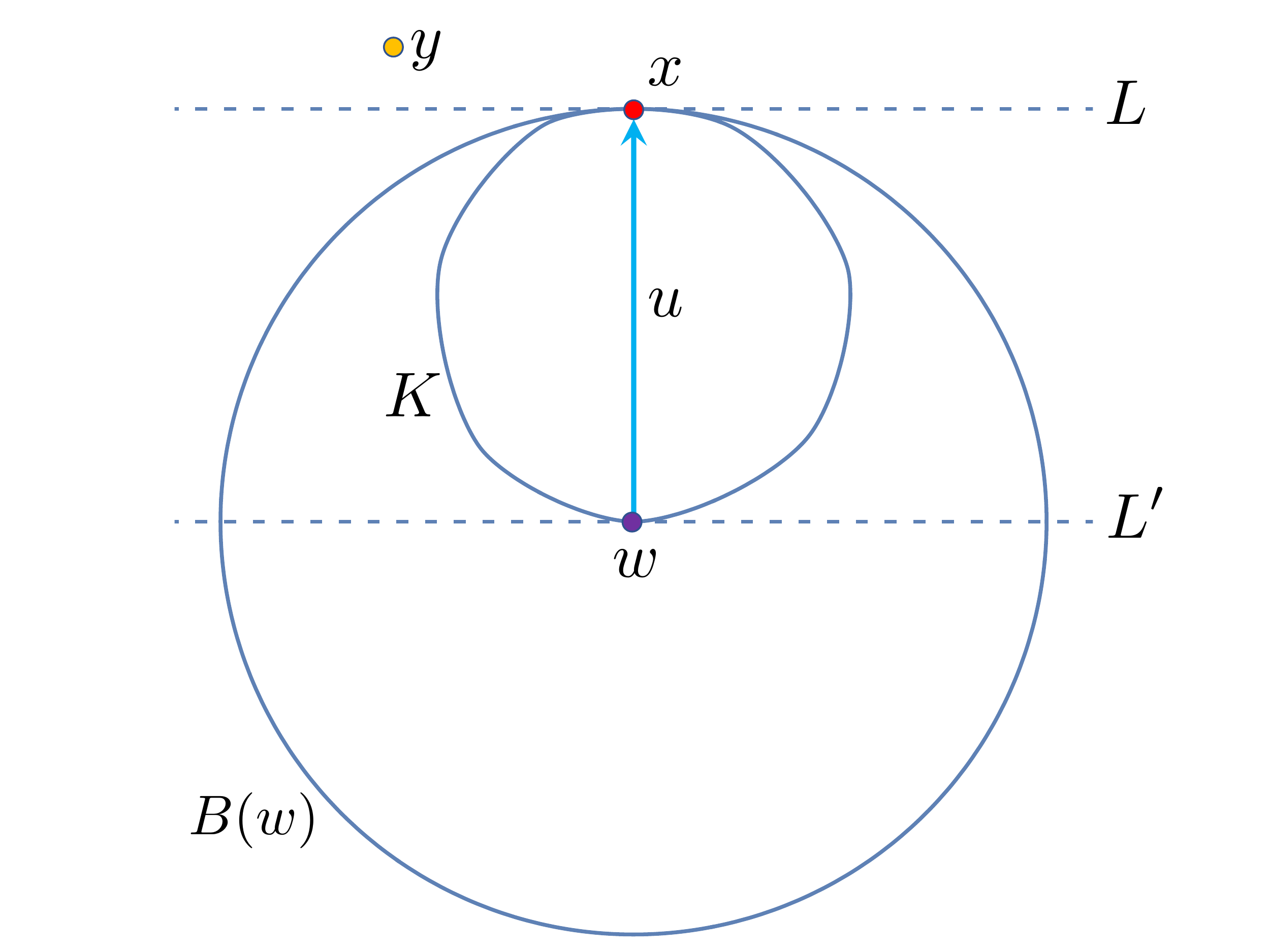}
 \caption{ Here we have a convex body $K$ with $x\in \partial K$. Note that $L$ a supporting plane for $K$ at $x$ with outward unit normal $u$ and $\partial B(w)$ is a supporting sphere for $K$ at $x$ which lies on the same side of $L$ as $K$ does.   Moreover, $L$ separates $y$ and $K$ and $L'$ is a plane parallel to $L$.}\label{BKequalKpic}
\end{figure}

\par Conversely suppose $K=B(\partial K)$. It is not hard to check that $K=B(\partial K)$ implies that $K$ has diameter at most one.  Assume $L,L'$ are parallel supporting planes for $K$ and $x\in L\cap \partial K$.  Note that the distance between $L$ and $L'$ is not more than one as $\text{diam}(K)\le 1$.  Since $K$ is spindle convex, there is a supporting sphere $\partial B(w)$ for $K$ at $x$ such that $B(w)$ lies on the same side of $L$ as $K$ does. Note in particular, that since $\partial K\subset K\subset B(w)$, $w\in B(\partial K)$.  Also notice that if the distance between $L$ and $L'$ is less than one, $w$ could not belong to $K$. We again refer to Figure \ref{BKequalKpic}.  In view of our hypothesis that $B(\partial K)=K$, the distance between $L$ and $L'$ is necessarily equal to one. We conclude that $K$ has constant width. 
\end{proof}
\begin{rem}
If we denote $\textup{ext}(K)$ as the extreme points of a convex body $K\subset \R^n$, then $B(K)=B(\textup{ext}(K)).$ As a result,  $$B(\partial K)= B(K)$$ for any strictly convex $K$. In particular, $K=B(K)$ for each constant width $K$. 
\end{rem}
\begin{rem}
We used a few things in the proof above which bear repeating for a given convex body $K\subset \R^n$. The inclusion $K\subset B(\partial K)$ is equivalent to the diameter of $K$ being at most one. And $K\supset B(\partial K)$ if and only if $K\subset B(y)$ implies $y\in K$. 
\end{rem}
\begin{rem}\label{DenseIsEnough}
In the sequel, we will also use a basic fact that if $X\subset \R^n$ and $X'\subset X$ is dense, then 
 $$
 B(X)=B(X').
 $$
To see this, let $y\in B(X')$, $x\in X$, and choose $x_n\in X'$ converging to $x$ as $n\rightarrow\infty$. Then $|y-x|=\lim_{n\rightarrow\infty}|y-x_n|\le 1$, and $y\in B(X)$.  It follows that  $B(X')\subset B(X)$. It is also clear that $ B(X)\subset B(X')$ as $X'\subset X$.  
 \end{rem}

\section{Reuleaux polyhedra}\label{RPolysect}
Suppose $a_1,a_2,a_3,a_4\in \R^3$ satisfy $|a_i-a_j|=1$ for $i\neq j$. As we noted in the introduction, the corresponding Reuleaux tetrahedron  $R=B(\{a_1,a_2,a_3,a_4\})$ has has four vertices, six circular edges, and four spherical faces. Moreover, the vertices of $R$ are exactly the centers of the spheres which define $R$. In this section, we will define a general class of such objects with these properties which we will call Reuleaux polyhedra.  In particular, we will mostly introduce topics and survey the results from the seminal paper of Kupitz, Martini, and Perles \cite{MR2593321}.  These ideas have also been covered in detail in chapter 6 of the monograph by Martini, Montejano, and Oliveros \cite{MR3930585}.

\subsection{Ball polyhedra}
We will say that $B(X)\subset \R^3$ is a {\it ball polyhedron} when $X\subset \R^3$ is nonempty and finite with 
$$
\text{diam}(X)\le 1.
$$
An elementary observation is as follows. 

\begin{lem}
Suppose $B(X)$ is a ball polyhedron. Then $X\subset B(X)$, $B(X)\subset\R^3$ is spindle convex, and $B(X)$ has nonempty interior. 
\end{lem}
\begin{proof}
As $\text{diam}(X)\le 1$, it follows immediately that $X\subset B(X)$. Moreover, it is easy to check that $\text{diam}(B(X))\le 2$. And since $B(X)$ is the intersection of balls of radius one, it is spindle convex.   
\par Jung's theorem implies that $X\subset B_\delta(y)$ for $\delta=\sqrt{3/8}$ and some $y\in \R^3$. If $w\in B_{1-\delta}(y)$ and $x\in X$, then 
$$
|w-x|\le |w-y|+|y-x|\le 1-\delta+\delta=1. 
$$
It follows that $w\in B(X)$. Therefore, $B_{1-\delta}(y)\subset B(X)$.   
\end{proof}

\par It will useful for us to identify when there are no redundancies in the definition of a ball polyhedron.  To this end, we will say that $x\in X$ is {\it essential} provided that 
$$
B(X)\subsetneq B(X\setminus\{x\}).
$$
This means that if we remove $x$ from $X$, then $B(X)$ is no longer equal to  $B(X\setminus\{x\})$. This 
is the case precisely when there exists $y\in \R^3$ for which $X\setminus\{x\}\subset B(y)$ while $x\not\in B(y)$.  We will additionally  say that $X$ is {\it tight} provided that each $x\in X$ is essential.  

\par The following observations regarding essential points were made in section 5 of \cite{MR2593321}. 
\begin{lem}\label{essLem}
Suppose $X\subset \R^3$ is a finite set of points having diameter one. \\
$(i)$ If $x\in X$ is essential, there is $y\in\R^3$ with $|y-x|=1$ and $|y-z|<1$ for each $z\in X\setminus\{x\}$.\\  $(ii)$ If $Y\subset X$ is the collection of essential points of $X$, then $B(Y)=B(X)$. \\  
$(iii)$ If $x\in X$ and there are distinct $y,z\in X$ with $|x-z|=|x-y|=1$, then $x$ is essential. 
\end{lem}

\par Let us assume now that $B(X)$ is a ball polyhedron and $X$ is tight.  We wish to describe the boundary of $B(X)$. As we saw with the Reuleaux tetrahedron above, we will see that the boundary of $B(X)$ consists of finitely many vertices, circular edges and spherical faces.  To this end, we will denote 
$$
\text{val}(x,X)=\# \{y\in X: |x-y|=1\}
$$
as the {\it valence} of a given vertex $x\in B(X)$ which is defined below. 
\\\\
\noindent {\bf Faces}.   Observe that $y$ belongs to the interior of $B(X)$ if and only if $|y-x|< 1$ for all $x\in X$. Therefore, $y$ belongs to the boundary of $B(X)$ if and only if $y\in B(X)$ and $|y-x|=1$ for some $x\in X$.  This implies 
$$
\partial B(X)=\bigcup_{x\in X}\left(B(X)\cap \partial B(x)\right)
$$
We define $B(X)\cap \partial B(x)$ as the {\it face of $B(X)$ opposite $x$}.  By Lemma \ref{SphereConvSubsetLem}, $B(X)\cap \partial B(x)$ is a spherically convex subset of $\partial B(x)$. And by Lemma \ref{essLem}$(i)$, there are as many distinct faces of $B(X)$ as elements of $X$.  We will denote the faces of $B(X)$ as $\text{face}(B(X))$. 
\\\\
\noindent {\bf Vertices}. A point $y\in B(X)$ is a {\it principal vertex} if $\text{val}(y,X)\ge 3$.  Namely, $y\in B(X)$ is a principal vertex provided that $y$ belongs to at least three faces of $B(X)$. We  also define $y\in X$ as a {\it dangling vertex} if $\text{val}(y,X)=2$. That is, $y\in X$ and $y$ belongs to exactly two faces of $B(X)$.  The collection of principal and dangling vertices comprise the collection of vertices of $B(X)$ and will be denoted $\text{vert}(B(X))$.  Lemma \ref{essLem} $(iii)$ implies that each dangling vertex is essential and also that, if a principle vertex belongs to $X$, then it is also essential. 
\\\\
\noindent {\bf Edges}. Let $x,y\in \R^3$ and recall that $z\in \partial B(x)\cap \partial B(y)$ if and only if 
$$
\left|z-\frac{x+y}{2}\right|=\sqrt{1-\left|\frac{x-y}{2}\right|^2}
$$
and 
$$
\left(z-\frac{x+y}{2}\right)\cdot (x-y)=0.
$$
In particular, $\partial B(x)\cap \partial B(y)$ is a circle.  An {\it edge} of $B(X)$ is a connected component of $\partial B(x)\cap \partial B(y)\cap B(X)\setminus\text{vert}(B(X))$ with nonempty interior in $\partial B(x)\cap \partial B(y)$ for some $x,y\in X$. The edges of $B(X)$ will be denoted as $\text{edge}(B(X))$. Note that an edge is shared by exactly two faces of $B(X)$, each relative boundary point of $\partial B(x)\cap \partial B(y)\cap B(X)$ is a principle vertex, and a dangling vertex belongs to the relative interior of $\partial B(x)\cap \partial B(y)\cap B(X)$.  
\\
\par A key fact about a ball polyhedron in $\R^3$ is that its vertices, edges, and faces satisfy an Euler type formula. This was proved in Proposition 6.2 of \cite{MR2593321}; see also Corollary 6.10 of \cite{MR2343304}.
\begin{thm}\label{EulerTheorem}
Suppose $B(X)\subset \R^3$ is a ball polyhedron, $X$ is tight, and $X$ has at least three points. Then 
$$
V-E+F=2,
$$
where $V=\#\textup{vert}(B(X))$, $E=\#\textup{edge}(B(X))$, and $F=\#\textup{face}(B(X))$.

\end{thm}
\begin{figure}[h]
\centering
 \includegraphics[width=.37\textwidth]{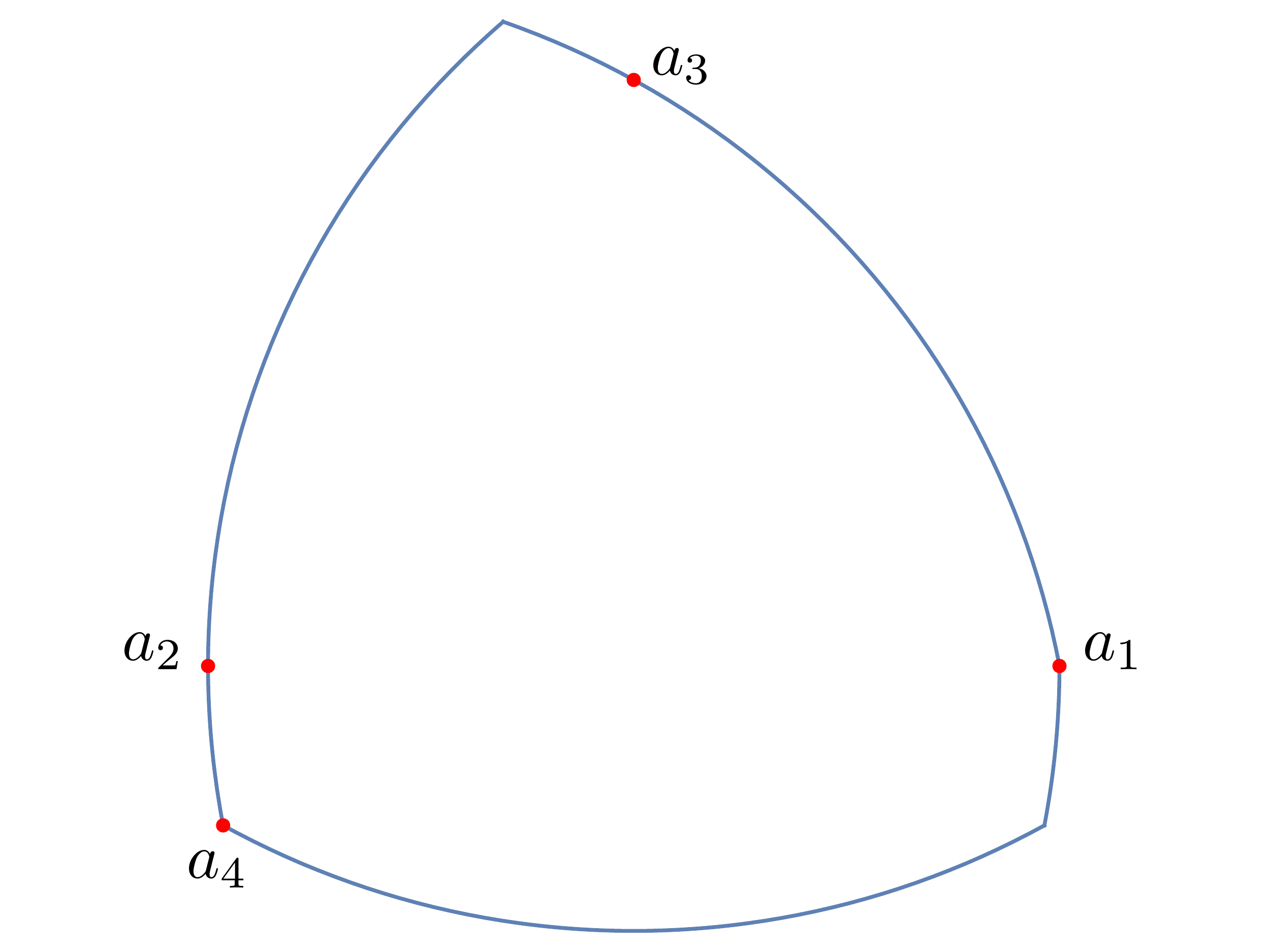}
 \hspace{.6in}
  \includegraphics[width=.37\textwidth]{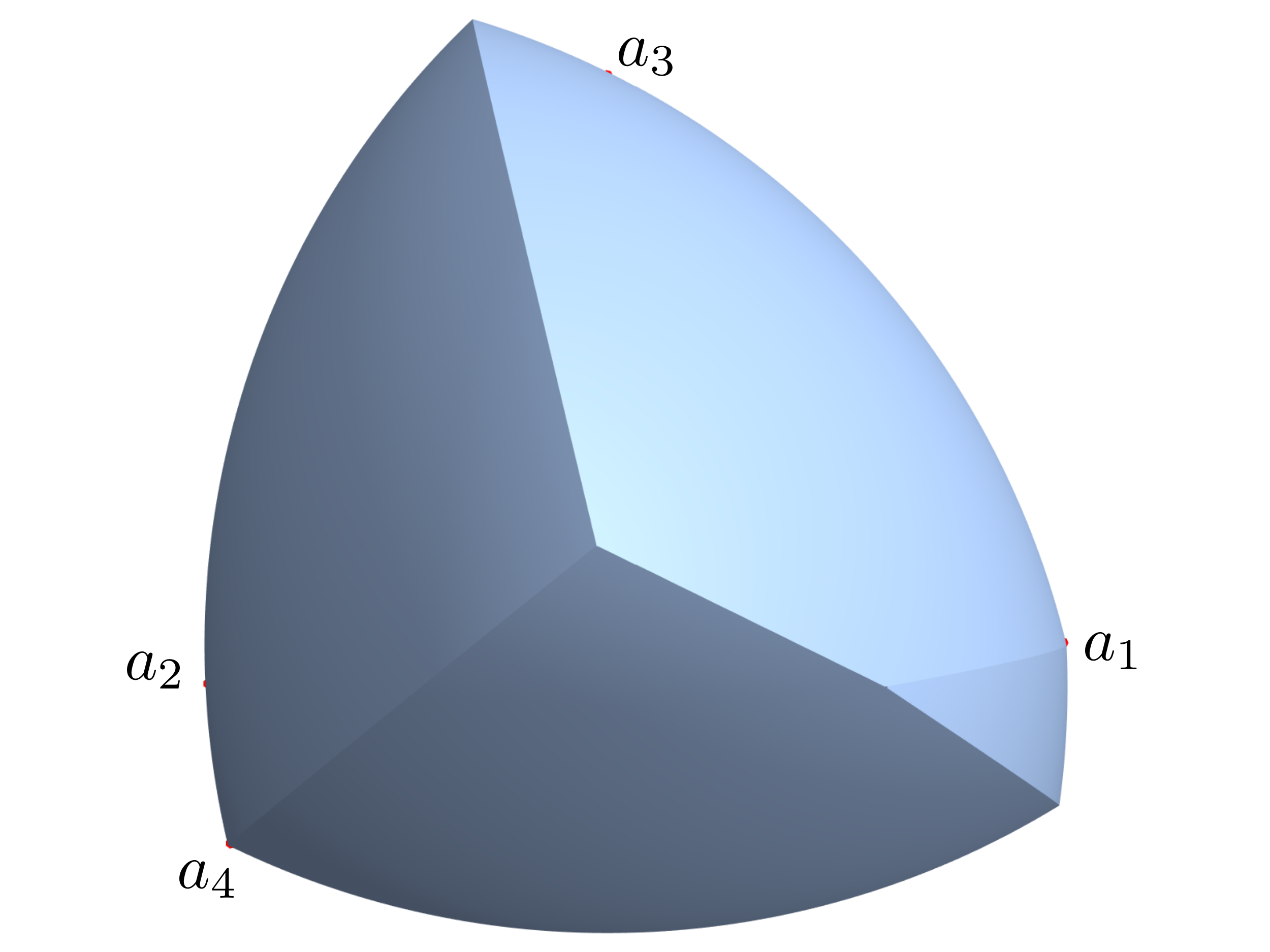}
   \includegraphics[width=.47\textwidth]{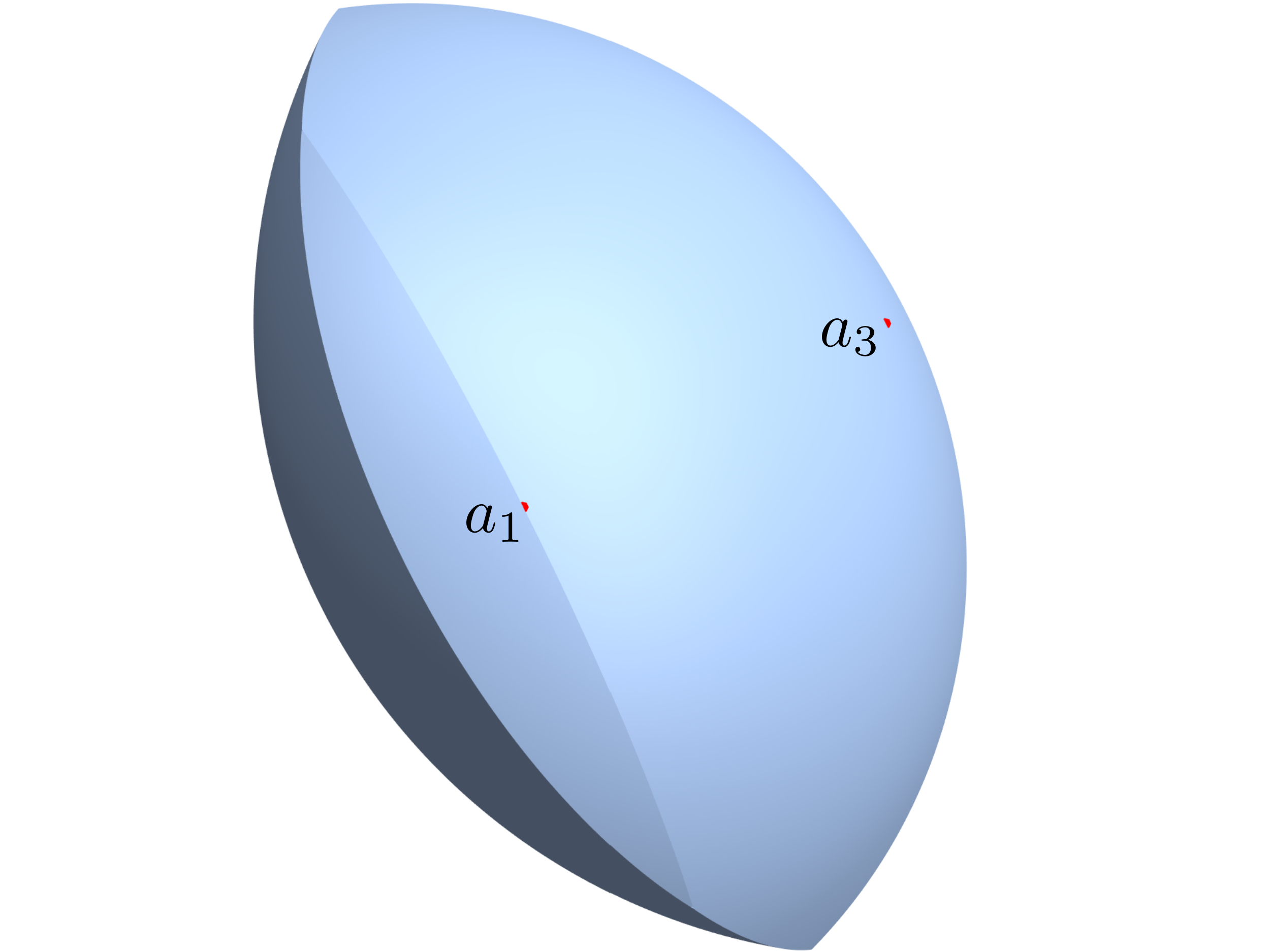} 
      \includegraphics[width=.47\textwidth]{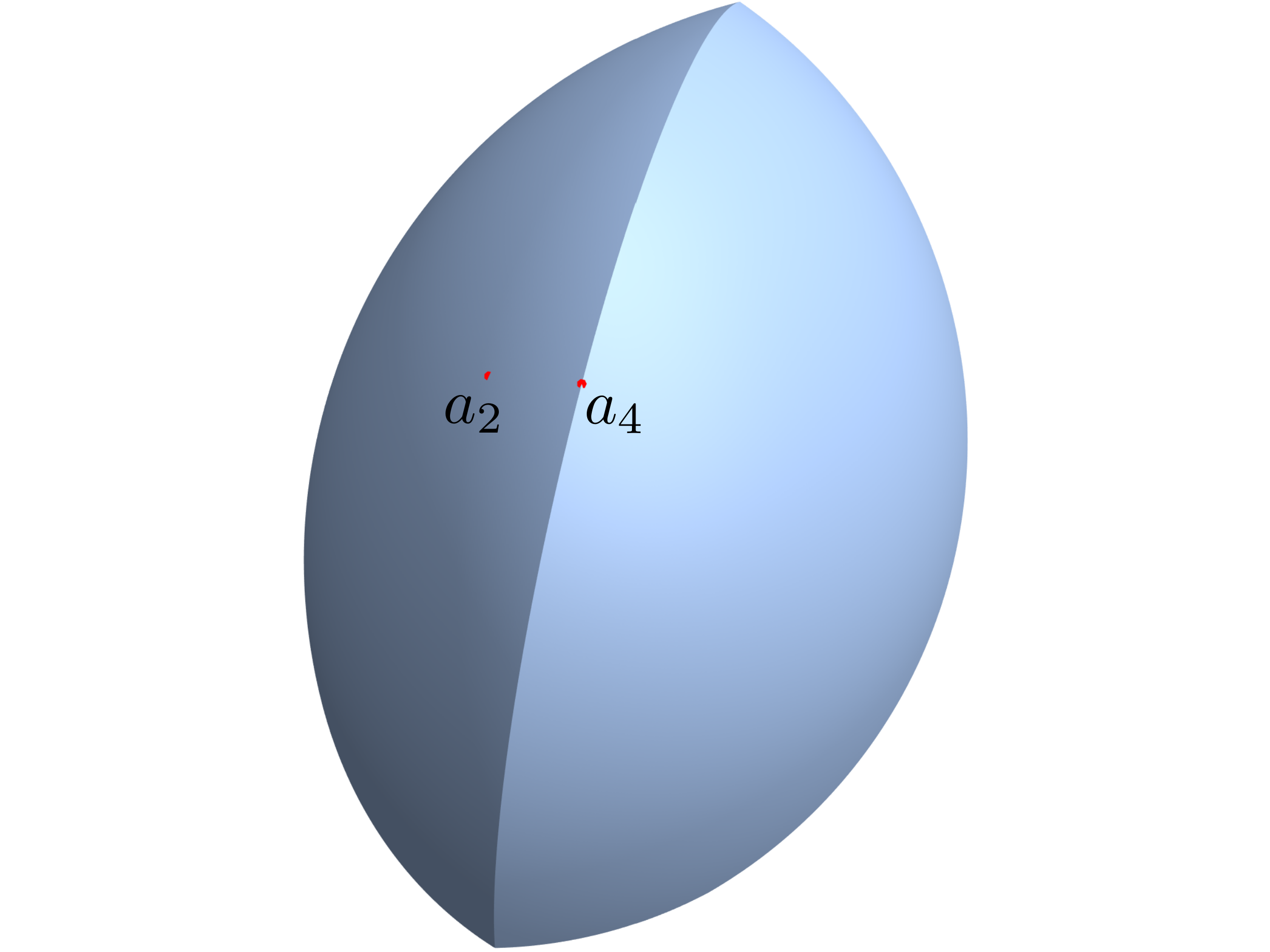} 
       \caption{This is an example of a ball polyhedron $B(\{a_1,a_2,a_3,a_4\})$, where $a_1,a_2,a_3,a_4$ belong to a common plane and $\{a_1,a_2,a_3,a_4\}$ is tight.   The upper left figure is a cross section of $B(\{a_1,a_2,a_3,a_4\})$, $a_1$ and $a_4$ are dangling vertices, and there are four principle vertices which do not belong to $\{a_1,a_2,a_3,a_4\}$.  Consequently, $V=6$, $E=8,$ and $F=4$. }\label{BallPolyEx}
\end{figure}


\subsection{V\'azsonyi problem}
Suppose $X\subset \R^3$ is a finite set with $m$ elements that has diameter equal to one. A natural problem is to determine 
how many pairs of points $x,y\in X$ are there with $|x-y|=1$.  We will say that $X$ is {\it extremal} if it has the maximum possible 
amount of diametric pairs among finite sets of $m$ elements with diameter one.  Let us denote $e(X)$ for the number of diametric pairs of $X$. V\'azsonyi conjectured that 
\be\label{VazsyoniIneq}
e(X)\le 2m-2. 
\ee
This conjecture was verified independently by Gr\"unbaum \cite{MR87115}, Heppes \cite{MR87116}, and Straszewicz \cite{MR0087117}; see Chapter 13 of \cite{MR1354145} for a concise discussion of this result.   This theorem was also extended by Kupitz, Martini, and Perles \cite{MR2593321} as follows.  
\begin{thm}\label{ExtendedGHS}
Suppose $X\subset \R^3$ is a finite set with $m\ge 4$ points and $\textup{diam}(X)=1$. The following are equivalent. \\
$(i)$ $X$ is extremal. \\
$(ii)$ $e(X)=2m-2$.\\
$(iii)$ $X=\textup{vert}(B(X)).$
\end{thm}

\begin{ex}\label{ReuleauxEx5} A basic example of an extremal set is $\{a_1,a_2,a_3,a_4\}$, where $a_1,\dots, a_4$ are the four vertices of a regular tetrahedron of side length one.  Note that $|a_i-a_j|=1$ for all $i\neq j$ and that there are six such pairs $\{a_1,a_2\},\{a_1,a_3\},\{a_1,a_4\},\{a_2,a_3\},\{a_2,a_4\},\{a_3,a_4\}$. Since $2\cdot 4-2=6$, $\{a_1,a_2,a_3,a_4\}$ is indeed extremal.  In addition, we can add points to this set in order to construct an extremal set having as many points as we wish. Consider the edge $e$ which joins $a_3$ to $a_4$. For any $a_5\in e$, $|a_5-a_1|=|a_5-a_2|=1$ and $|a_5-a_i|<1$ for $i=3,4$. Thus, $\{a_1,\dots,a_5\}$ has $2+6=8=2\cdot 5-2$ diametric pairs; see Example \ref{AlmostReulEx} below. Moreover, we can add any  number $\{a_5,\dots,a_m\}$ of distinct points to $e$ and deduce that $\{a_1,\dots, a_m\}$ is extremal.
\end{ex}

\par Suppose $X\subset \R^3$ has $m\ge 4$ points and $X$ is extremal. By Theorem \ref{ExtendedGHS}, $X=\text{vert}(B(X))$, so that $B(X)$ has $m$ vertices. We've already seen that $B(X)$ has $m$ faces. In view of the Euler formula discussed in Theorem \ref{EulerTheorem}, $2=V-E+F=2m- E$. That is, $B(X)$ has 
$$
E=2(m-1)
$$
edges.  Another important observation of Kupitz, Martini, and Perles (Theorem 8.1 of \cite{MR2593321}) is that edges of $B(X)$ are naturally grouped in pairs. We will call any two such edges as in the theorem below {\it a dual edge pair}. By the computation of $E$ above, $B(X)$ has $m-1$ dual edge pairs. 
\begin{thm}\label{edgeThm}
Assume a finite $X\subset \R^3$ with at least four points, $\textup{diam}(X)=1$, and $X$ is extremal. Further suppose $e\subset \partial B(x)\cap \partial B(y)\cap B(X)$ is an edge of $B(X)$ with endpoints $x',y'\in X$. There is a unique edge $e'\subset \partial B(x')\cap \partial B(y')\cap B(X)$ of $B(X)$ with endpoints $x,y$.   
\end{thm}

\par  Suppose $X\subset \R^3$ is finite with at least four points and that the diameter of $X$ is equal to one. Following Montejano and Rold\'an-Pensado \cite{MR3620844} and Sallee \cite{MR296813}, we will say that the ball polyhedron $B(X)$ is a {\it Reuleaux polyhedron} provided that $X$ is extremal.   Just as Meissner's tetrahedra are designed from a Reuleaux tetrahedron, various constant width shapes can be constructed starting from 
Reuleaux polyhedra. The Density Theorem asserts that essentially all constant width shapes can be built this way.   

\par However, we note that Montejano and Rold\'an-Pensado \cite{MR3620844} and Sallee \cite{MR296813} restricted their definitions to the case in which $B(X)$ has no dangling vertices (Sallee called these shapes {\it frames} and reserved the term `Reuleaux polyhedra' for another class of constant width shapes).  In this case, $X$ is equal to the principal vertices of $B(X)$ and we will say that $X$ is {\it critical}.    For instance, in Example \ref{ReuleauxEx5} above, $\{a_1,a_2,a_3,a_4,a_5\}$ is extremal and $\{a_1,a_2,a_3,a_4\}$ is critical.    

\subsection{Examples}
Suppose $X\subset \R^3$ is finite and the diameter of $X$ is equal to one. The {\it skeleton} of $X$ is the graph 
whose vertices and edges are $\text{vert}(B(X))$ and $\text{edge}(B(X))$, respectively. 
The skeleton of 
a ball polyhedra is known to be planar graph which is 2--connected  (Theorem 6.1 of \cite{MR2593321}).  When $X$ is extremal, the face complex of $B(X)$ is strongly self-dual as detailed in Section 9 of \cite{MR2593321} and Chapter 6 of \cite{MR3930585}.   We will display skeletons along with our plots of Reuleaux polyhedra in order to help visualize these shapes.  Each figure in this paper was generated with \texttt{Mathematica} as described in the appendix.  Moreover, many of the examples displayed either can  be found or are based on the examples in Chapter 6 and 8 of \cite{MR3930585} and in the paper \cite{MR3620844}. We also recommend the article \cite{MR4156257} for a procedure which has been used to find hundreds of other examples.

\begin{ex} Our first example is a Reuleaux tetrahedron. The vertices $\{a_1,a_2,a_3,a_4\}$ of this shape are also the vertices of a regular tetrahedron with side length one.  Each pair of vertices has an edge between them, so the skeleton is simply the complete graph on four vertices.  In the graph below, we have indicated dual edges with the same color.  
\begin{figure}[h]
\centering
 \includegraphics[width=.36\textwidth]{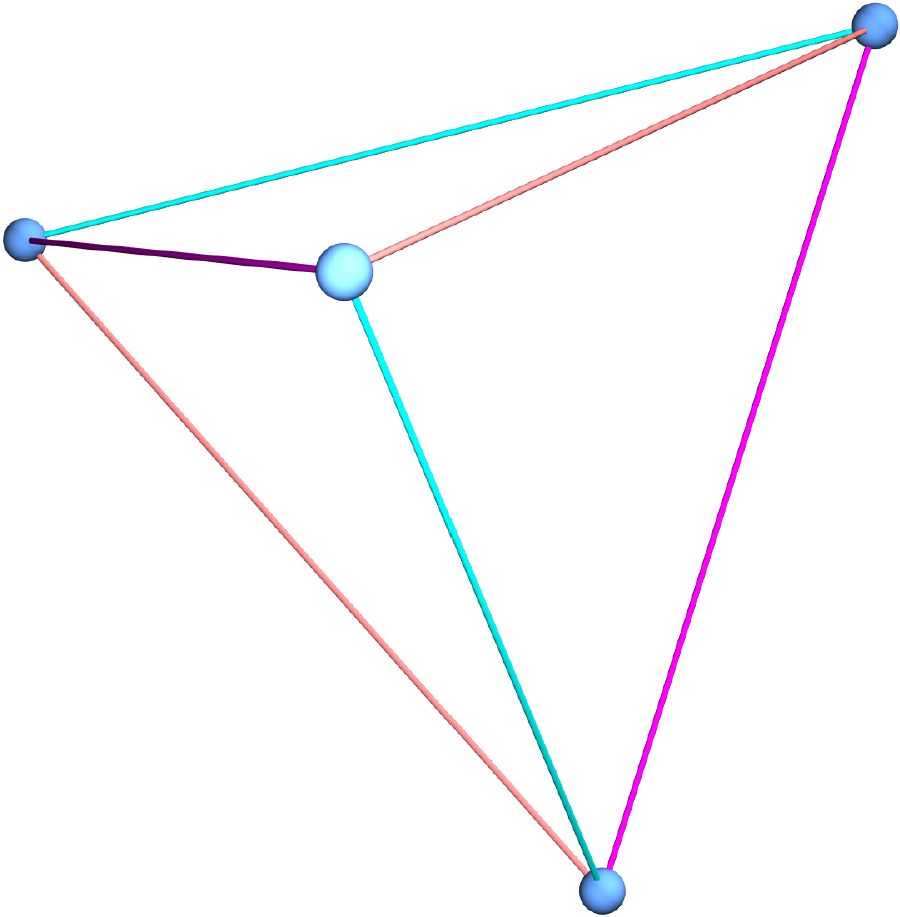}
 \hspace{.3in}
  \includegraphics[width=.4\textwidth]{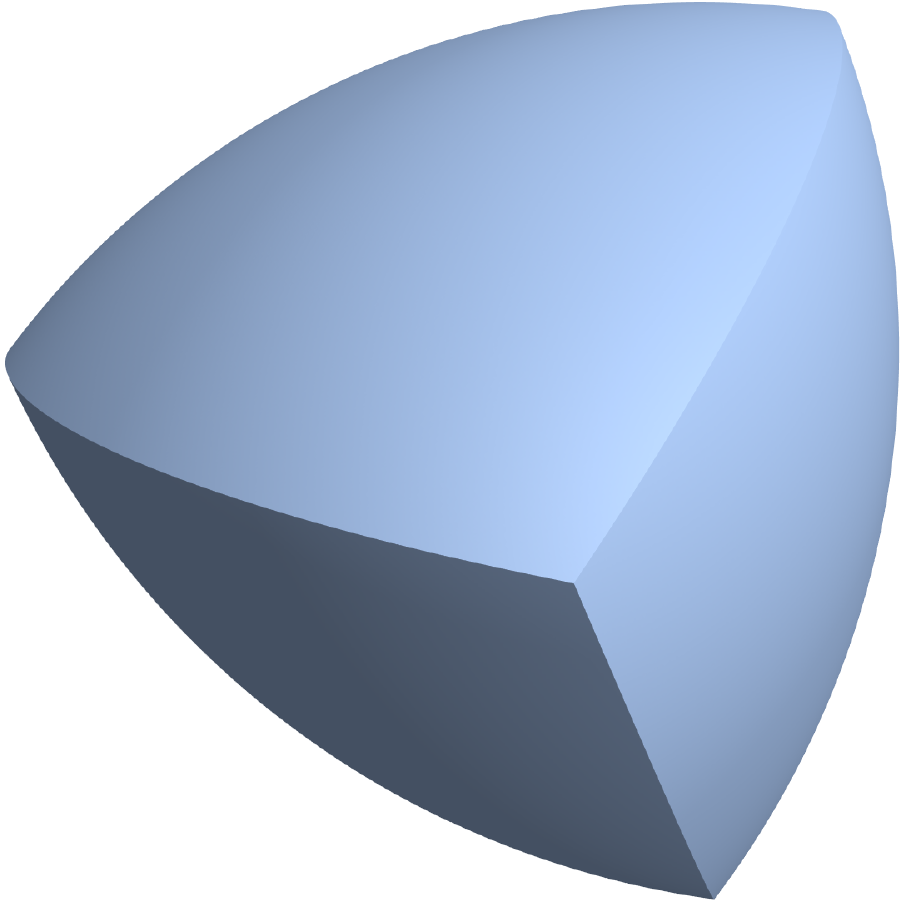}
  
  \vspace{.3in}
  
   \includegraphics[width=.4\textwidth]{ReulTetraUno1} 
    \hspace{.3in}
      \includegraphics[width=.4\textwidth]{ReulTetraDos1} 
\end{figure}
\end{ex}

 \newpage 
\begin{ex}\label{AlmostReulEx}  Building on our previous example, we assume $\{a_1,a_2,a_3,a_4\}$ are the vertices of a Reuleaux tetrahedron. We also choose $a_5$ as the midpoint of the edge joining $a_3$ and $a_4$ and consider $\{a_1,a_2,a_3,a_4, a_5\}$. As noted above, this set of points is extremal but not critical as $a_5$ is a dangling vertex of the associated Reuleaux polyhedron displayed below.  Again we have indicated dual edges with the same color.   
\begin{figure}[h]
\centering
 \includegraphics[width=.36\textwidth]{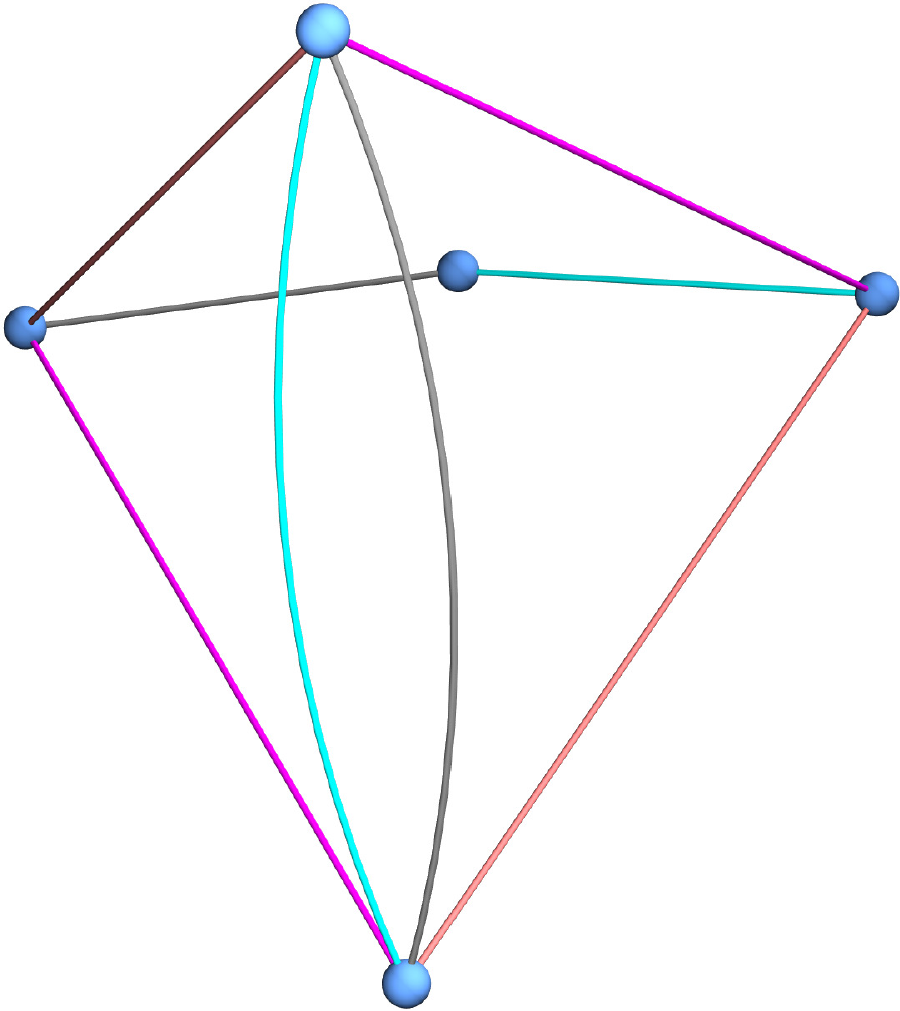}
 \hspace{.3in}
  \includegraphics[width=.4\textwidth]{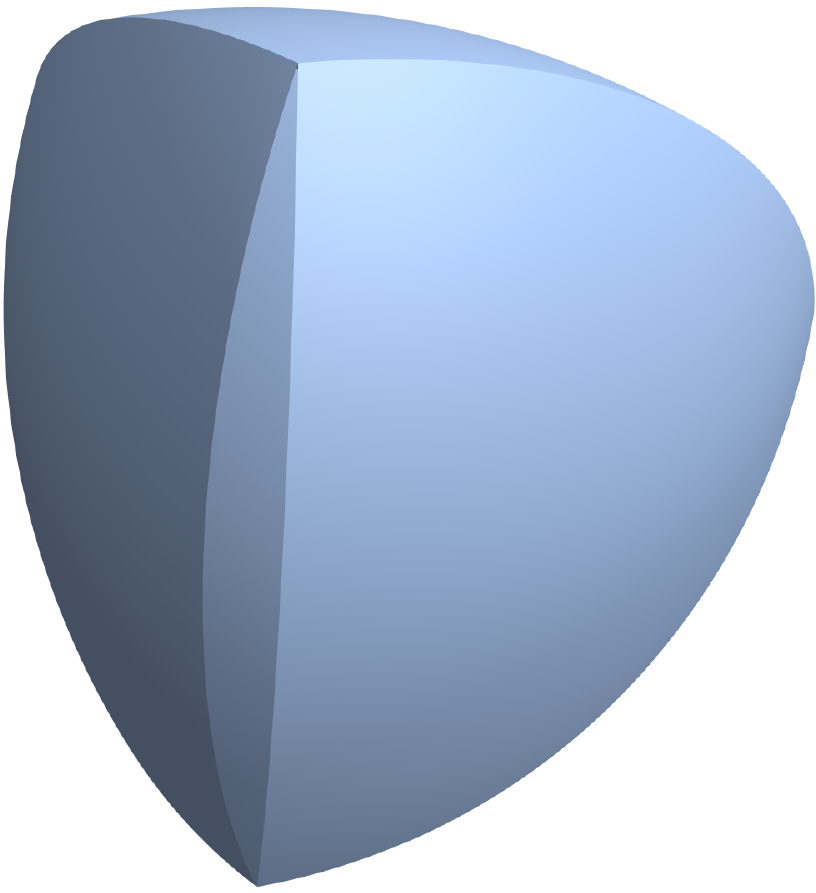}
  
  \vspace{.3in}
  
   \includegraphics[width=.4\textwidth]{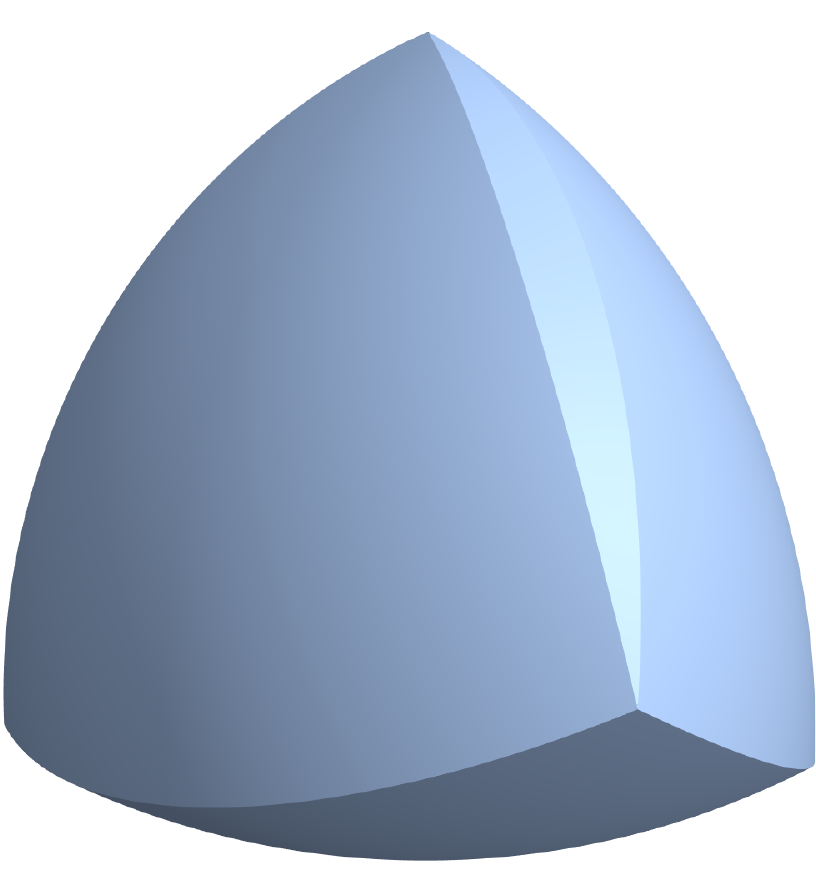} 
    \hspace{.3in}
      \includegraphics[width=.4\textwidth]{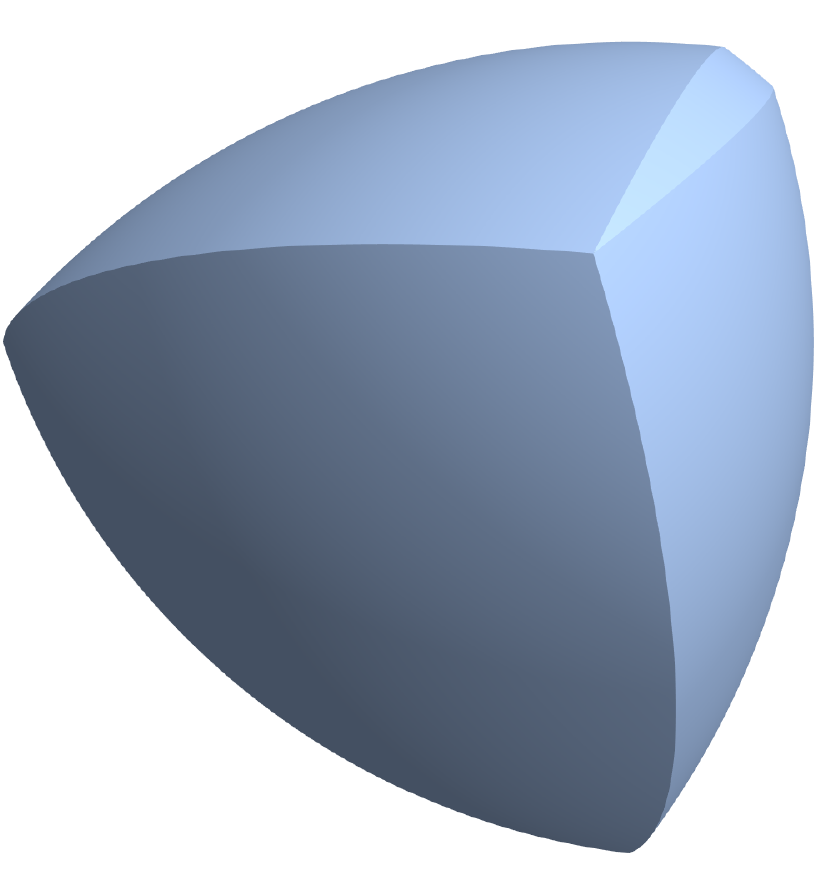} 
\end{figure}
\end{ex}

\newpage 
\begin{ex}\label{nonpolyEx}
Here is an example of an critical set $X=\{a_1,\dots, a_8\}$ discussed in the introduction of \cite{MR2593321}, although it was not explicitly described.   This example is of particular interest as the skeleton of $B(X)$ is 2--connected but not 3--connected.  $X$ also includes the vertices of a regular tetrahedron as a proper subset. 
\begin{figure}[h]
\centering
 \includegraphics[width=.36\textwidth]{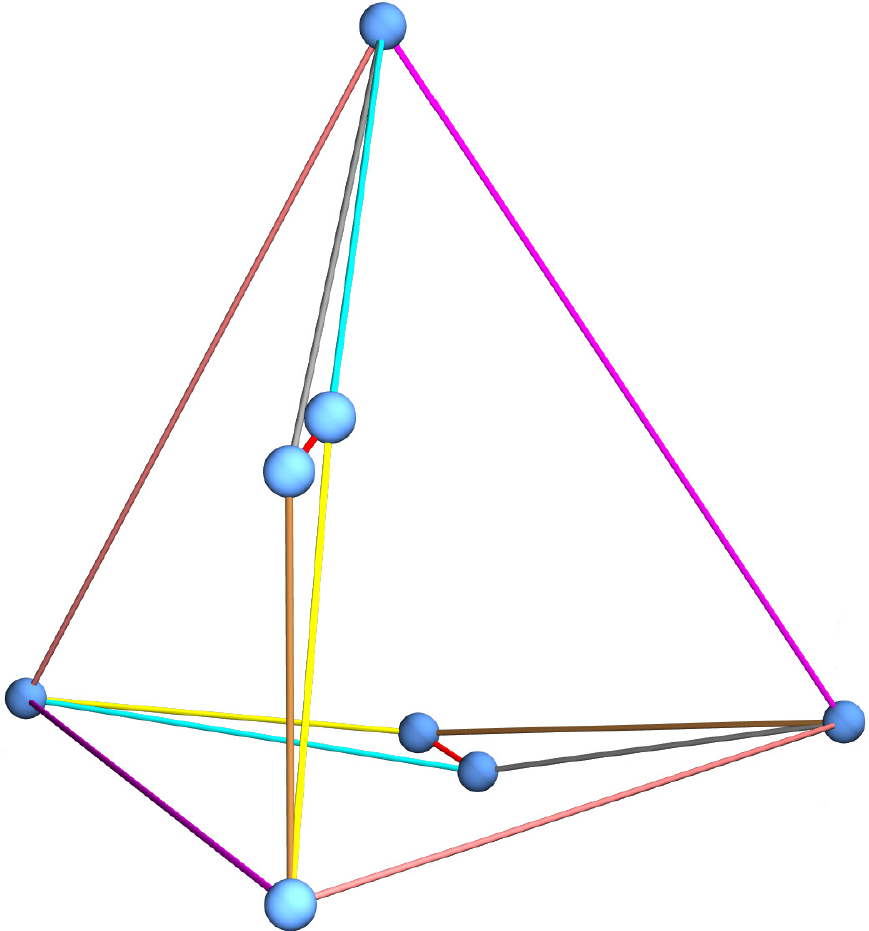}
 \hspace{.3in}
  \includegraphics[width=.4\textwidth]{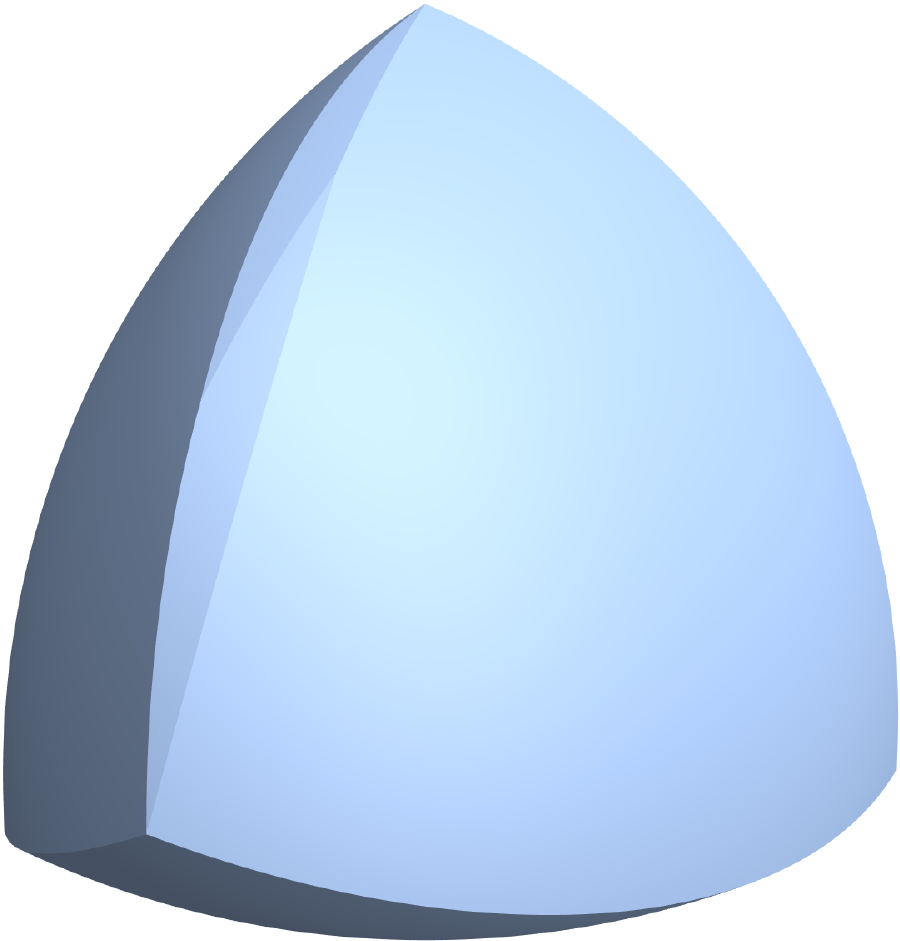}
  
  \vspace{.3in}
  
   \includegraphics[width=.4\textwidth]{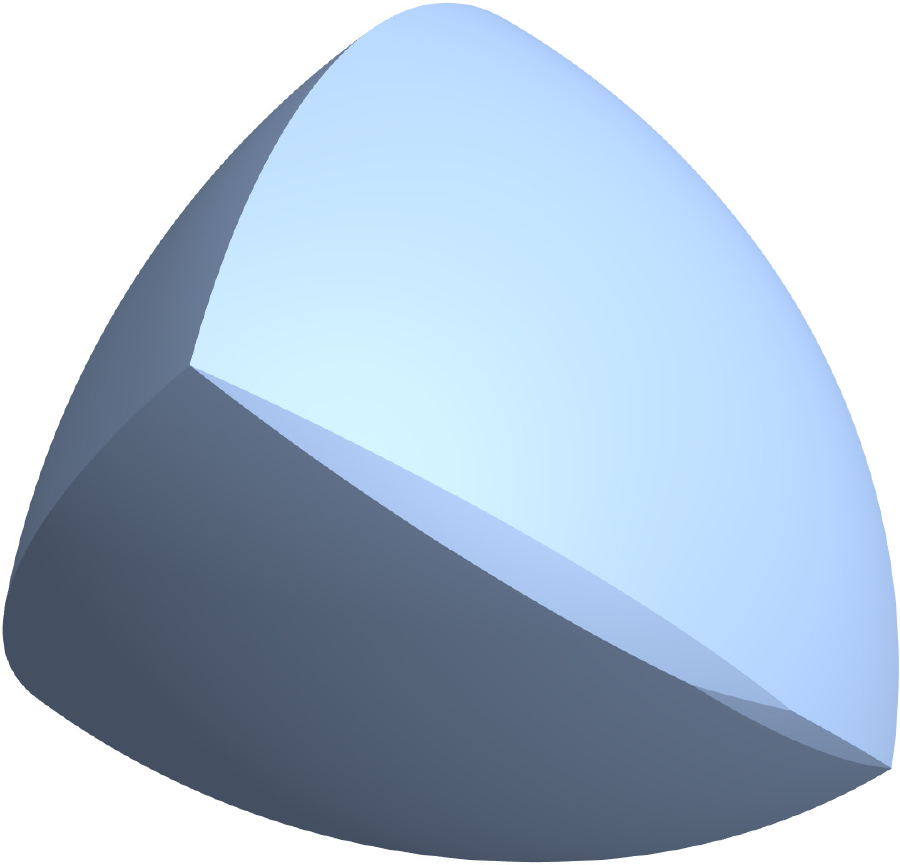} 
    \hspace{.3in}
      \includegraphics[width=.4\textwidth]{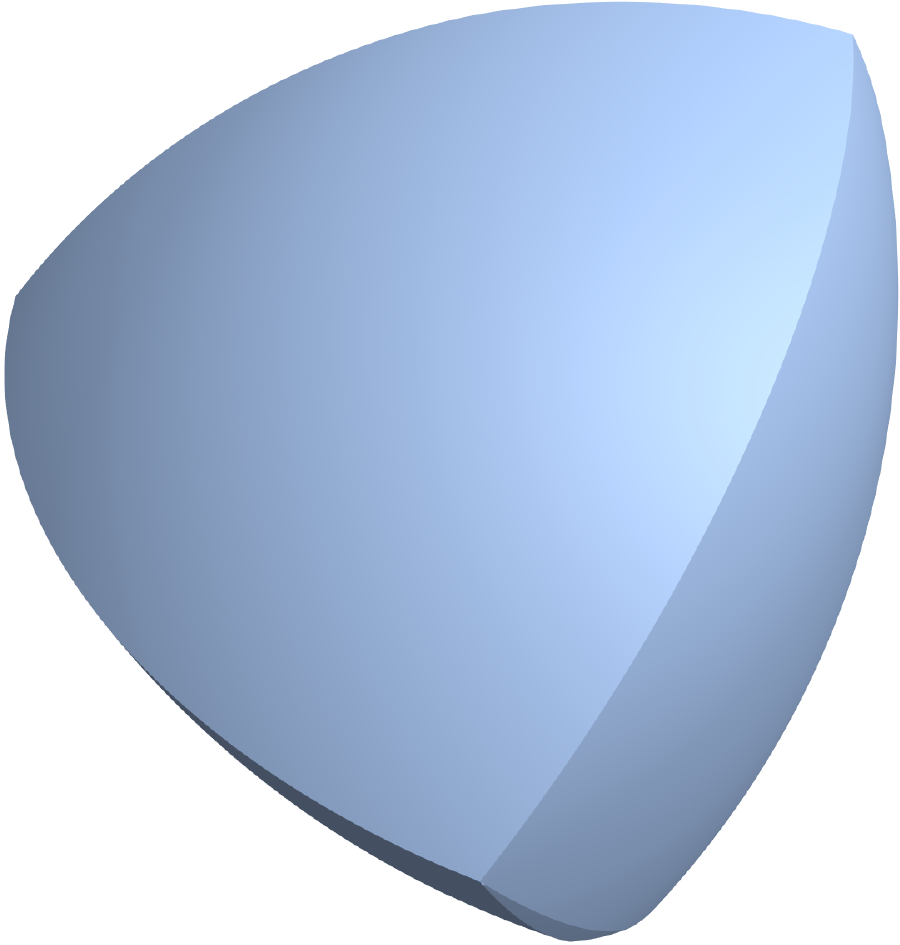} 
\end{figure}
\end{ex}

\newpage 
\begin{ex}\label{PentPyrEx} Suppose $a_1,\dots,a_5\subset \R^3$ belong to a plane and are the vertices of a regular pentagon of diameter one. It is possible to choose $a_6\in \R^3$ such that $|a_6-a_j|=1$ for $j=1,\dots, 5$.  It turns out that $\{a_1,\dots,a_6\}\subset\R^3$ is critical.  Three views of $B(\{a_1,\dots,a_6\})$ are shown below along with a corresponding skeleton.  We also note that it is possible to generalize this construction with other odd-sided regular polygons in order to obtain more Reuleaux polyhedra. This is described in example 1.2 of \cite{MR2593321}. 
\begin{figure}[h]
\centering
 \includegraphics[width=.37\textwidth]{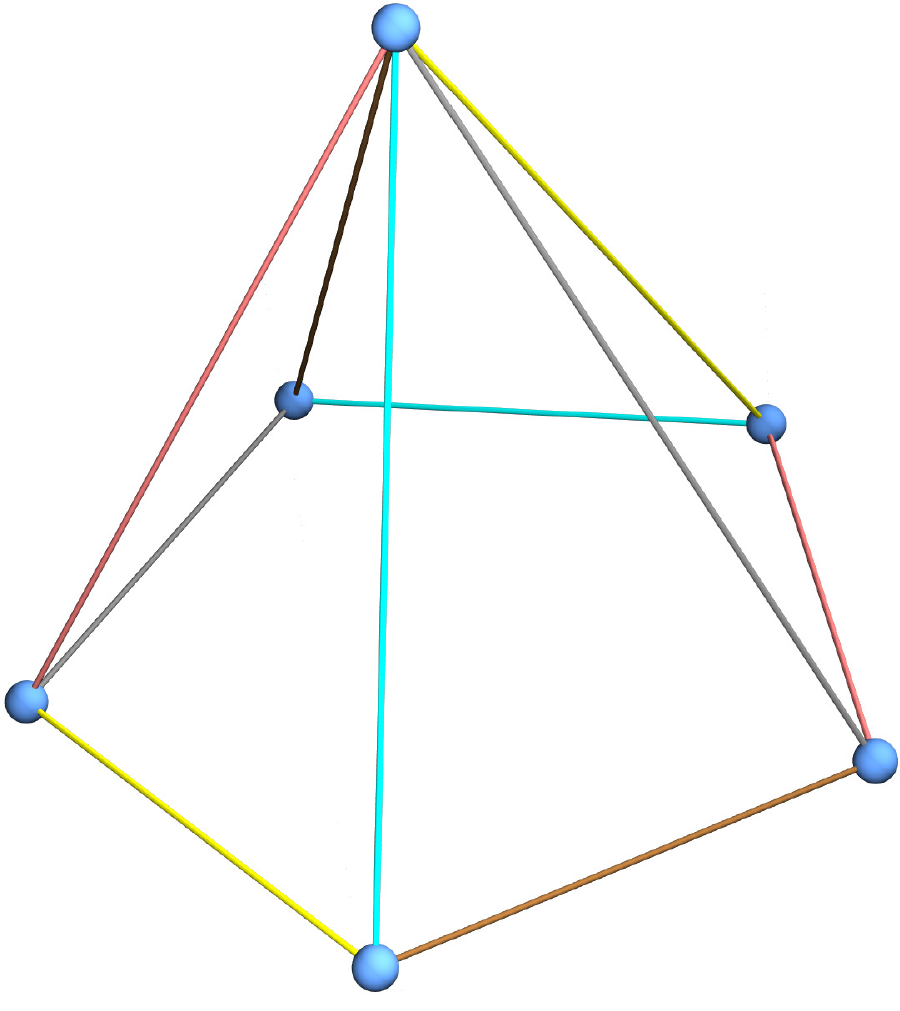}
     \hspace{.3in}
  \includegraphics[width=.4\textwidth]{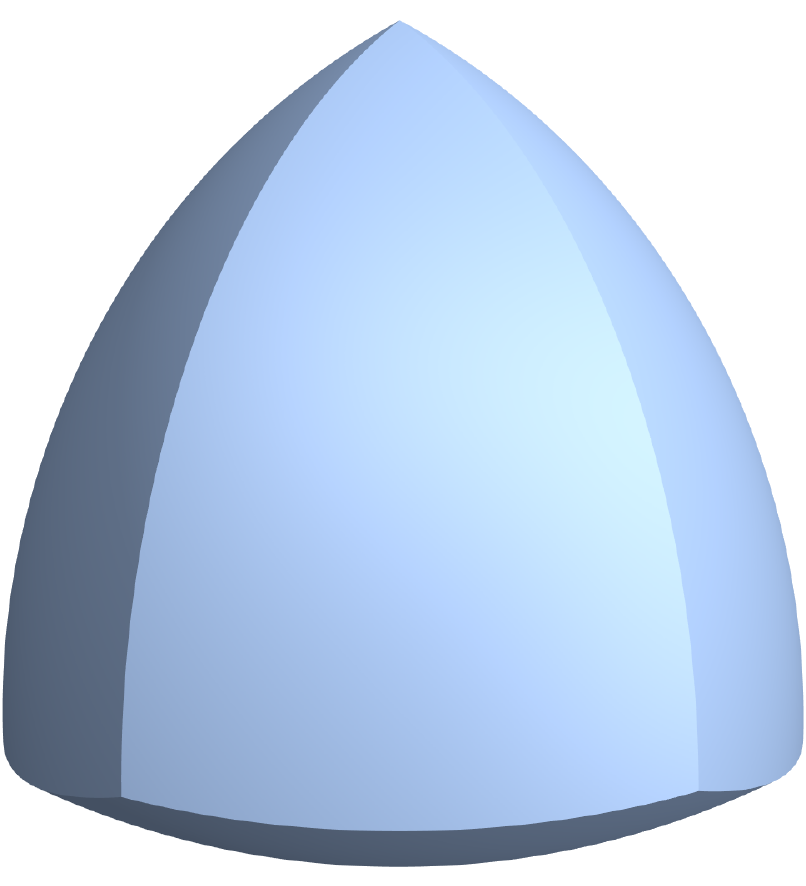}
 
  \vspace{.3in}
  
   \includegraphics[width=.4\textwidth]{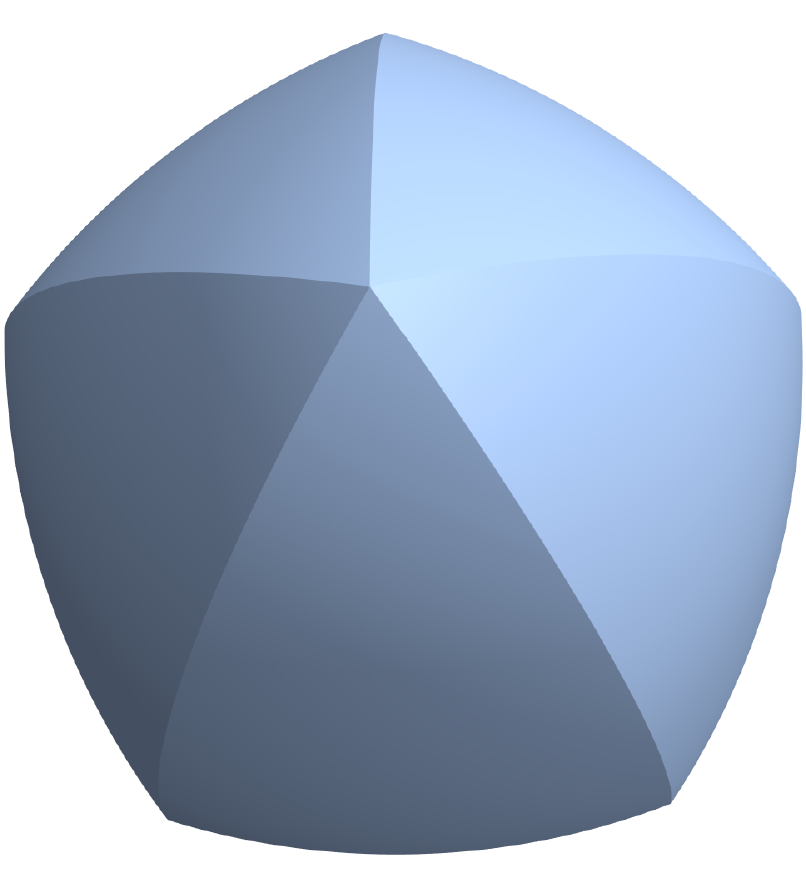}
       \hspace{.3in} 
      \includegraphics[width=.4\textwidth]{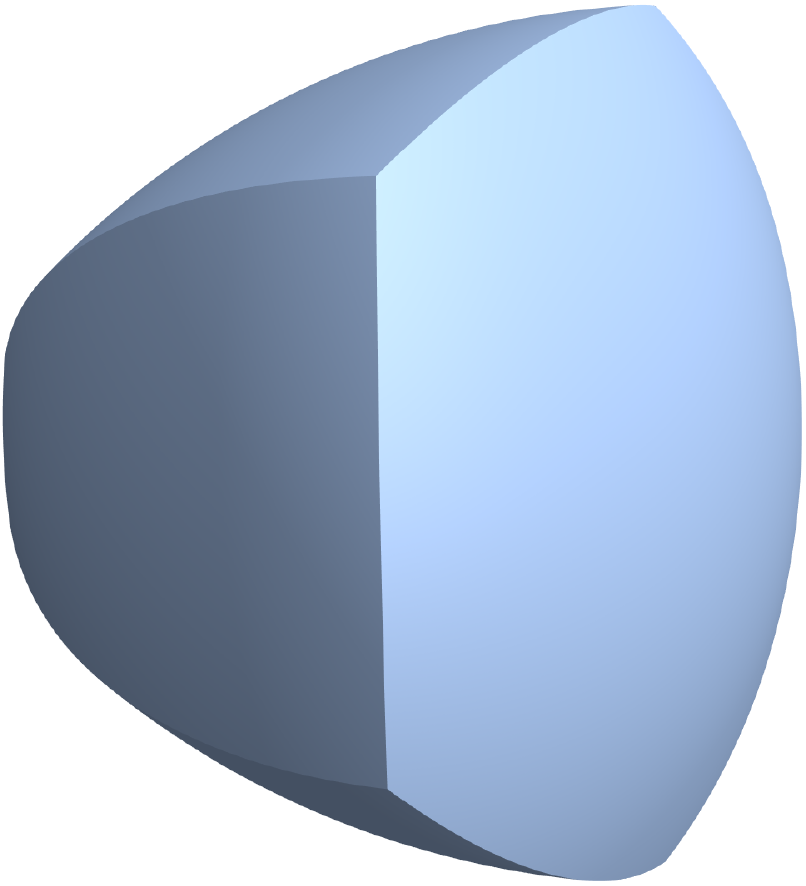} 
\end{figure}
\end{ex}

\newpage 
\begin{ex}\label{PentPyrExTwo} Recall the critical set $\{a_1,\dots,a_6\}\subset \R^3$ from our previous example.  If we choose two points $a_7$ and $a_8$ in two distinct (non dual) edges of $B(\{a_1,\dots,a_6\})$, then $\{a_1,\dots,a_8\}$ is an extremal set.  Its associated skeleton and Reuleaux polyhedron are displayed below. 

\begin{figure}[h]
\centering
 \includegraphics[width=.37\textwidth]{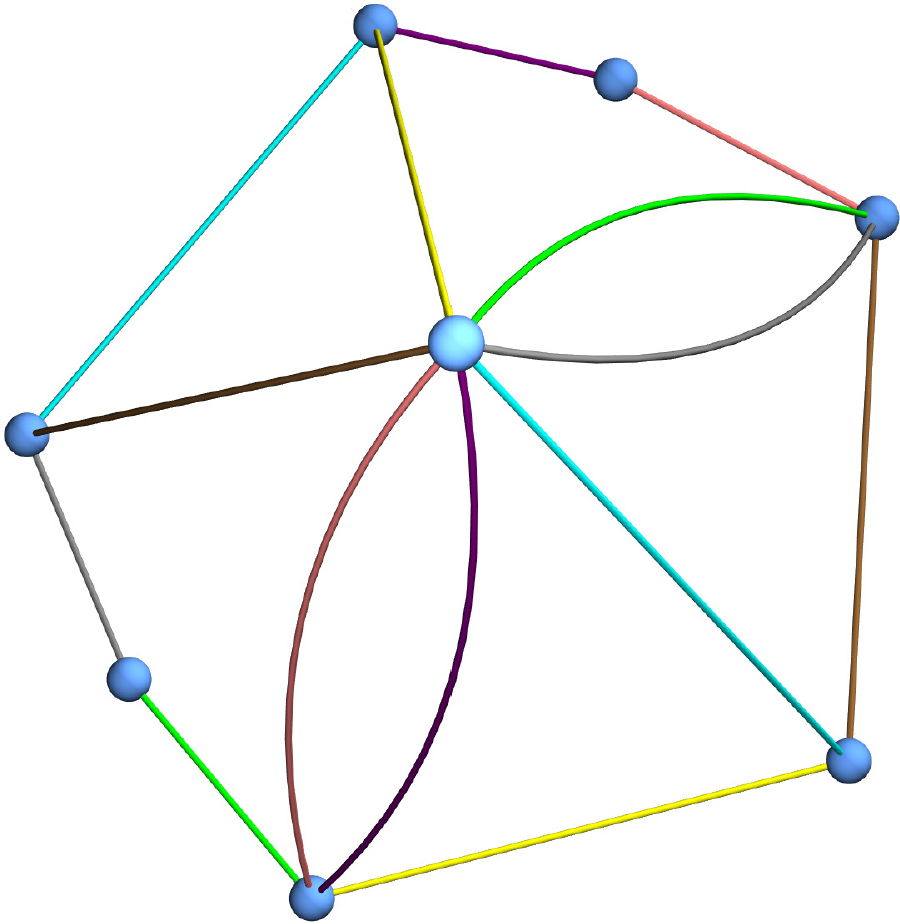}
     \hspace{.3in}
  \includegraphics[width=.4\textwidth]{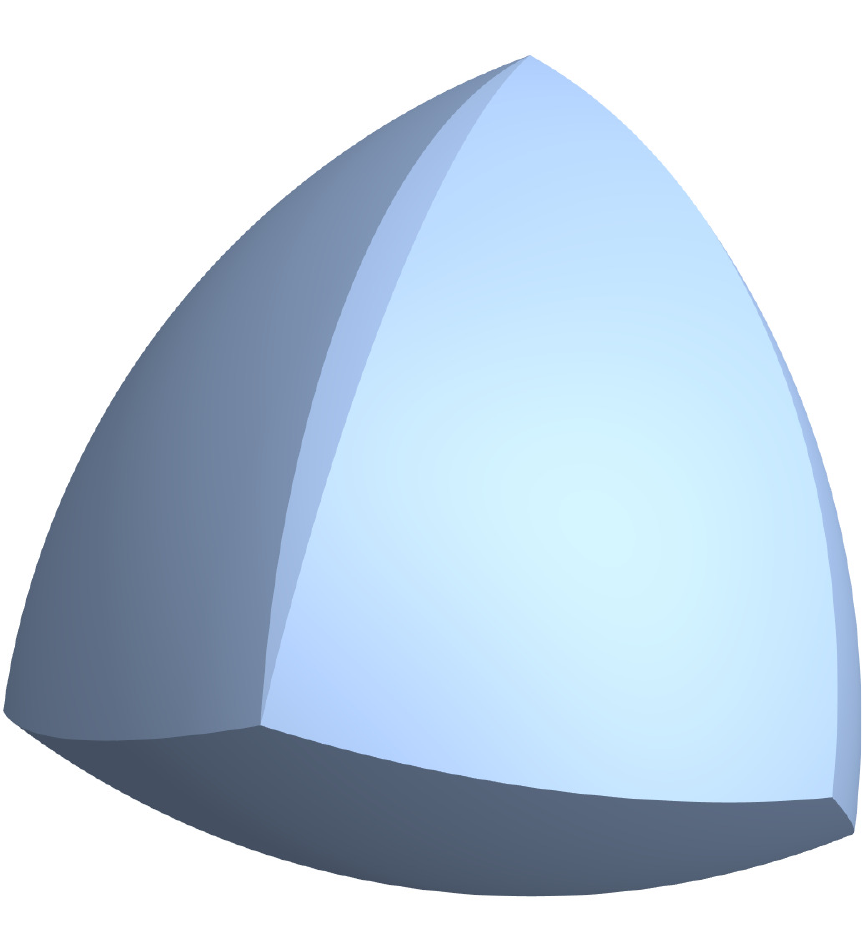}
 
  \vspace{.3in}
  
   \includegraphics[width=.4\textwidth]{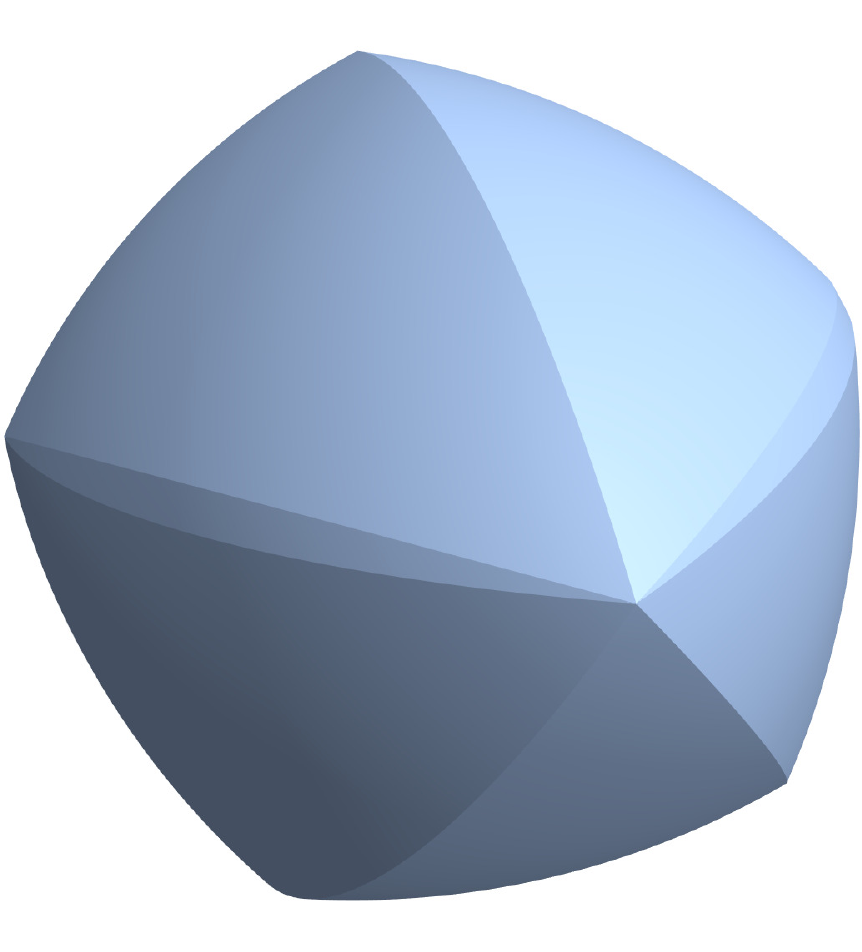}
       \hspace{.3in} 
      \includegraphics[width=.4\textwidth]{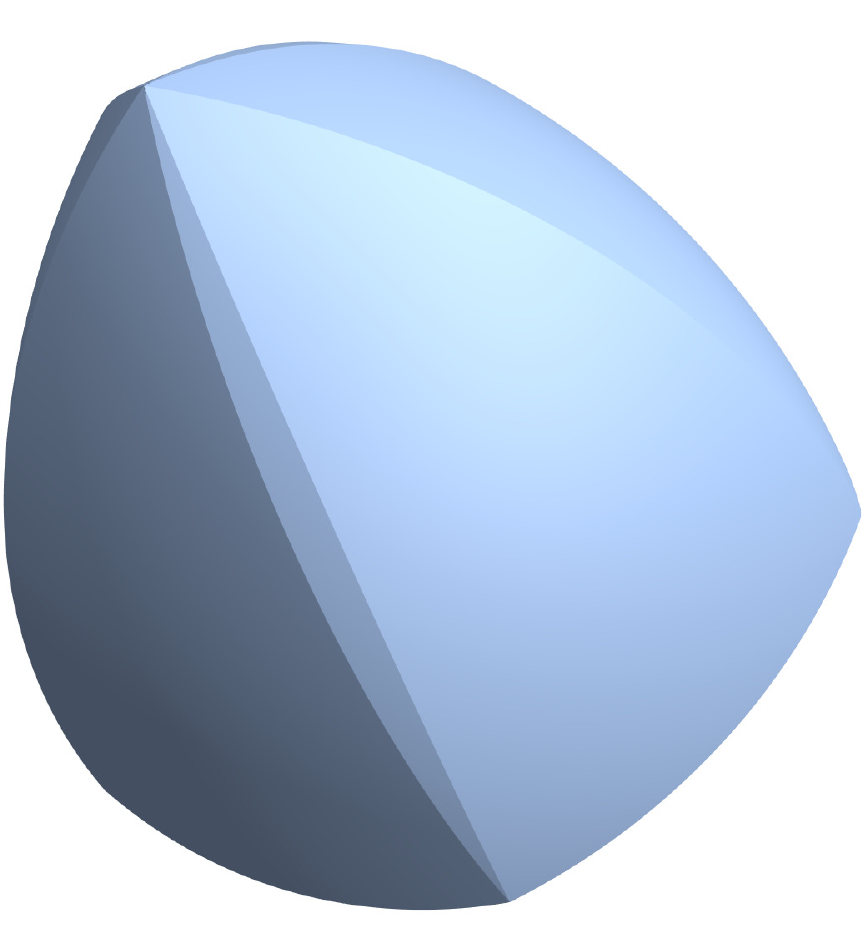} 
\end{figure}
\end{ex}

\newpage 
\begin{ex}\label{ETPexample} We can design a Reuleaux polyhedron by choosing vertices  $\{a_1,\dots, a_7\}$ analogous to those of an elongated triangular pyramid.  Indeed for each $t\in (0,1)$ 
$$
\begin{cases}
a_1=(0,0,\sqrt{1-t^2/3}-\sqrt{1-(t^2+t+1)/3})\\
a_2=(0,1/\sqrt{3},0)\\
a_3=(-1/2,-1/(2\sqrt{3}),0)\\
a_4=(1/2,-1/(2\sqrt{3}),0)\\
a_5=ta_2-(0,0,\sqrt{1-(t^2+t+1)/3})\\
a_6=ta_3-(0,0,\sqrt{1-(t^2+t+1)/3})\\
a_7=ta_4-(0,0,\sqrt{1-(t^2+t+1)/3})
\end{cases}
$$ 
is critical.  The corresponding skeleton and Reuleaux polyhedron $B(\{a_1,\dots, a_7\})$ are displayed below. 

\begin{figure}[h]
\centering
 \includegraphics[width=.33\textwidth]{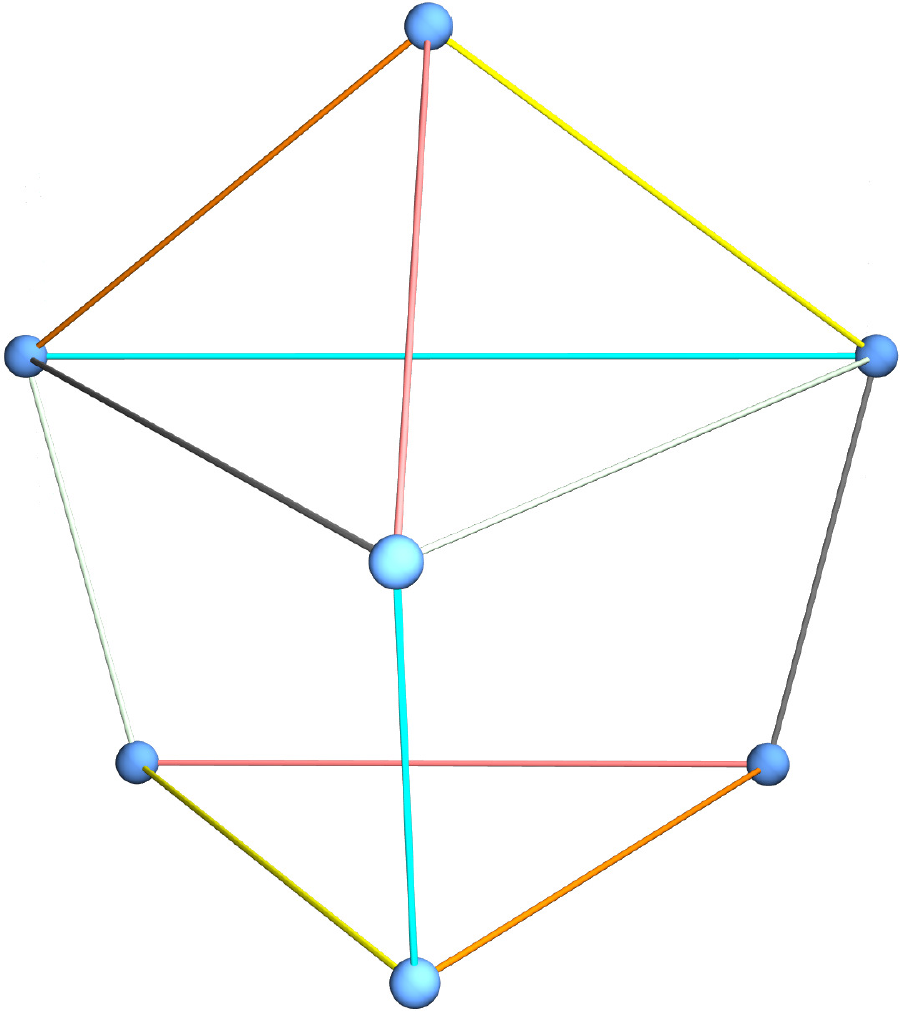}
      \hspace{.3in}
  \includegraphics[width=.4\textwidth]{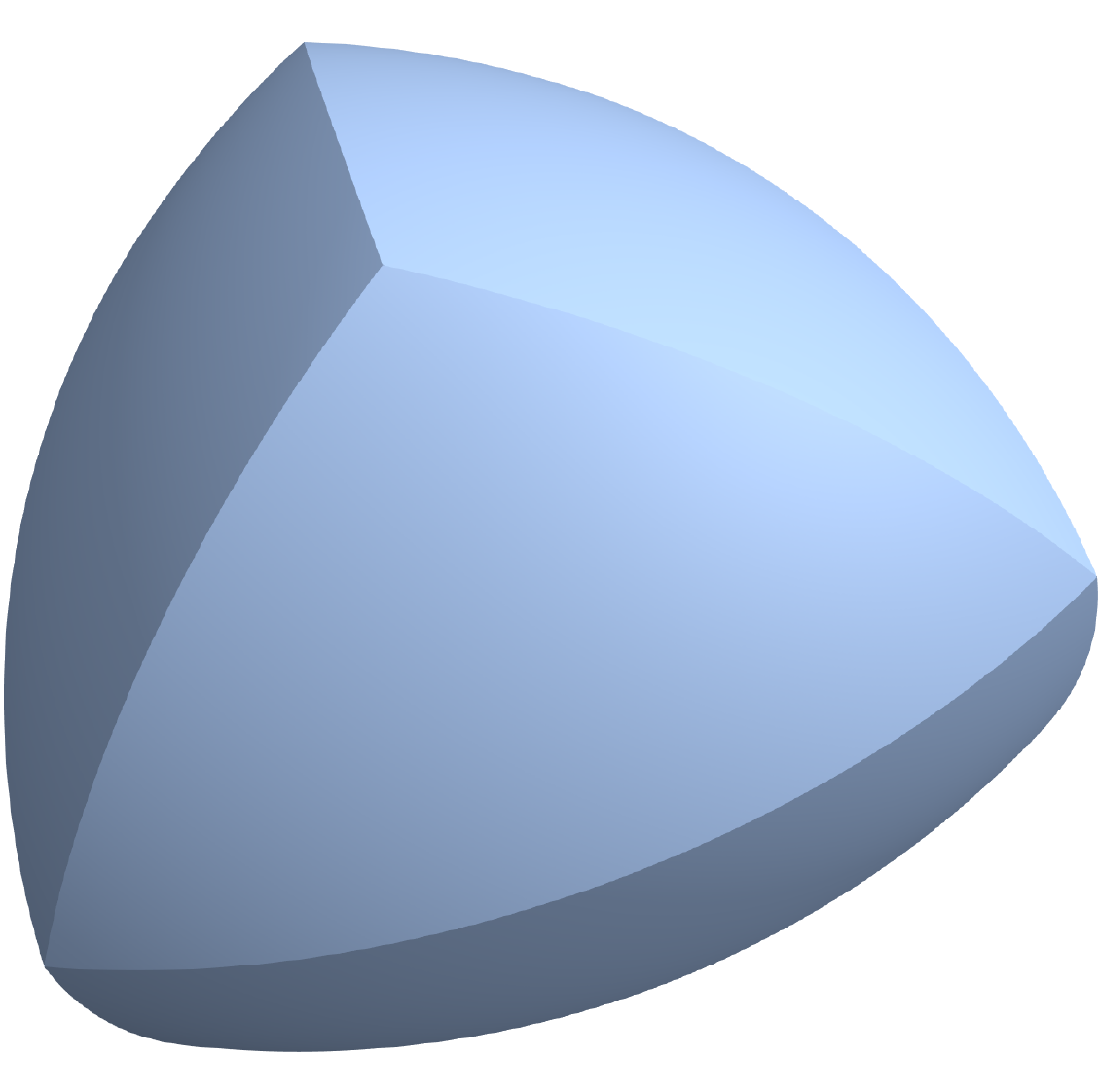}
  
  \vspace{.3in}
  
   \includegraphics[width=.4\textwidth]{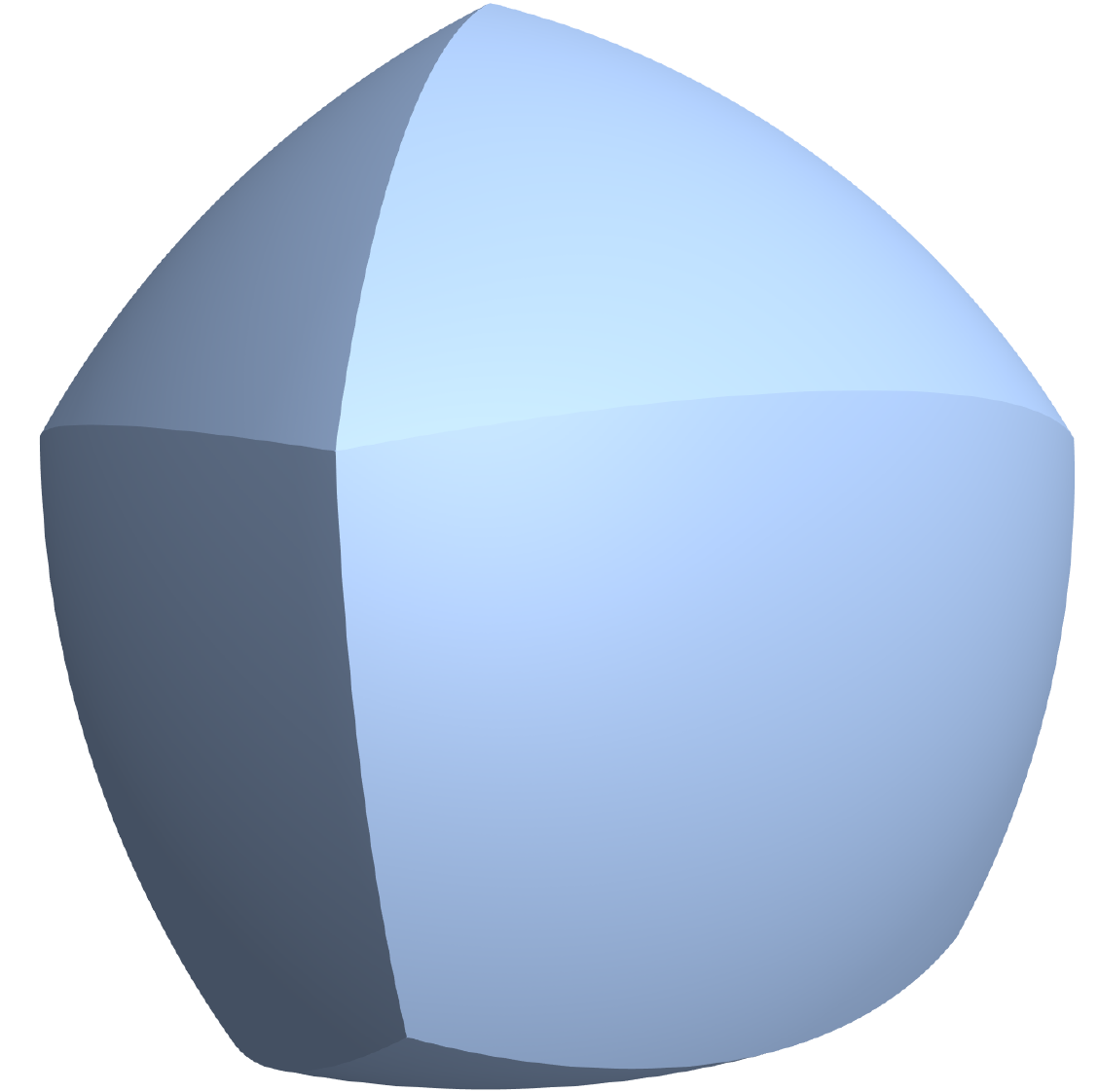}
        \hspace{.3in} 
      \includegraphics[width=.4\textwidth]{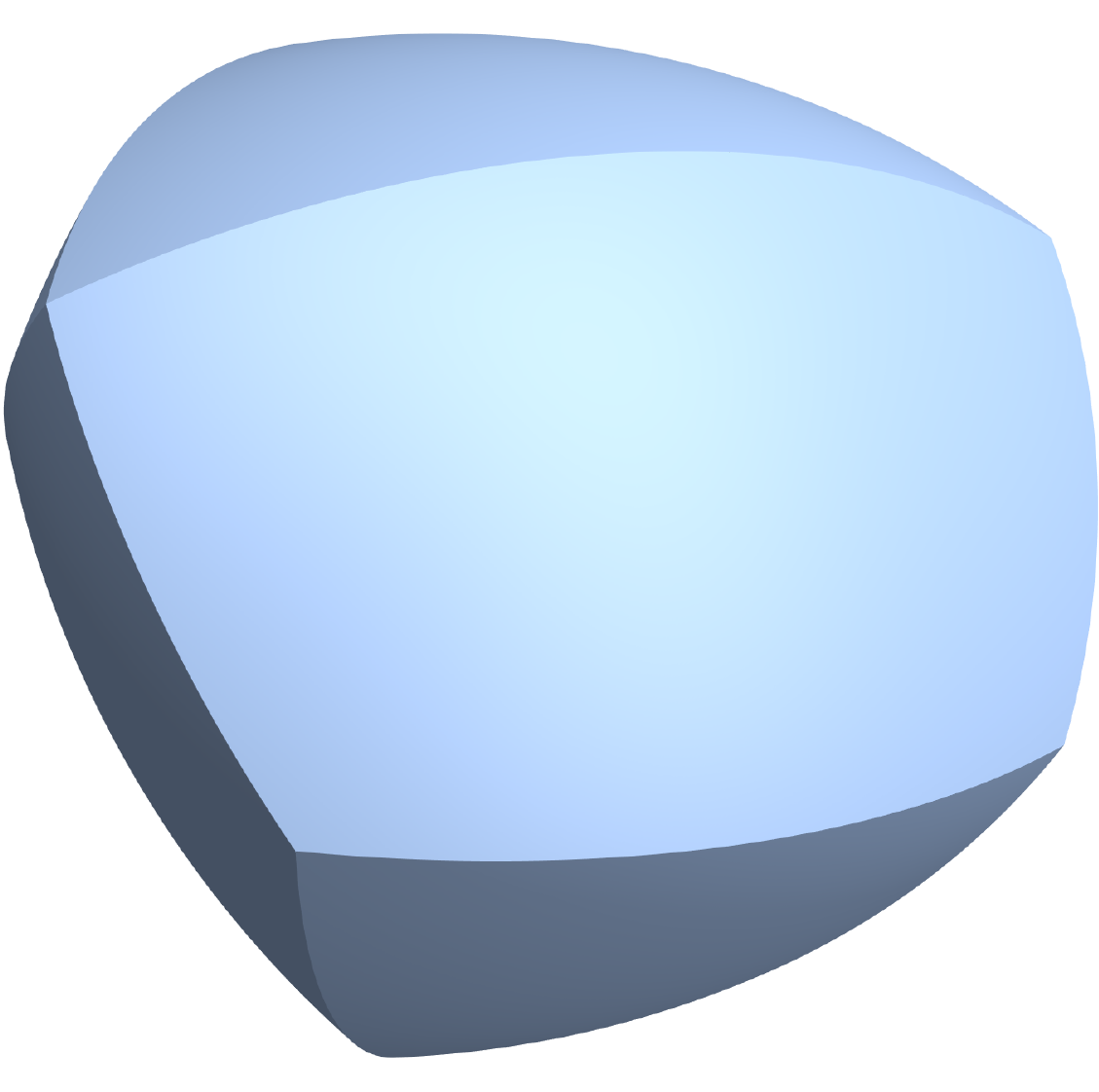} 
\end{figure}
\end{ex}

\newpage 
\begin{ex} The elongated triangular pyramid example above can be extended to a pyramid with any odd number of vertices at least equal to three.  For instance, we can extend a pentagonal pyramid with vertices $\{a_1,\dots, a_6\}$. If $\{a_2,\dots, a_6\}$ form the base of the pyramid, this can be accomplished by choosing five more vertices $\{a_7,\dots, a_{11}\}$ in a plane parallel to the one containing $\{a_2,\dots, a_6\}$ and which is on the opposite side of this plane than $a_1$ is. Moreover,  the apex $a_1$ should be adjusted so that $|a_1-a_j|=1$ for $j=7,\dots, 11$. The Reuleaux polyhedron of the critical set $\{a_1,\dots, a_{11}\}$ along with its skeleton is shown below. 

\begin{figure}[h]
\centering
 \includegraphics[width=.33\textwidth]{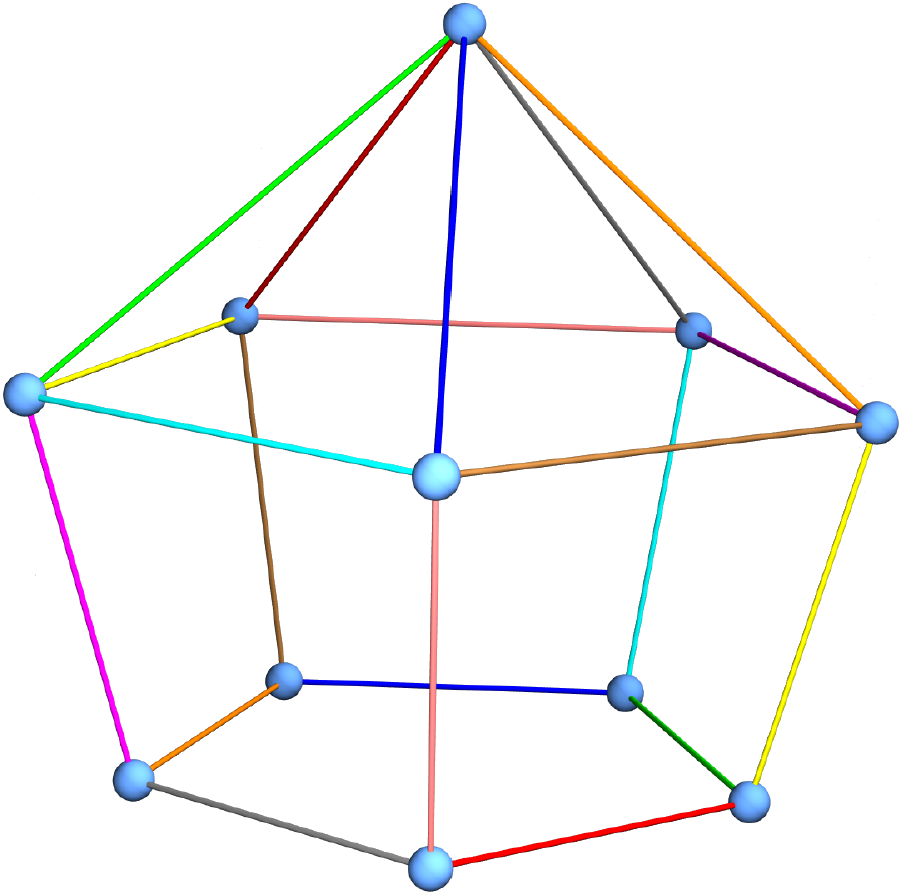}
      \hspace{.3in}
  \includegraphics[width=.4\textwidth]{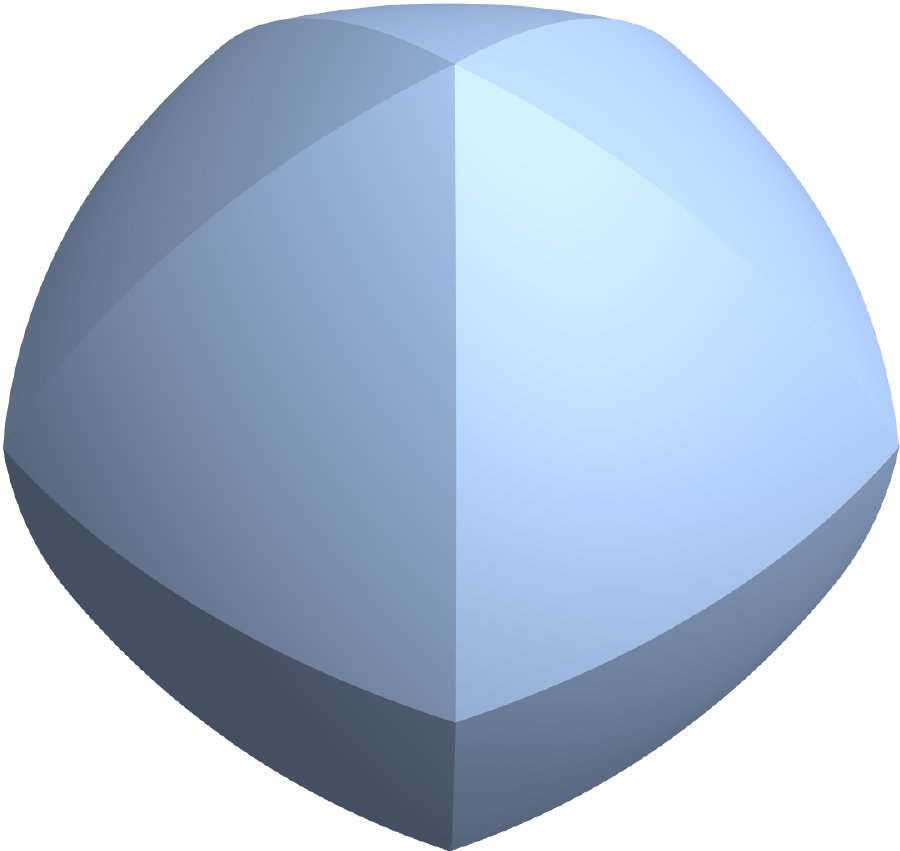}
  
  \vspace{.3in}
  
   \includegraphics[width=.4\textwidth]{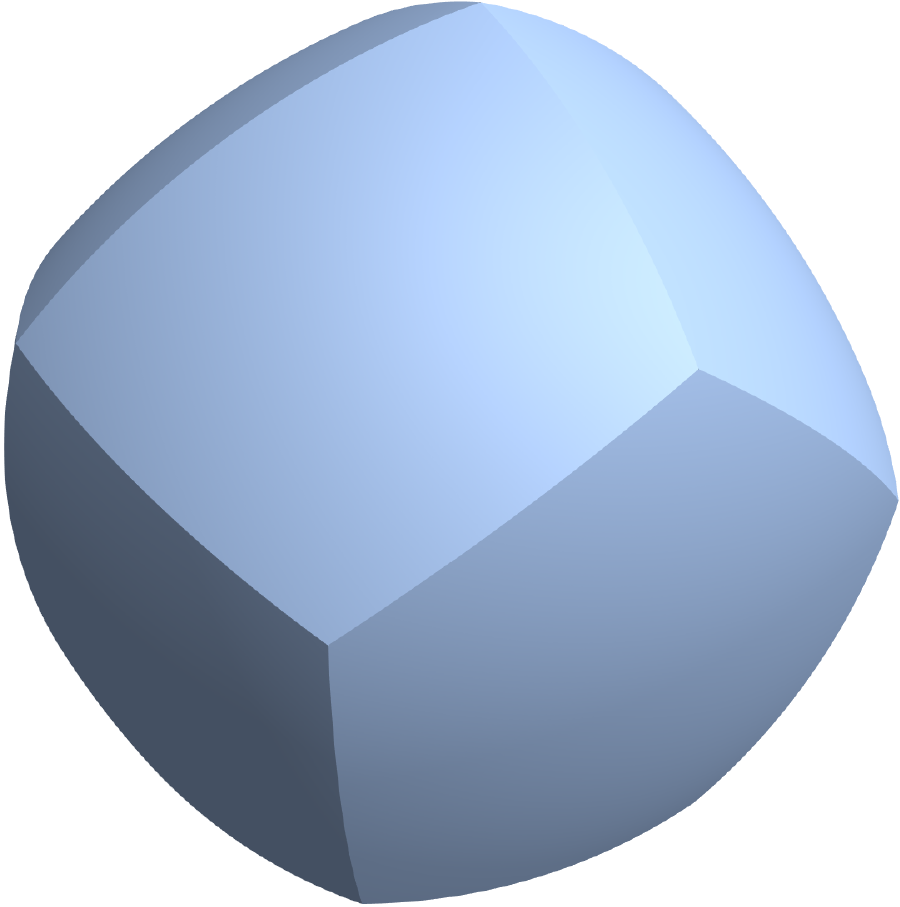}
        \hspace{.3in} 
      \includegraphics[width=.4\textwidth]{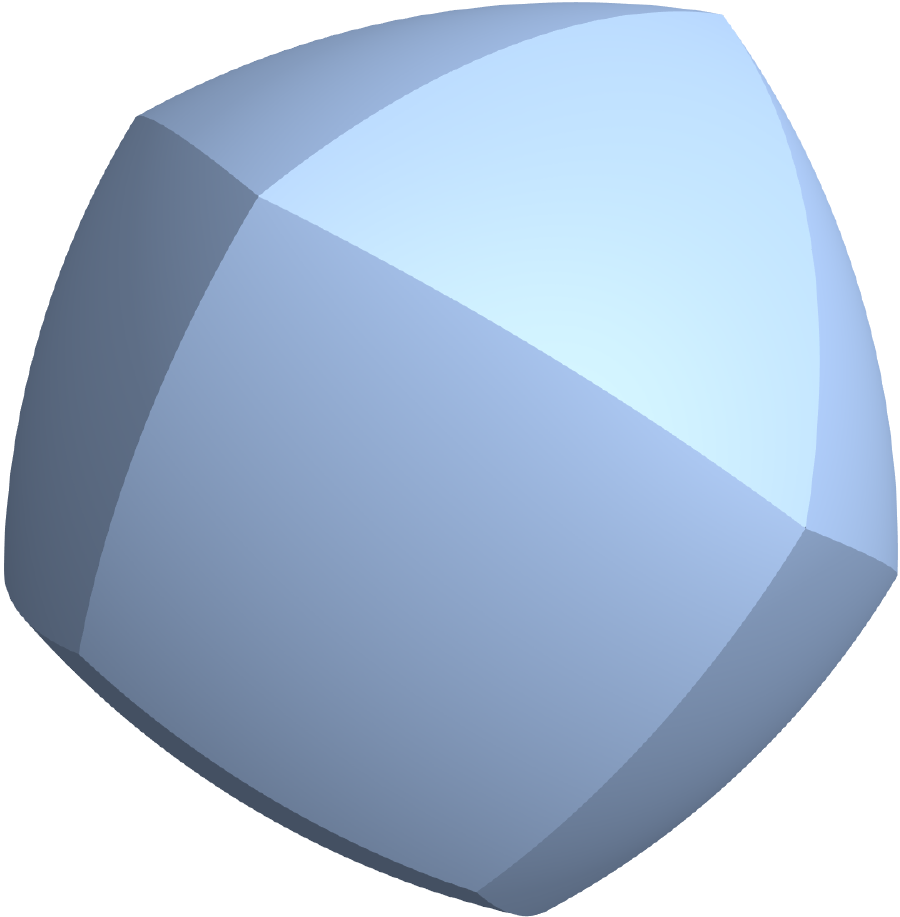} 
\end{figure}
\end{ex}

\newpage 
\begin{ex}\label{DTexample}
A critical set with nine points, which are also the vertices of a diminished  trapezohedron with a square base, is
$$
\begin{cases}
a_1=(0,0,(1-2s^2)^{1/2})\\
a_2=(0,1/2,b)\\
a_3=(1/2,0,b)\\
a_4=(-1/2,0,b)\\
a_5=(0,-1/2,b)
\end{cases}
\quad 
\begin{array}{l}
a_6=(s,s,0)\\
a_7=(-s,s,0)\\
a_8=(-s,-s,0)\\
a_9=(s,-s,0).
  \end{array}
$$
Here $s=1/(2\sqrt{2})$ and $b=(1-s^2-(s+1/2)^2)^{1/2}$.
\begin{figure}[h]
\centering
 \includegraphics[width=.35\textwidth]{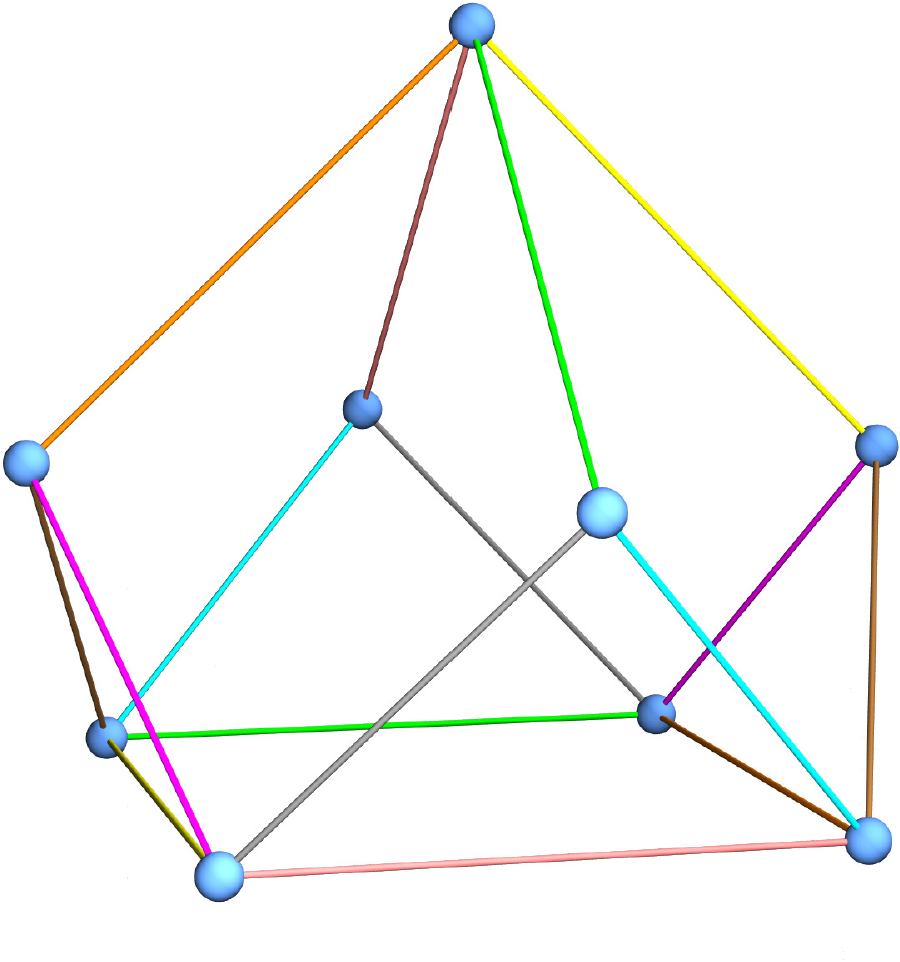}
      \hspace{.3in}
  \includegraphics[width=.4\textwidth]{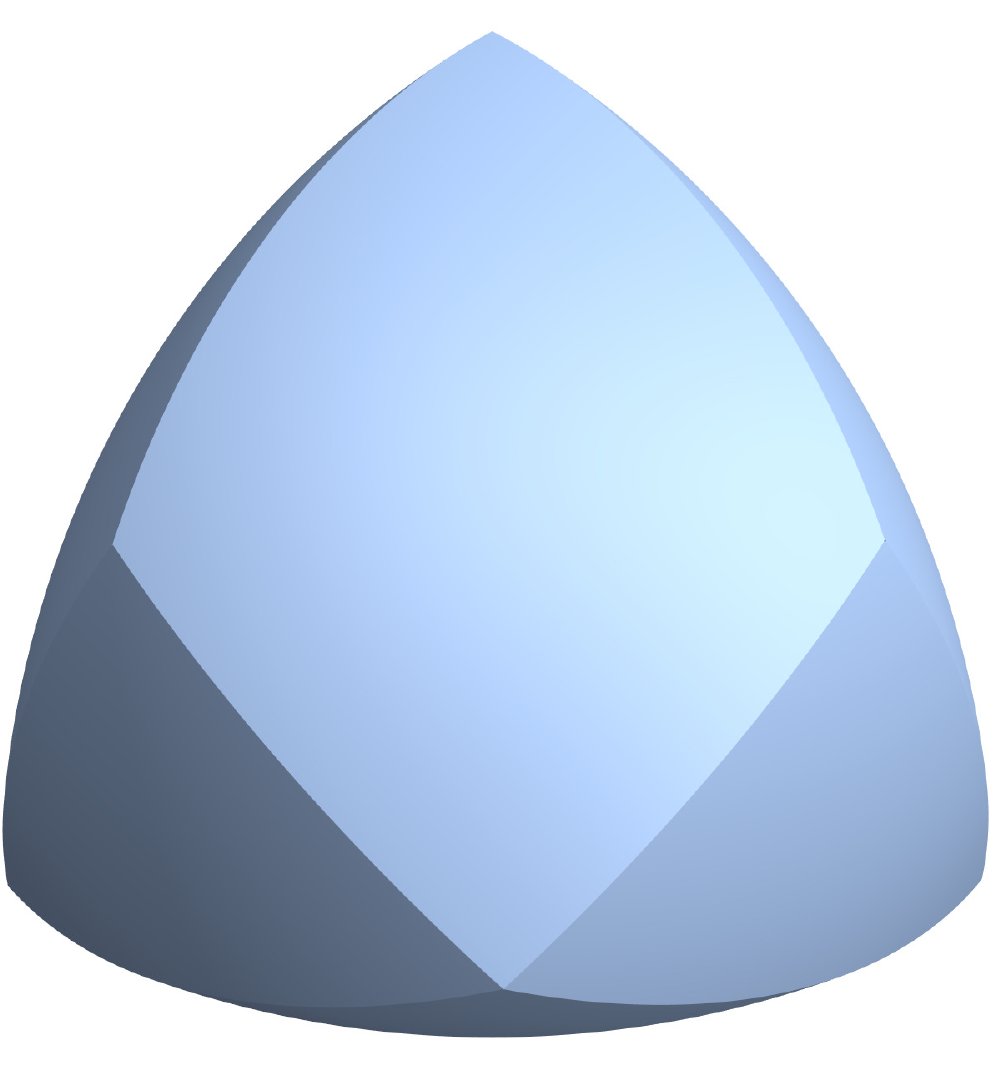}
  
  \vspace{.3in}
  
   \includegraphics[width=.4\textwidth]{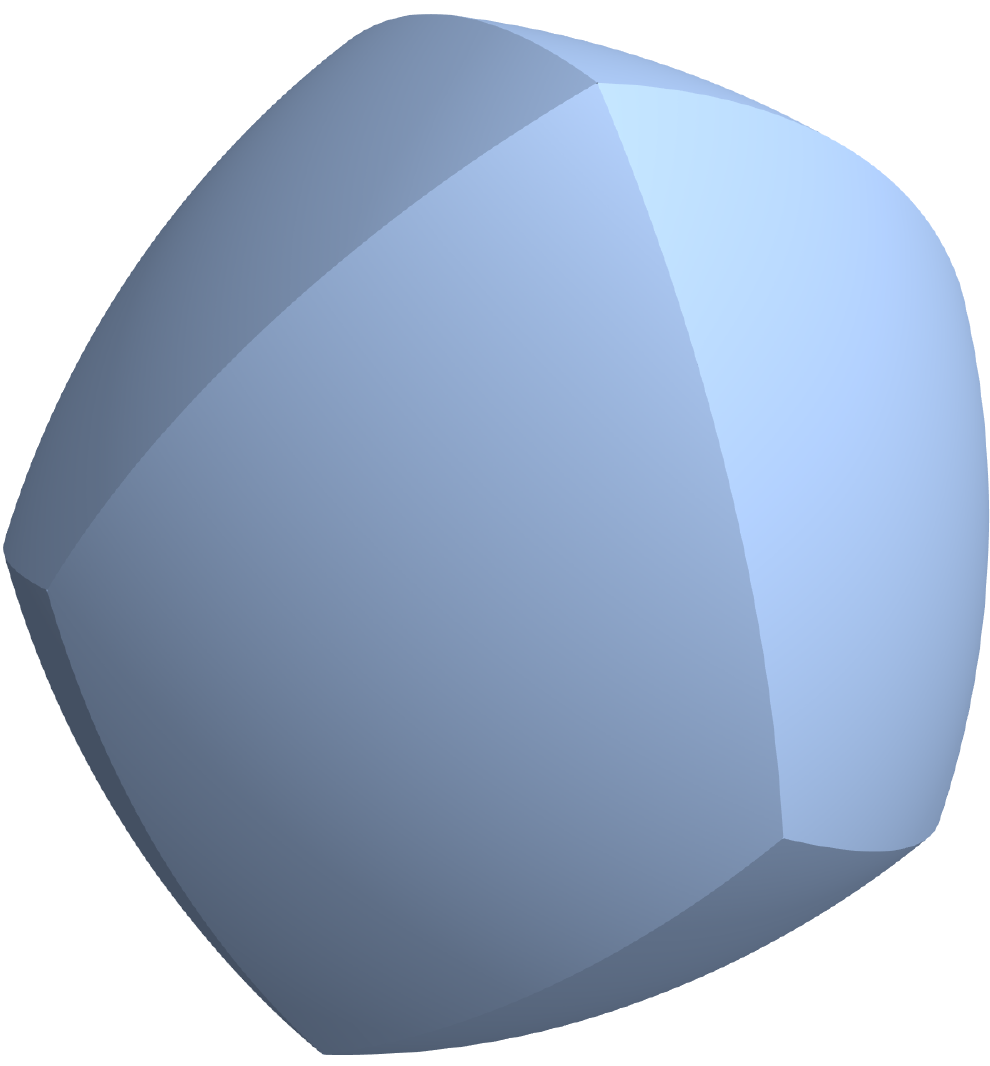} 
        \hspace{.3in}
      \includegraphics[width=.4\textwidth]{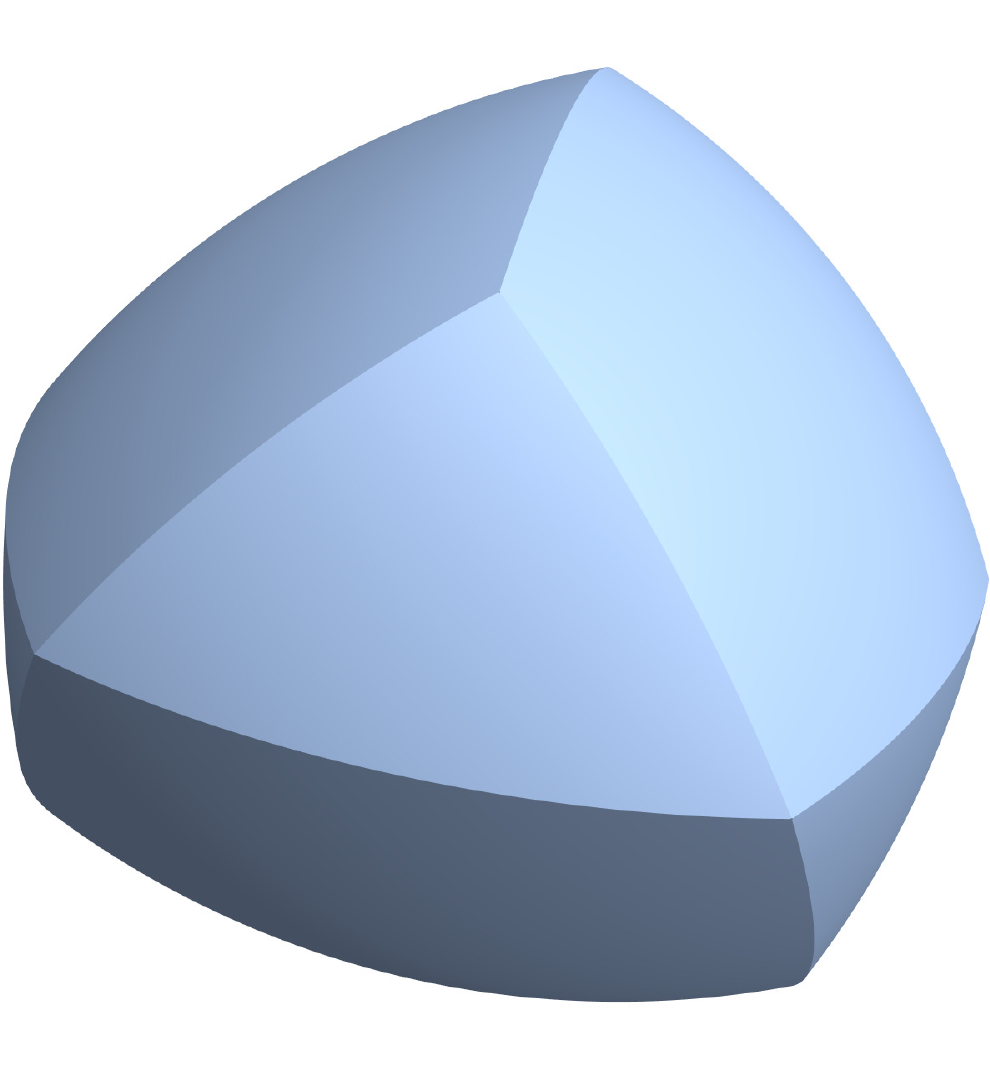} 
\end{figure}
\end{ex}

\begin{ex}\label{DTexample}
More generally, it is possible to design a Reuleaux polyhedron with the same vertices as a diminished trapezohedron whose base 
consists of the vertices of regular polygon with an even number of sides.  The previous example is an instance of such 
a shape with a square base. Arguing similarly as we did for the previous example, we can find an extremal set which is
the vertices of a diminished trapezohedron $\{a_1,\dots, a_{13}\}$ with a hexagonal base. Its corresponding Reuleaux polyhedron and
skeleton are shown below.  

\begin{figure}[h]
\centering
 \includegraphics[width=.35\textwidth]{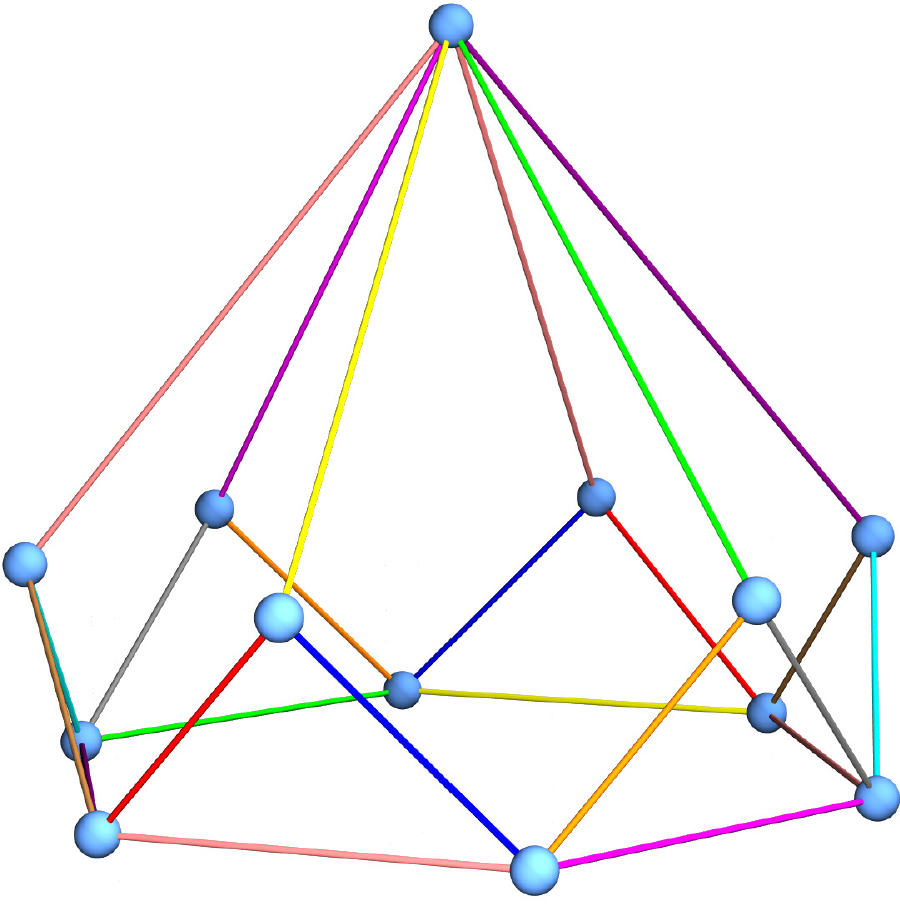}
      \hspace{.3in}
  \includegraphics[width=.4\textwidth]{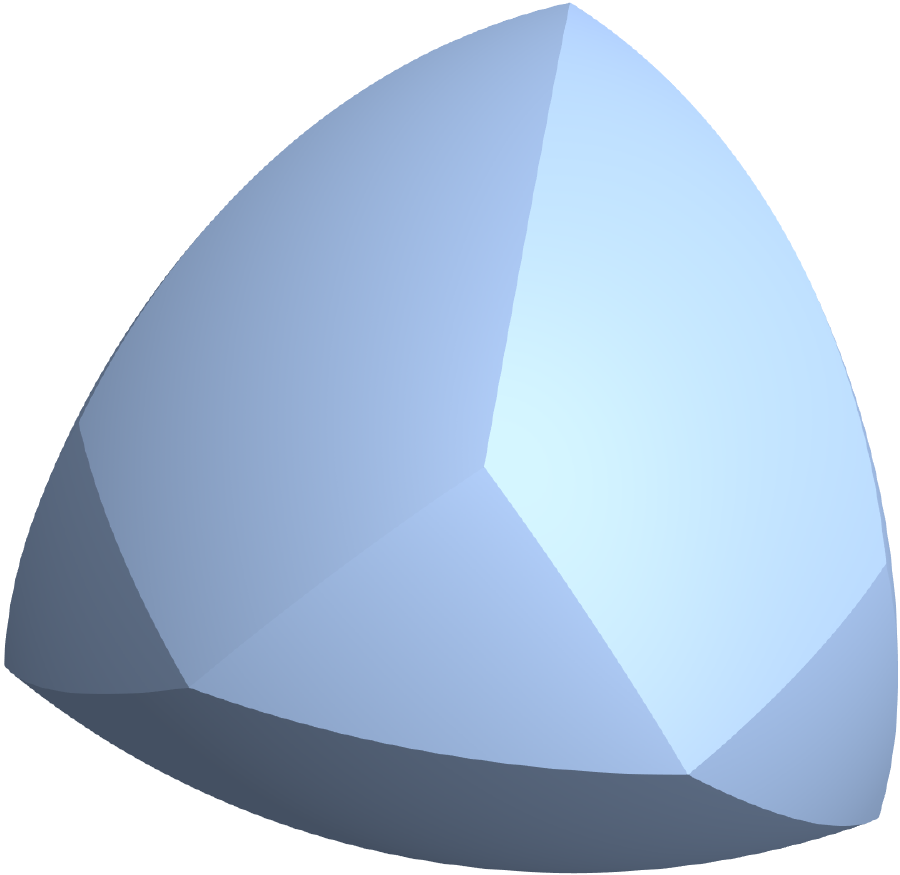}
  
  \vspace{.3in}
  
   \includegraphics[width=.4\textwidth]{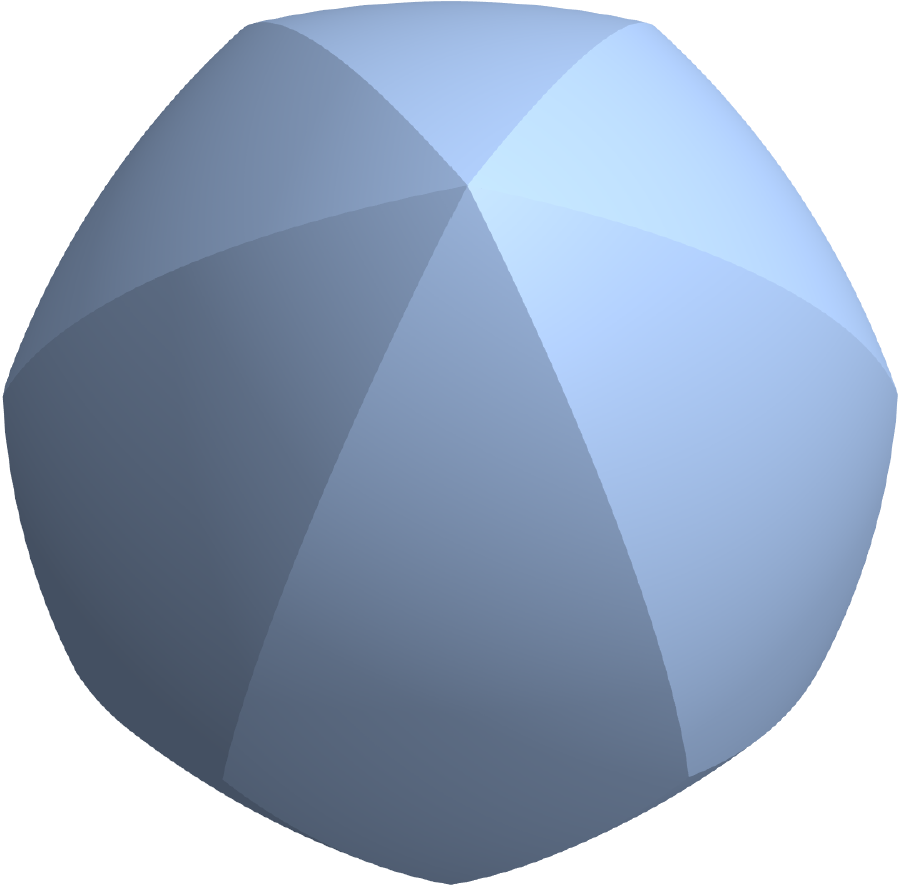} 
        \hspace{.3in}
      \includegraphics[width=.4\textwidth]{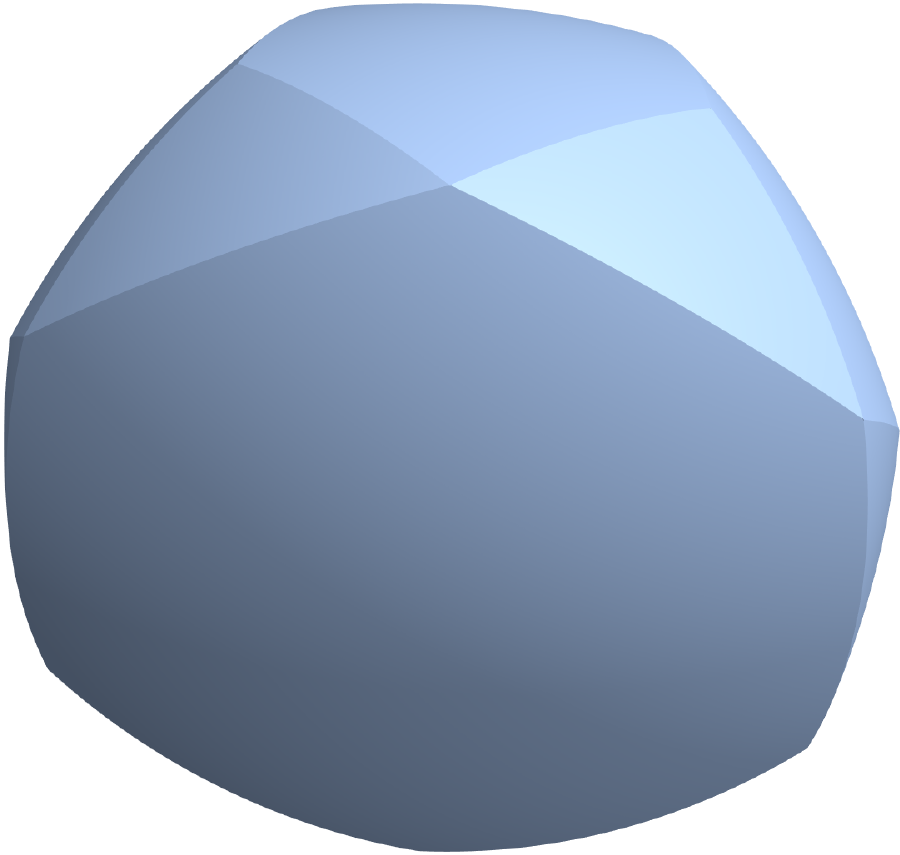} 
\end{figure}
\end{ex}

\newpage 
\begin{ex}\label{BullExample}
The collection of points 
$$
\begin{cases}
a_1=(0,0,0)\\
a_2=(1/\sqrt{2},0)\\
a_3=(\sqrt{7/8},\sqrt{1/8},0)\\
a_4=(\sqrt{1/8},\sqrt{7/8},0)\\
a_5=(0,1/\sqrt{2},0)
\end{cases}
\quad \begin{array}{l}
a_6=(\sqrt{2}/(1+\sqrt{7}),(1/3)\sqrt{1+2\sqrt{7}})\\
a_7=\left(1/(2\sqrt{2}),(4-\sqrt{7})/(2\sqrt{2}),\sqrt{-2+\sqrt{7}}\right)\\
a_8=\left((4-\sqrt{7})/(2\sqrt{2}),1/(2\sqrt{2}),\sqrt{-2+\sqrt{7}}\right)
\end{array}
$$
has diameter one and is critical. This example can be seen as a variant of Example \ref{PentPyrEx} since $\{a_1,\dots, a_5\}$ are the vertices of a pentagon with diameter one which is not regular. 
\begin{figure}[h]
\centering
 \includegraphics[width=.34\textwidth]{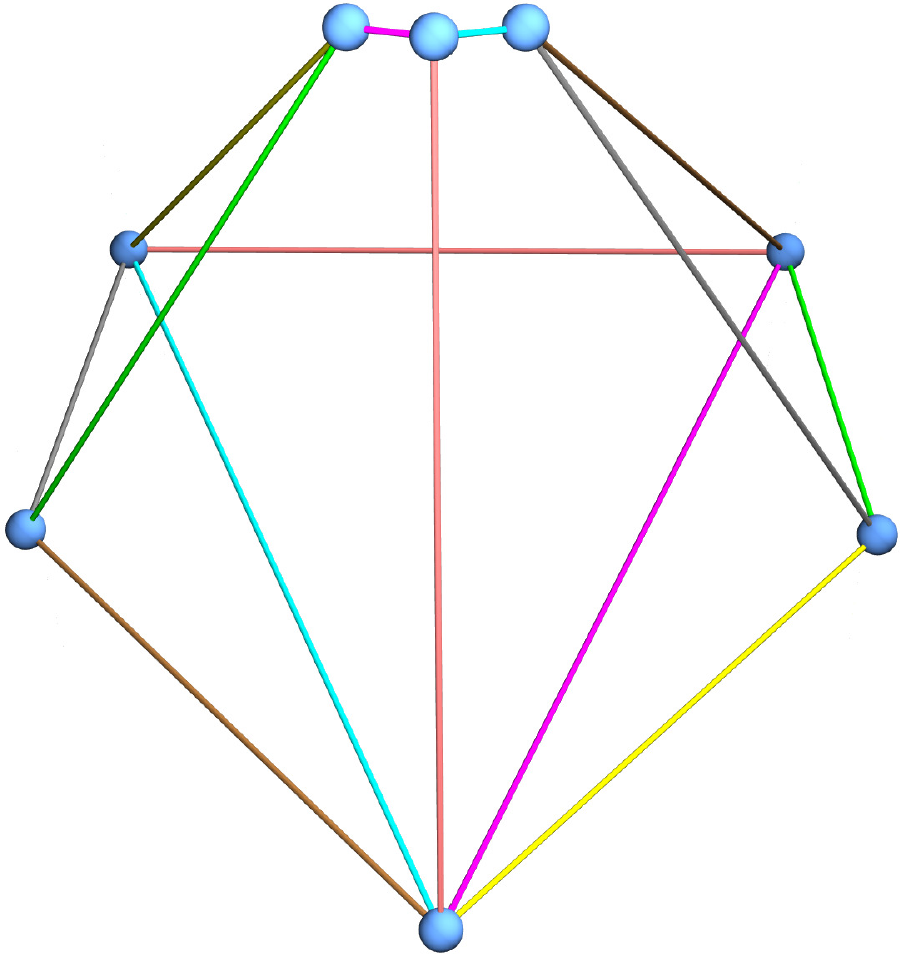}
      \hspace{.3in}
  \includegraphics[width=.4\textwidth]{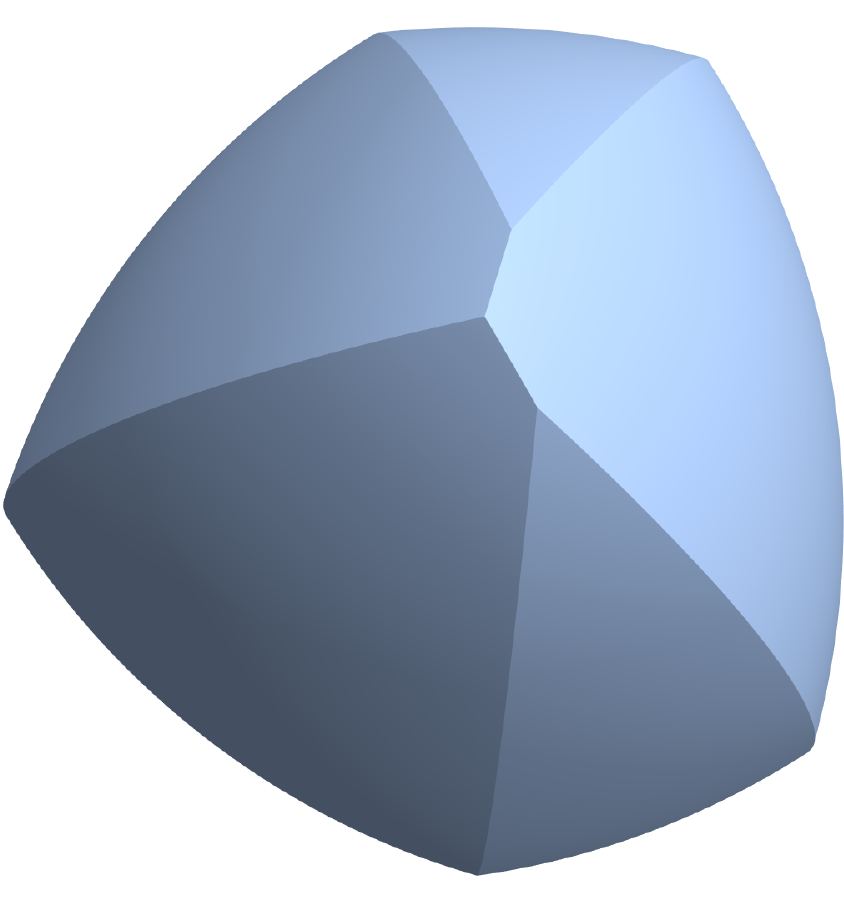}
    \vspace{.2in}
   \includegraphics[width=.4\textwidth]{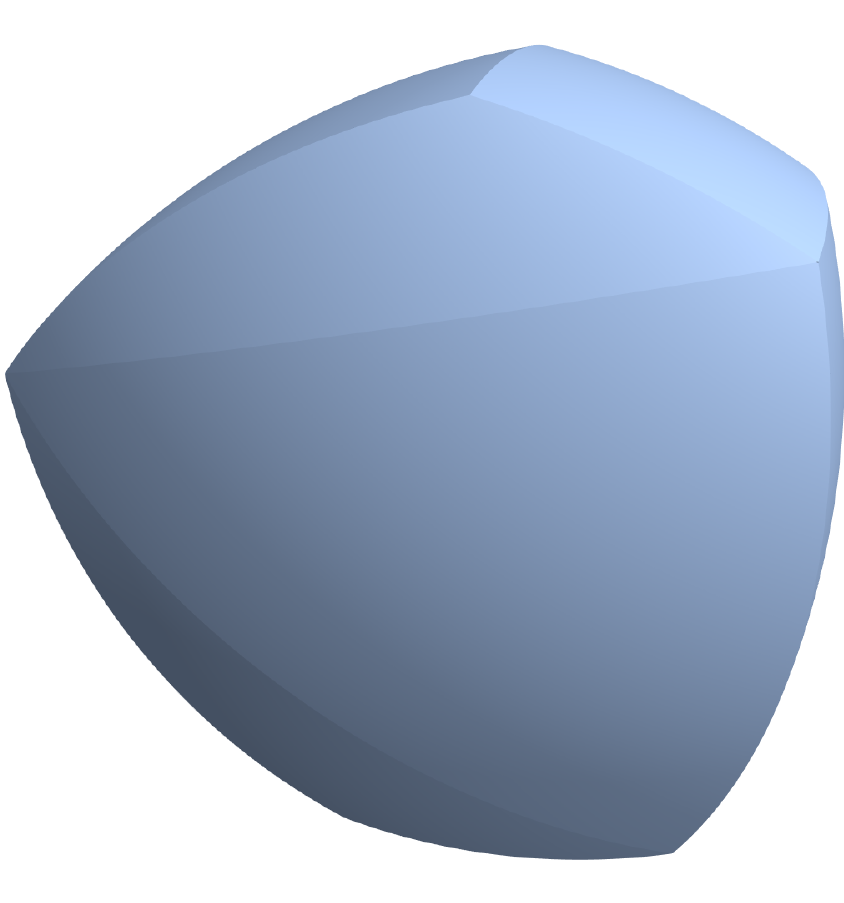} 
        \hspace{.3in}
      \includegraphics[width=.4\textwidth]{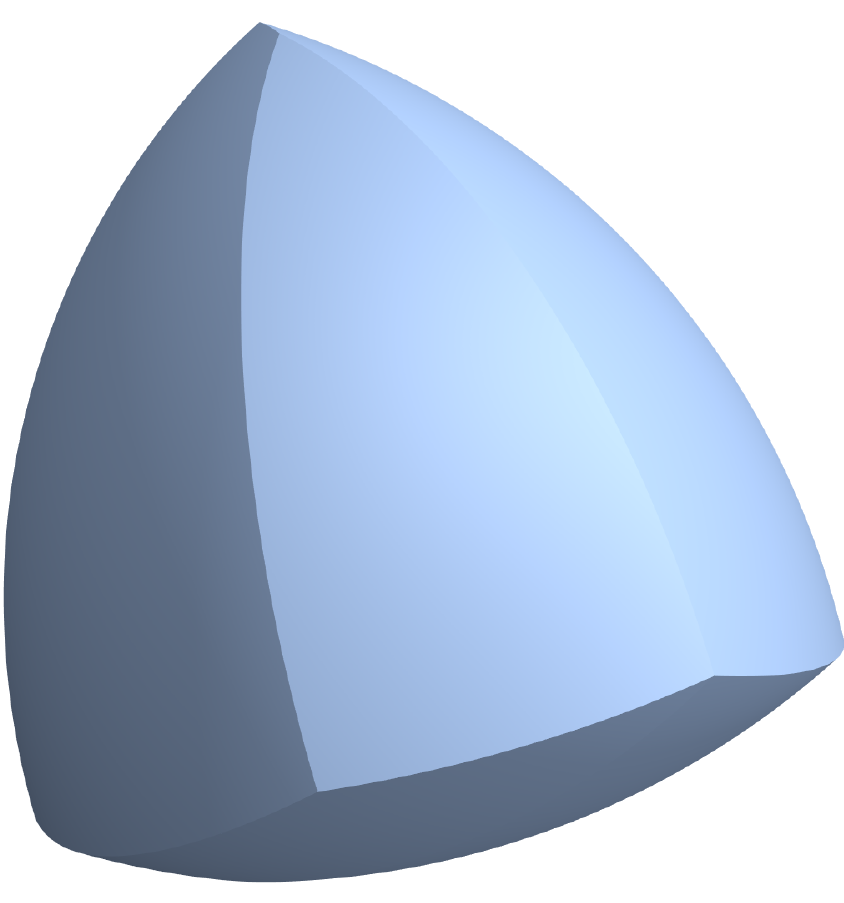} 
\end{figure}
\end{ex}

\newpage

\begin{ex}\label{BullExample}
Montejano and Rold\'an-Pensado (Theorem 5.2 of \cite{MR3620844}) verified that if \\ $\{a_1,\dots, a_n\}\subset\{x\in \R^3: x_3=0\}$ are the vertices of a Reuleaux polygon, and  $\{a_{n+1},\dots, a_m\}$ are the principal vertices of $B( \{a_1,\dots, a_n\})$ in the half-space $\{x\in \R^3: x_3>0\}$, then $\{a_1,\dots, a_m\}$ is extremal.  The previous example is an instance of this result; in that example, $\{a_1,\dots, a_5\}$ are the vertices of a Reuleaux pentagon. 
Another example based on a Reuleaux septagon $\{a_1,\dots, a_7\}$ which has twelve vertices in total is displayed below. 

\begin{figure}[h]
\centering
 \includegraphics[width=.34\textwidth]{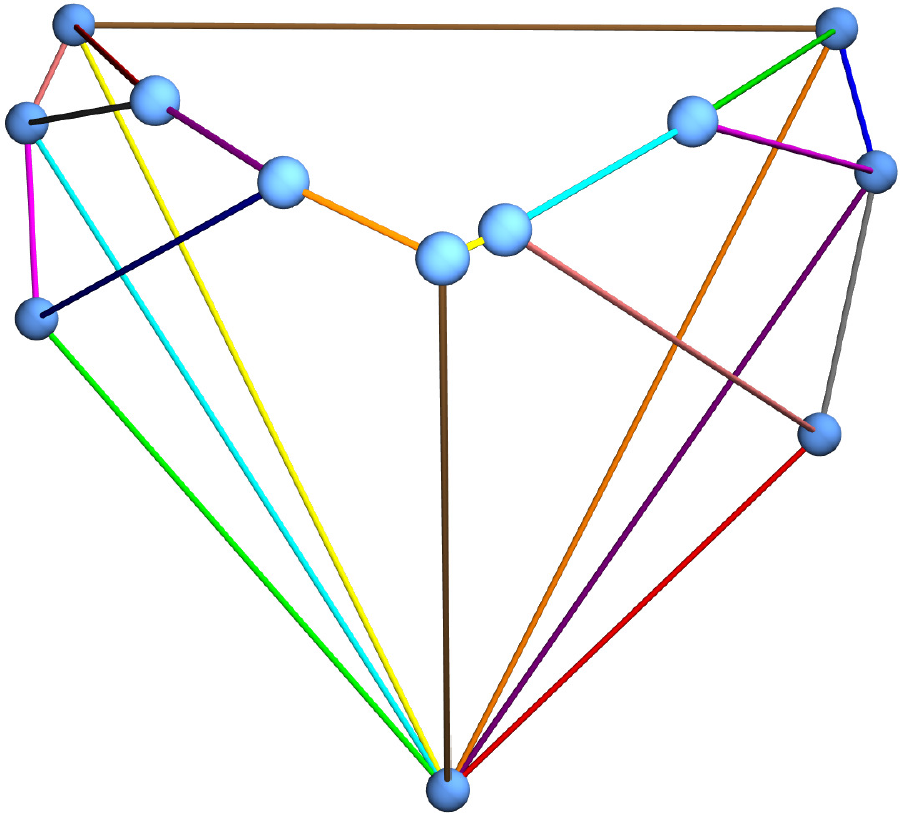}
      \hspace{.3in}
  \includegraphics[width=.4\textwidth]{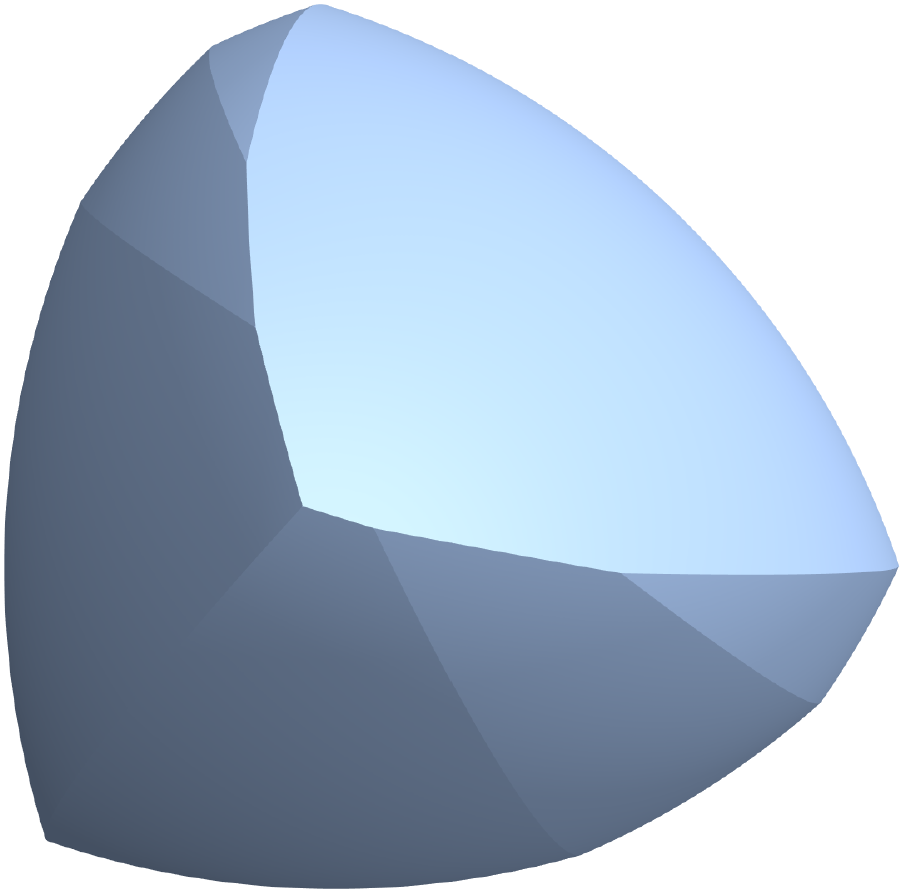}
    \vspace{.2in}
   \includegraphics[width=.4\textwidth]{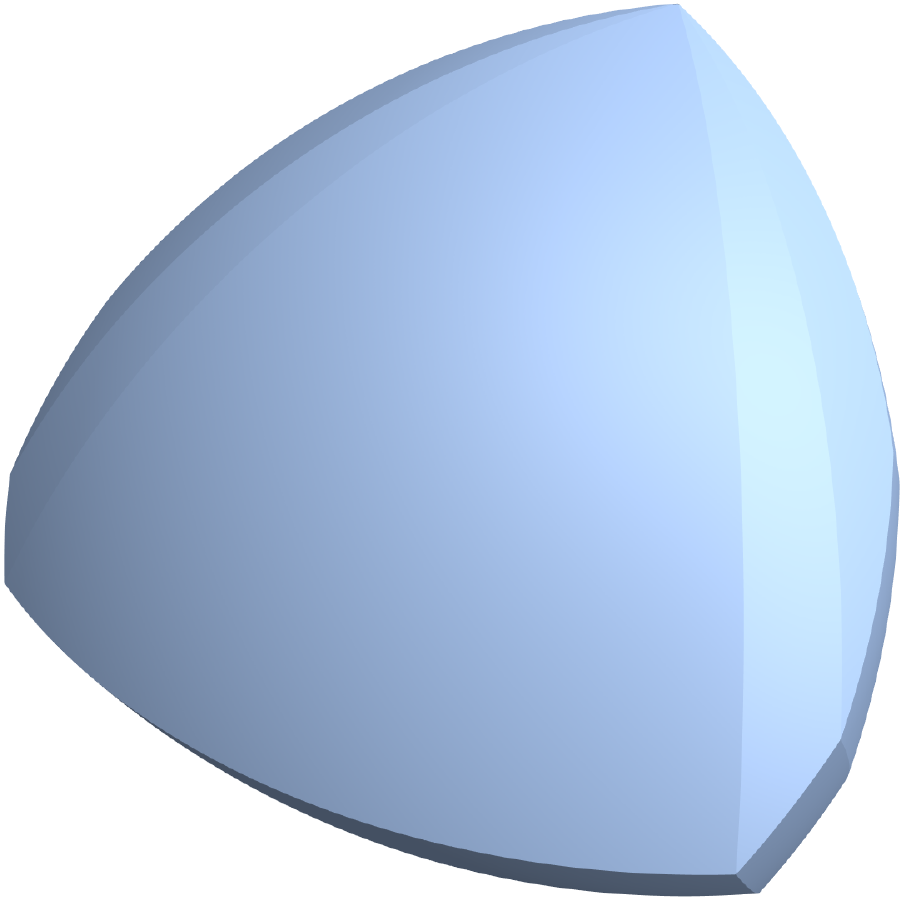} 
        \hspace{.3in}
      \includegraphics[width=.4\textwidth]{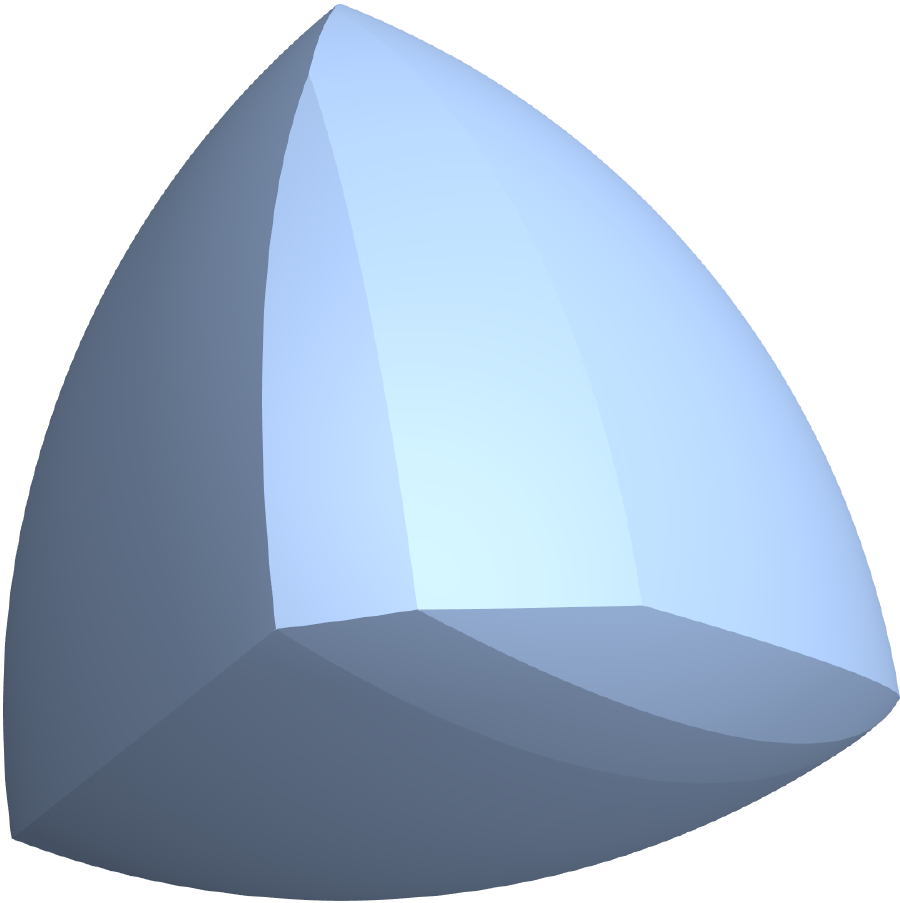} 
\end{figure}
\end{ex}

\newpage

\section{Meissner polyhedra}\label{MeissSect}
Let $X\subset \R^3$ be a finite set of $m\ge 4$ points with diameter one and suppose $X$ is extremal. Then $B(X)$ has $m-1$ dual edge pairs
$$
(e_1,e_1'),\dots, (e_{m-1},e_{m-1}').
$$ 
A convex body of the form
$$
B(X\cup e_1\cup \dots \cup e_{m-1})
$$
is a {\it Meissner polyhedron} based on $X$.   If $X=\{a_1,a_2,a_3,a_4\}$ is a set of vertices of a regular tetrahedron with side length one, we shall see that the associated Meissner polyhedra are the Meissner tetrahedra discussed in the introduction; recall Figures \ref{Meiss1}, \ref{Meiss2}.  We also refer the reader to  Figures \ref{AlmostMeiss}, \ref{MeissPent}, \ref{MeissETP}, \ref{MeissDT}, \ref{MeissBull} below for other examples of Meissner polyhedra.

Since the Meissner tetrahedra are known to have constant width, it is natural to inquire if the same is true for Meissner polyhedra.  It turns out that this is indeed the case.   We will prove the following theorem after establishing some lemmas and arguing that these shapes are essentially the same ones introduced by Montejano and Rold\'an-Pensado \cite{MR3620844}. 
\begin{thm}\label{ConstantWidthThm}
Each Meissner polyhedron has constant width. 
\end{thm} 
\begin{figure}[h]
\centering
 \includegraphics[width=.42\textwidth]{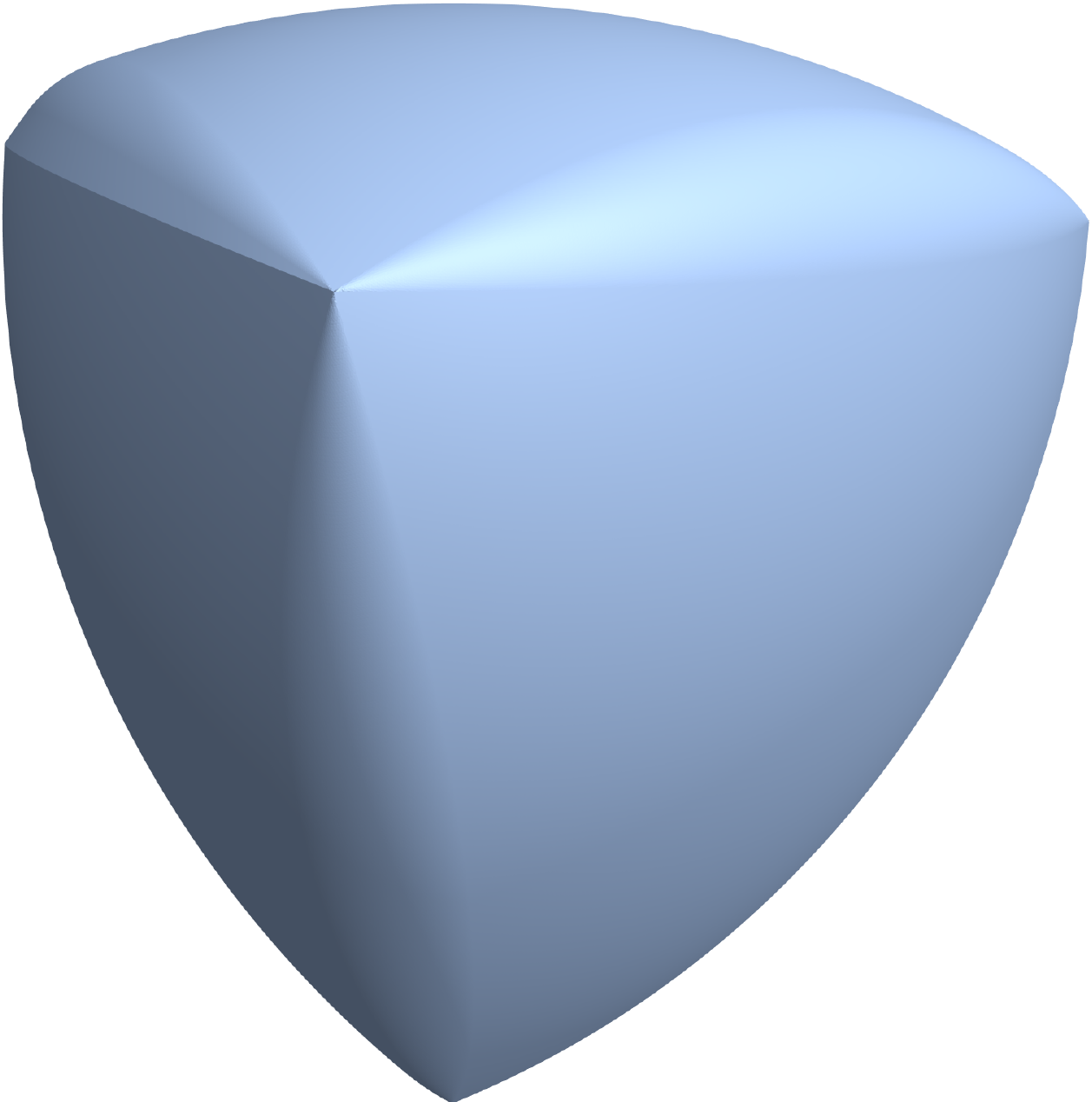}
 \hspace{.2in}
  \includegraphics[width=.42\textwidth]{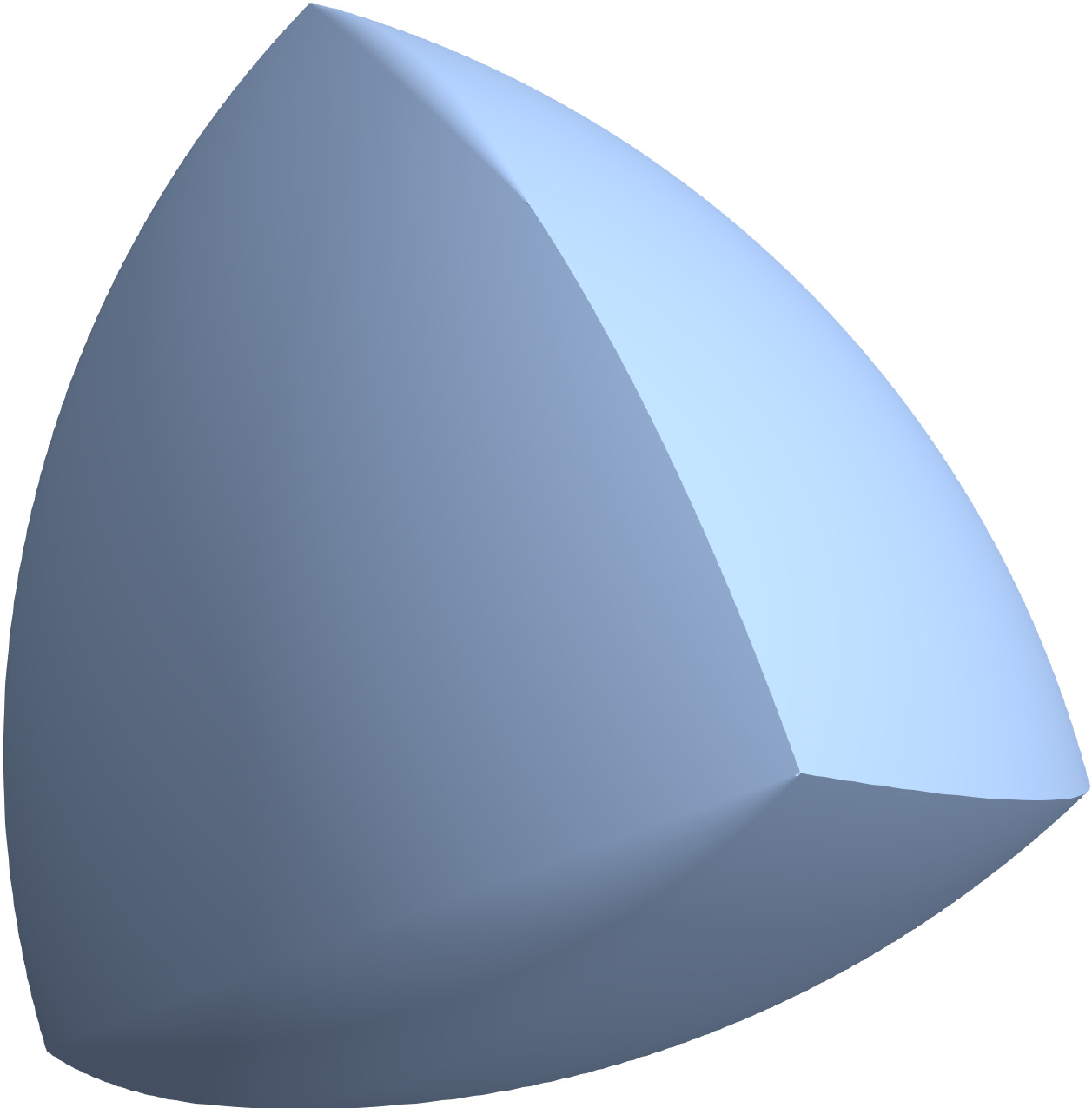}
 \caption{A Meissner polyhedron based on the extremal set $\{a_1,\dots,a_5\}$ from Example \ref{AlmostReulEx}.}\label{AlmostMeiss}
\end{figure}

\subsection{Wedges}
In this subsection, we will fix a subset $X\subset \R^3$ that has diameter one, has $m\ge 4$ points, and is extremal. 
Suppose $(e,e')$ is a dual edge pair for the Reuleaux polyhedron $B(X)$, and denote the endpoints of $e$ as $b,c$ and the endpoints of $e'$ as $b',c'$.  Theorem \ref{edgeThm} implies 
$$
|b-b'|=|b-c'|=|c-b'|=|c-c'|=1.
$$
Since $b',c'$ are distinct, non-antipodal points on the surface of the sphere of radius one centered at $b$, $\{b,b',c'\}$ is affinely independent. Likewise, $\{c,b',c'\}$ is affinely independent. It follows that there are unique half-spaces 
$H,L\subset \R^3$ for which $b,b',c'\in \partial H$ while $c\not \in H$ and $c,b',c'\in \partial L$ while $b\not \in L.$ 
We define 
$$
W_{e'}=H\cap L\cap B(X). 
$$
as the {\it wedge} associated with the edge $e'$ relative to $B(X)$. See Figure \ref{WedgeFig} for an illustration.  We will also use the notation 
$$
\text{int}W_{e'}=(H\cap L)^0\cap B(X)
$$
for the interior of $H\cap L$ relative to $B(X)$. 
\begin{figure}[h]
\centering
 \includegraphics[width=.49\textwidth]{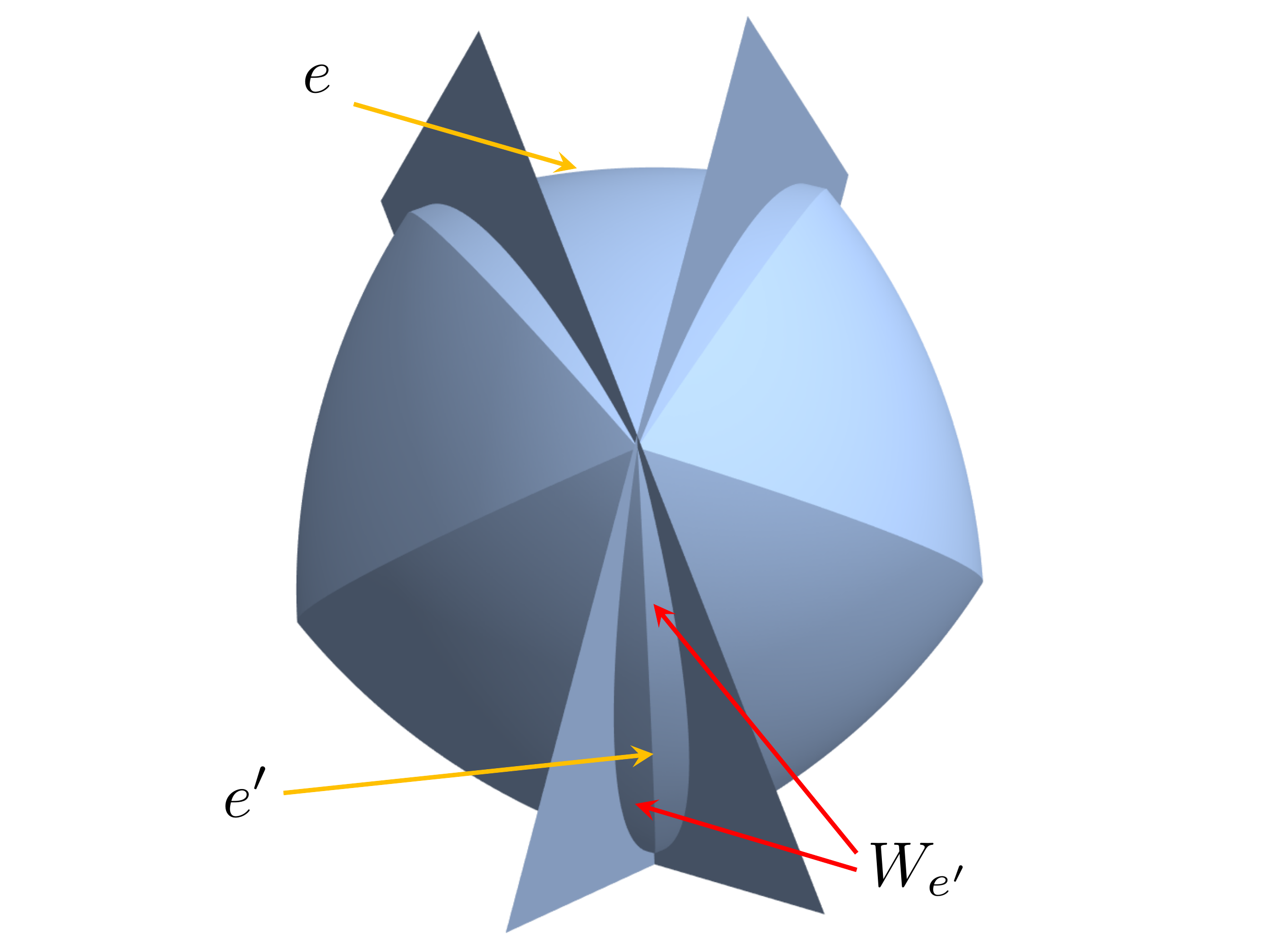}
  \includegraphics[width=.49\textwidth]{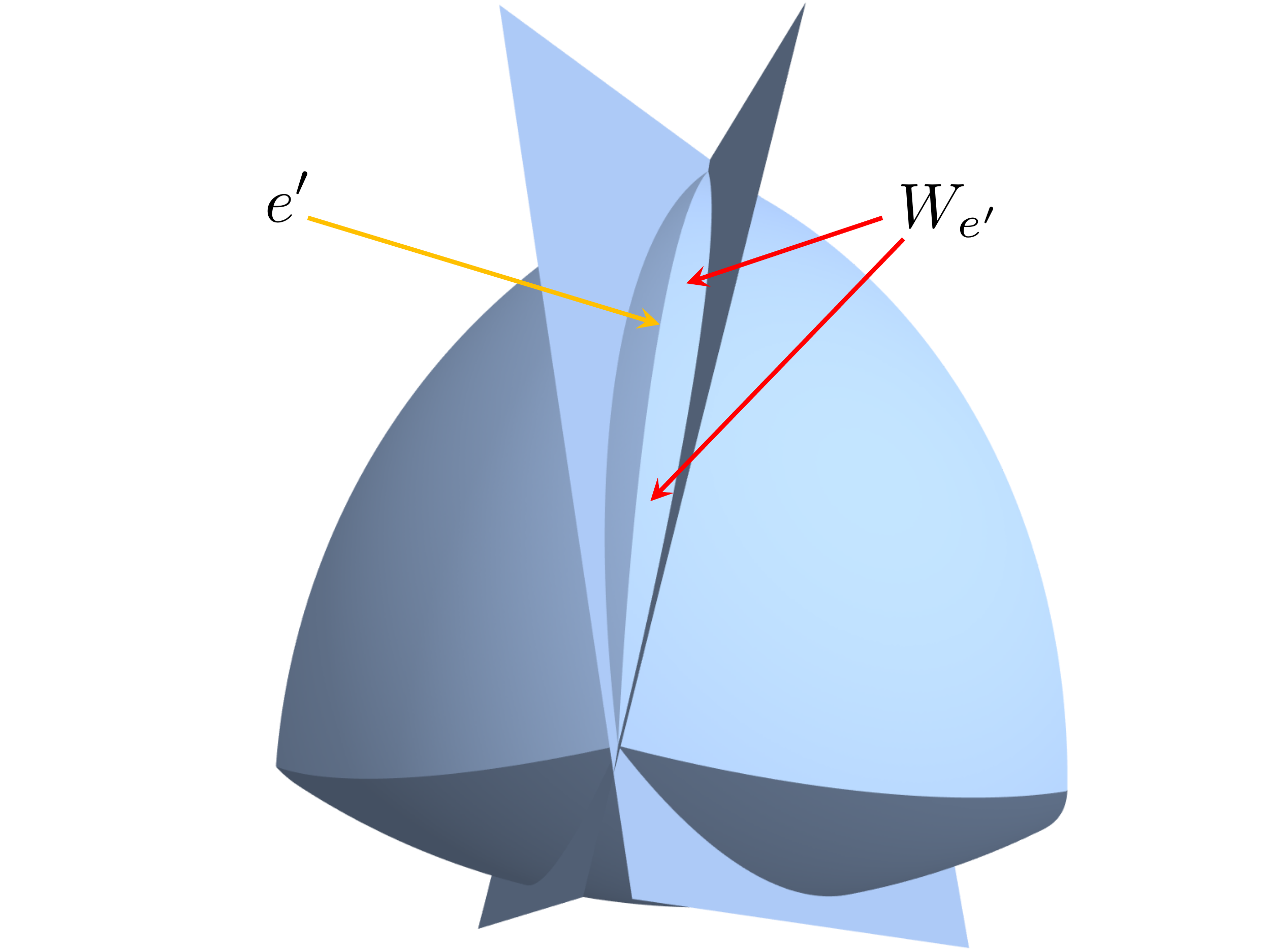}
 \caption{Two views of a Reuleaux polyhedron  $B(X)$ with dual edges $e,e'$ indicated along with the wedge $W_{e'}$. Note in particular, that $W_{e'}$ is a portion of $B(X)$ which is wedged between the two planes.}\label{WedgeFig}
\end{figure}
\par The subsequent lemma implies that if $y$ belongs to an edge $e$ of a Reuleaux polyhedron $B(X)$ and $z\in B(X)$ is at a distance from $y$ larger than one, then $z$ belongs to the interior of the wedge associated with the dual edge $e'$ relative to $B(X)$.  We also note that a more general statement is Proposition 6.1 of \cite{MR296813}. 
\begin{lem}\label{WedgeLemma}
Suppose $(e,e')$ is a dual edge pair of $B(X)$.  If $z\in B(X)$ and $z\not\in B(e)$, then $z\in \textup{int}W_{e'}$. 
Furthermore, $z$ is not a vertex of $B(X)$. 
\end{lem}
\begin{proof}
1. Let us assume the endpoints of $e$ are $b,c$ and the endpoints of $e'$ are $b',c'$. We will use $w=(w_1,w_2,w_3)$ as coordinates for $\R^3$. By changing variables if necessary, we may suppose that $b'$ and $c'$ lie on the $w_3$ coordinate axis and $b$ and $c$ lie in the $w_1w_2$ plane. That is, 
\be\label{specialcoordinatefourtuple}
\begin{cases}
b=\sqrt{1-a^2}e_1\\
c=\sqrt{1-a^2}(\cos(\phi),\sin(\phi),0)\\
b'=ae_3\\
c'=-ae_3
\end{cases}
\ee
for some $a\in (0,1/2]$ and $\phi$ with
\be\label{ThetaPhi}
0< \phi\le\cos^{-1}\left(1-\frac{1}{2(1-a^2)}\right).
\ee
See Figure \ref{FourPointsFig}.  It is also easy to verify that $\phi<\pi/2$ and that in these coordinates the subspaces $H$ and $L$ in the definition of $W_e'$ are given by 
\be\label{HLSpecificCoordinates}
\begin{cases}
H=\{w\in \R^3: w_2\le 0\}\\
L=\{w\in \R^3: -w_1\sin(\phi)+w_2\cos(\phi)\ge 0\}.
\end{cases}
\ee
\begin{figure}[h]
\centering
 \includegraphics[width=.49\textwidth]{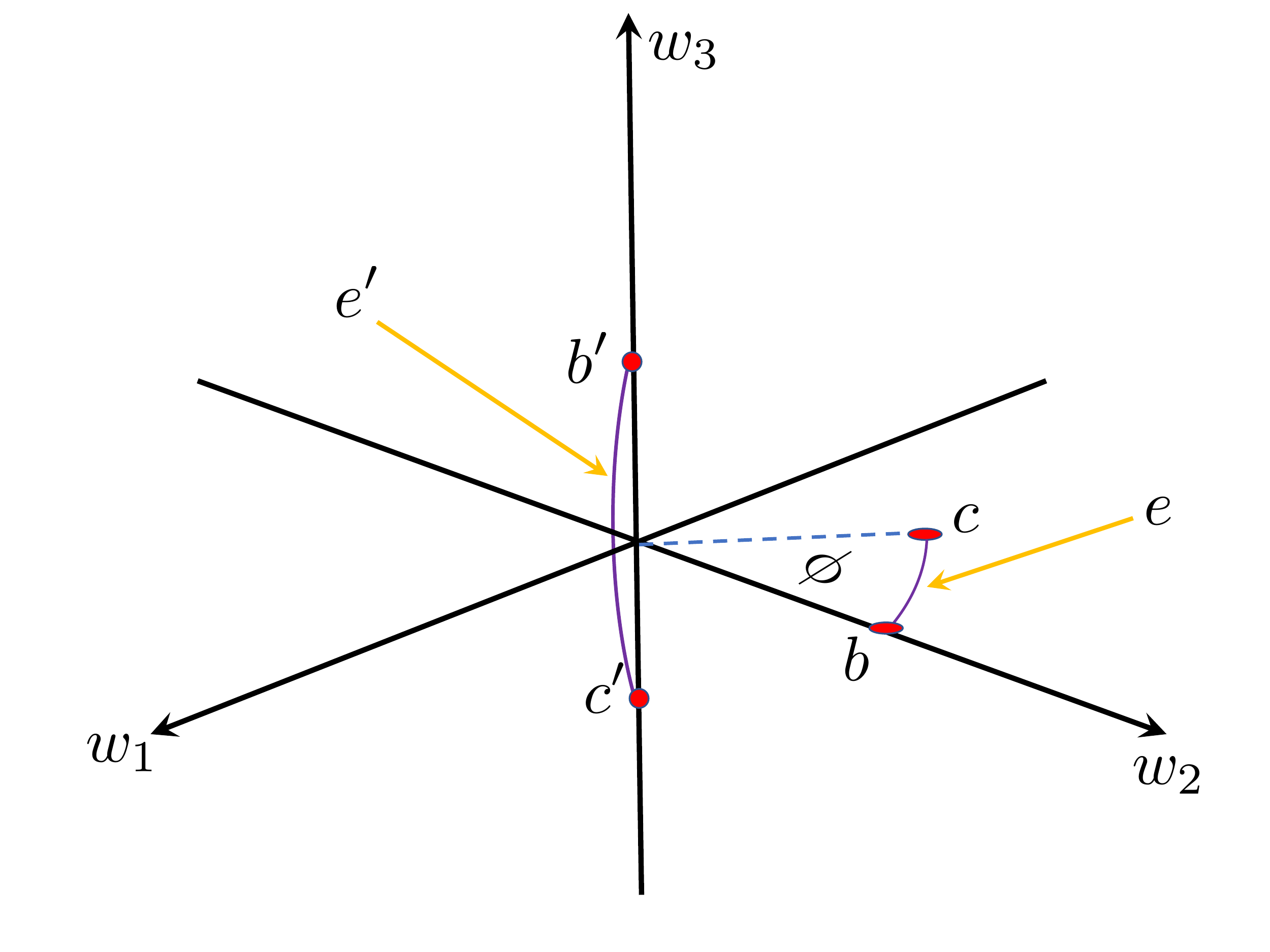}
  \includegraphics[width=.49\textwidth]{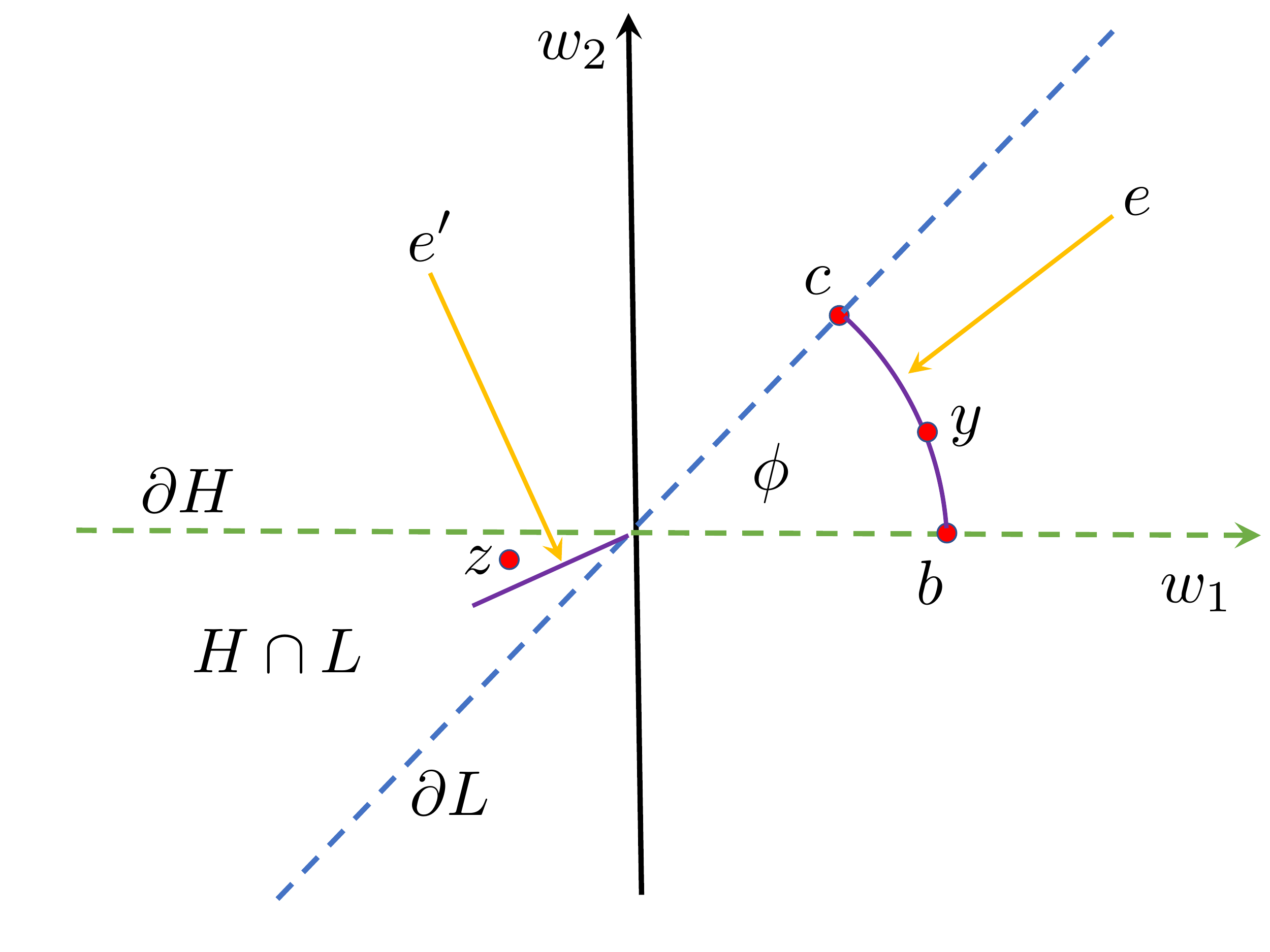}
 \caption{These figures give some perspective on our choice of coordinates \eqref{specialcoordinatefourtuple}. On the left, we have edge $e$ lying in the $w_1w_2$ plane with endpoints $b$ and $c$; we also have edge $e'$ with endpoints $b'$ and $c'$ belonging to the $w_3$ axis. The figure on the right shows the same diagram along with $y\in e$ and $z\in B(X)$ with $|z-y|>1$ from the point of view of the positive $w_3$ axis. Lemma \ref{WedgeLemma} asserts that $z$ is an interior point of $H\cap L$ as displayed.}\label{FourPointsFig}
\end{figure}

\par 2. By hypothesis, there is $y$ on the edge between $b$ and $c$ for which $|z-y|>1$. Then $y$ is given by 
\be\label{yForm}
y=\sqrt{1-a^2}(\cos(\theta),\sin(\theta),0)
\ee
for some $\theta\in (0,\phi)$.  The inequalities $|z-y|>1\ge |z-b|$ also give
$$
|z|^2+|y|^2-2y\cdot z>  |z|^2+|b|^2-2b\cdot z.
$$
And since $|y|=|b|$, we find $z\cdot (b-y)>0$. This inequality is equivalent to 
\be\label{zangle1}
-z_1\sin(\theta/2)+z_2\cos(\theta/2)<0.
\ee
Similarly, the inequalities $|z-y|>1\ge |z-c|$ give 
\be\label{zangle2}
-z_1\sin((\phi+\theta)/2)+z_2\cos((\phi+\theta)/2)>0.
\ee
Moreover, 
{\small\begin{align}
0&<-z_1\sin((\phi+\theta)/2)+z_2\cos((\phi+\theta)/2)\\
&=-z_1\left[\sin(\phi/2)\cos(\theta/2)+\sin(\theta/2)\cos(\phi/2)\right] +z_2\left[\cos(\phi/2)\cos(\theta/2)-\sin(\theta/2)\sin(\phi/2)\right]\\
&=\left[-z_1\sin(\theta/2)+z_2\cos(\theta/2)\right]\cos(\phi/2)-\left[z_1\cos(\theta/2)+z_2\sin(\theta/2)\right]\sin(\phi/2)\\
&<-\left[z_1\cos(\theta/2)+z_2\sin(\theta/2)\right]\sin(\phi/2).
\end{align}}
It follows that 
\be\label{zangle3}
z_1\cos(\theta/2)+z_2\sin(\theta/2)<0.
\ee

\par 3.  Multiplying \eqref{zangle3} by $\sin(\theta/2)$ and  \eqref{zangle1} by $\cos(\theta/2)$ preserves the inequalities as $0<\theta<\phi<\pi/2$. Adding the resulting inequalities yields $z_2<0$, so $z\in H^0$.  We also find from \eqref{zangle2} that 
$$
0<-z_1\sin((\phi+\theta)/2)+z_2\cos((\phi+\theta)/2)<-z_1\sin((\phi+\theta)/2).
$$
It follows that $z_1<0$.  Moreover, the leftmost inequality above gives
$$
\frac{z_2}{z_1}<\tan((\phi+\theta)/2)<\tan(\phi).
$$
That is, $z\in L^0$.  

\par 4. The only vertices which belong to $W_e'$ are $ ae_3$ and $-ae_3$. As $(z_1,z_2)\neq (0,0)$, $z$ is not a vertex of $B(X)$.  
\end{proof}
\begin{cor}\label{LittleEdgeLemma}
If $e$ is an edge of $B(X)$, then $X\subset B(e)$. 
\end{cor}
\begin{proof}
Suppose $x\in X$. Since $X$ has diameter one, $x\in B(X)$; and as $X$ is extremal, $x$ is a vertex of $B(X)$.  According to the previous lemma, $x\in B(e)$.
\end{proof}
\par Next we claim that if we intersect $B(X)$ with all balls centered on the edge $e$, then the only portion of $B(X)$ which is affected lies in $W_{e'}$.   Furthermore, we will argue below that this intersection is a part of the spindle $\textup{Sp}(b',c')$, where $b',c'$ are the endpoints of $e'$.

\begin{lem}\label{OneReducedEdgeLem}
Suppose $e$ and $e'$ are dual edges of $B(X)$ and the endpoints of $e'$ are  $b',c'\in X$. Then  \be\label{FirstWedgeIdentity}
B(X)\cap B(e)\cap W_{e'}=\textup{Sp}(b',c')\cap W_{e'},
\ee
and
\be\label{SecondWedgeIdentity}
B(X)\cap B(e)\cap \left(W_{e'}\right)^c=B(X)\cap \left(W_{e'}\right)^c.
\ee
\end{lem}
\begin{proof}
We will use the notation $P=B(X)$ and $P'=P\cap B(e)$.  Without loss of generality, we may also assume $b,c,b',c'$ satisfy \eqref{specialcoordinatefourtuple} for some $a\in (0,1/2]$ and 
 $\phi$ which fulfills \eqref{ThetaPhi}; here $b,c$ are the endpoints of $e$. In these coordinates, 
\be\label{specificWedgeFormula}
 W:=W_{e'}=\left\{w\in P: w_2\le 0, -w_1\sin(\phi)+w_2\cos(\phi)\ge 0 \right\}
\ee
and we can also express the edge $e$ as
\be\label{specificEdgeFormula}
e=\left\{\sqrt{1-a^2}(\cos(t),\sin(t),0)\in \R^3: t\in (0,\phi)\right\}.
\ee
In addition, it will be convenient to write 
$$
\text{Sp}:=\text{Sp}(ae_3,-ae_3).
$$
\par  Recall that by Remark \ref{IntersectionOverConvBody}, $\text{Sp}=B(\partial B(ae_3)\cap \partial B(-ae_3))$.  Since $e\subset \partial B(ae_3)\cap \partial B(-ae_3)$, $\text{Sp}\subset B(e)$. Moreover, as $ae_3,-ae_3\in P$ and $P$ is spindle convex, $\text{Sp}\subset P$. It then follows that 
$$
\textup{Sp}\cap W\subset P\cap B(e)\cap W= P'\cap W.
$$
Now suppose $z\in P'\cap W$. If $z_1=z_{2}=0$, then $z$ must be on the line segment between $ae_3$ and $-ae_3$ since $z\in P$ and $ae_3,-ae_3\in \partial P$. As a result,  $z\in \text{Sp}$. Alternatively, if $z_1^2+z_2^2\neq 0$, then   
$$
-\frac{(z_1,z_2,0)}{\sqrt{z_1^2+z_{2}^2}}\sqrt{1-a^2}\in e\cup\{b,c\}
$$
as $z\in W$; here we recall the formulae \eqref{specificWedgeFormula} for $W$ and \eqref{specificEdgeFormula} for $e$.  Moreover,  since $z\in B(e)=B(e\cup\{b,c\})$ by Remark \ref{DenseIsEnough},
$$
1\ge \left|z+\frac{(z_1,z_2,0)}{\sqrt{z_1^2+z_{2}^2}}\sqrt{1-a^2}\right|^2=\left(\sqrt{z_1^2+z_{2}^2}+\sqrt{1-a^2}\right)^2+z_3^2.
$$
Again we conclude $z\in \text{Sp}$. It follows that  $z\in  \text{Sp}\cap W$ and $P'\cap W\subset \textup{Sp}\cap W$. 
Therefore, \eqref{FirstWedgeIdentity} holds.  

\par We now pursue \eqref{SecondWedgeIdentity}.  As $P'\subset P$, 
$$
P'\cap W^c\subset P\cap W^c.
$$ 
It order to establish the reverse inclusion, we must show 
\be\label{LastWedgeInc}
P\cap W^c\subset B(e).
\ee
Suppose $z\in P\cap W^c$. If $z\not\in  B(e)$, 
Lemma \ref{WedgeLemma} implies $z\in W$. This would contradict our assumption that $z\in W^c$, so it must be that $z\in  B(e)$. 
In particular, we conclude \eqref{LastWedgeInc}. 
\end{proof}
\begin{figure}[h]
\centering
 \includegraphics[width=.42\textwidth]{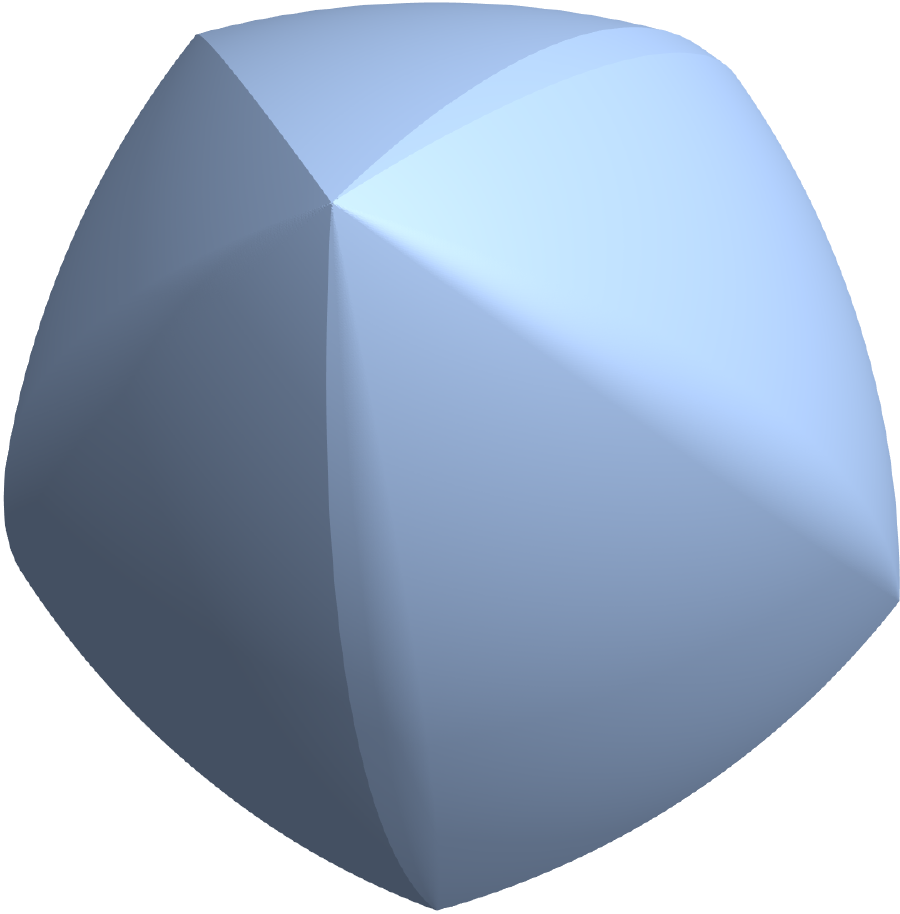}
 \hspace{.2in}
  \includegraphics[width=.42\textwidth]{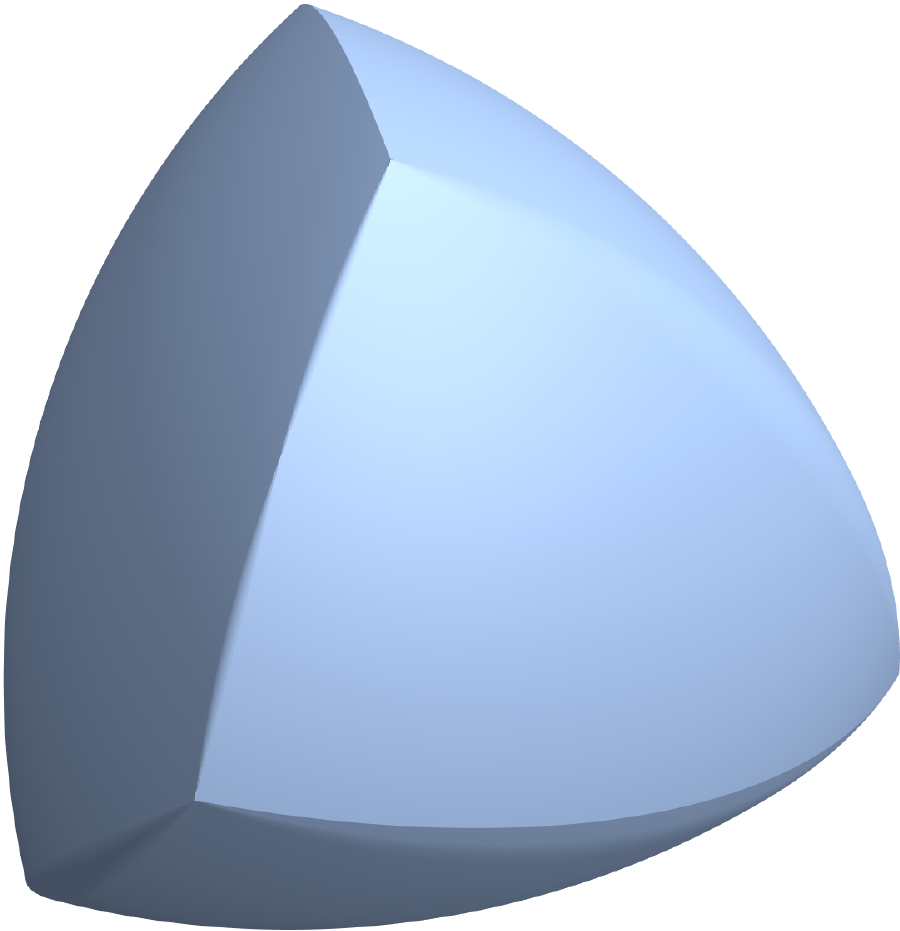}
 \caption{A Meissner polyhedron based on the extremal set $\{a_1,\dots,a_8\}$ from Example \ref{PentPyrExTwo}.}\label{MeissPent}
\end{figure}

\par The following claim asserts that the boundary of a Meissner polyhedron $M$ based on $X$ is realized by performing surgery on $\partial B(X)$ near one edge $e'$ in each dual edge pair $(e,e')$ of $B(X)$.   Here surgery means cutting out the intersection of the wedge associated with the edge $e'$ relative to $B(X)$ and $\partial B(X)$ and replacing this portion of $\partial B(X)$ with the appropriate portion of a spindle torus.  In particular, our Meissner polyhedra include the class of polyhedra defined by Montejano and Rold\'an-Pensado \cite{MR3620844} via this surgery procedure.   Indeed, the extremal sets $X$ considered by Montejano and Rold\'an-Pensado are critical and they generate a Reuleaux polyhedron $B(X)$ whose skeleton is 3--connected.  We saw in Example \ref{nonpolyEx} above that not every $B(X)$ with $X$ extremal has a 3--connected skeleton. 
\begin{prop}\label{structureofMprop}
Suppose $(e_1,e_1'),\dots, (e_{m-1},e_{m-1}')$ are the dual edge pairs of $B(X)$ and $b_j',c_j'$ are the endpoints of $e_j'$ for $j=1,\dots, m=1$.   Then the Meissner polyhedron
$$
M=B(X\cup e_1\cup \dots \cup e_{m-1})
$$
satisfies 
\be\label{structureofM}
\begin{cases}
 M\cap W_{e_j'}= \textup{Sp}(b_j',c_j')\cap W_{e_j'}\quad (j=1,\dots, m-1)
\\\\
 M\cap\left(\displaystyle\bigcup^{m-1}_{j=1}W_{e_j'}\right)^c=  \displaystyle B(X) \cap \left(\bigcup^{m-1}_{j=1}W_{e_j'}\right)^c
\end{cases}
\ee
and 
\be\label{structureofpartialM}
\begin{cases}
 \partial M\cap W_{e_j'}= \partial \textup{Sp}(b_j',c_j')\cap W_{e_j'}\quad (j=1,\dots, m-1)
\\\\
\partial M\cap\left(\displaystyle\bigcup^{m-1}_{j=1}W_{e_j'}\right)^c= \displaystyle \partial B(X) \cap \left(\bigcup^{m-1}_{j=1}W_{e_j'}\right)^c.
\end{cases}
\ee
\end{prop}
\begin{proof}
1. By Lemma \ref{OneReducedEdgeLem}, 
$$
B(X)\cap B(e_1)\cap  W_{e_1'}=\text{Sp}(b_1',c_1')\cap W_{e_1'}.
$$
As $b_1',c_1'$ are vertices of $B(X)$, $b_1',c_1'\in B(e_2)$ by Corollary \ref{LittleEdgeLemma}.  
Therefore, $\text{Sp}(b_1',c_1')\subset B(e_2)$.  It follows that 
$$
B(X)\cap B(e_1)\cap B(e_2)\cap  W_{e_1'}=\text{Sp}(b_1',c_1')\cap W_{e_1'}.
$$
Arguing this way for $j=3,\dots, m-1$, we find 
$$
B(X)\cap B(e_1)\cap \dots \cap B(e_{m-1})  \cap W_{e_1'}=\text{Sp}(b_1',c_1')\cap W_{e_1'}.
$$
That is, $M  \cap W_{e_1'}=\text{Sp}(b_1',c_1')\cap W_{e_1'}.$ Likewise, we conclude $M  \cap W_{e_i'}=\text{Sp}(b_i',c_i')\cap W_{e_i'}$ for $i=2,\dots, m-1$. 

\par 2. In view of \eqref{SecondWedgeIdentity},
$$
B(X)\cap B(e_j)\cap  W_{e_j'}^c=B(X)\cap W_{e_j'}^c.
$$
for $j=1,\dots, m-1$.  
This implies
\begin{align}
B(X)\cap \left(\displaystyle\bigcup^{m-1}_{j=1}W_{e_j'}\right)^c
&=\bigcap^{m-1}_{j=1}B(X)\cap W_{e_j'}^c\\
&=\bigcap^{m-1}_{j=1}B(X)\cap B(e_j)\cap W_{e_j'}^c\\
&=B(X)\cap B(e_1)\cap \dots \cap B(e_{m-1}) \cap \bigcap^{m-1}_{j=1}W_{e_j'}^c\\
&=M\cap \left(\displaystyle\bigcup^{m-1}_{j=1}W_{e_j'}\right)^c.
\end{align}
We conclude \eqref{structureofM}.

\par 3. Set $Q:=\bigcup^{m-1}_{j=1}W_{e_j'}$. Suppose $x\in \partial M \cap Q^c$. Then $x\in B(X)$ by \eqref{structureofM}.  Suppose $B_\delta(x)\subset B(X)$ for some $\delta>0$. Choosing $\delta$ smaller if necessary, $B_\delta(x)\subset Q^c$ since $Q$ is closed. In this case,  $B_\delta(x)\subset B(X)\cap Q^c= M\cap Q^c\subset M$. Then $x$ would not belong to $\partial M$ as we assumed.  As a result, $x\in \partial B(X)$ and 
$\partial M\cap Q^c\subset \partial B(X) \cap Q^c$. Likewise, we conclude $\partial B(X) \cap Q^c\subset \partial M\cap Q^c$. Therefore, $\partial M\cap Q^c=\partial B(X) \cap Q^c$.

\par 4. We may express $W_{e_1'}=B(X)\cap H\cap L$, where $H$ and $L$ are the two half spaces such that $b_1,b_1',c_1'\in \partial H$ while $c_1\not \in H$ and $c_1,b_1',c_1'\in \partial L$ while $b_1\not \in L.$ Here $b_1,c_1$ are the endpoints of $e_1$.  Since $M$ and $\textup{Sp}(b_1',c_1')$ are subsets of $B(X)$,
$$
\partial M\cap W_{e_1'}= \partial \textup{Sp}(b_1',c_1')\cap W_{e_1'}
$$ 
is equivalent to 
 \be\label{FourthWedgeIdentity}
\partial M\cap H\cap L=\partial\textup{Sp}(b_1',c_1') \cap H\cap L.
\ee
Also note that the same argument we used to prove $\partial M\cap Q^c=\partial B(X) \cap Q^c$ gives 
\be\label{FifthWedgeIdentity}
\partial M\cap (H\cap L)^0= \partial\textup{Sp}(b_1',c_1') \cap (H\cap L)^0.
\ee
In order to establish \eqref{FourthWedgeIdentity}, we are then left to verify
\be\label{SixWedgeIdentity}
\partial M\cap \partial (H\cap L)= \partial\textup{Sp}(b_1',c_1')\cap \partial (H\cap L).
\ee
\begin{figure}[h]
\centering
 \includegraphics[width=.7\textwidth]{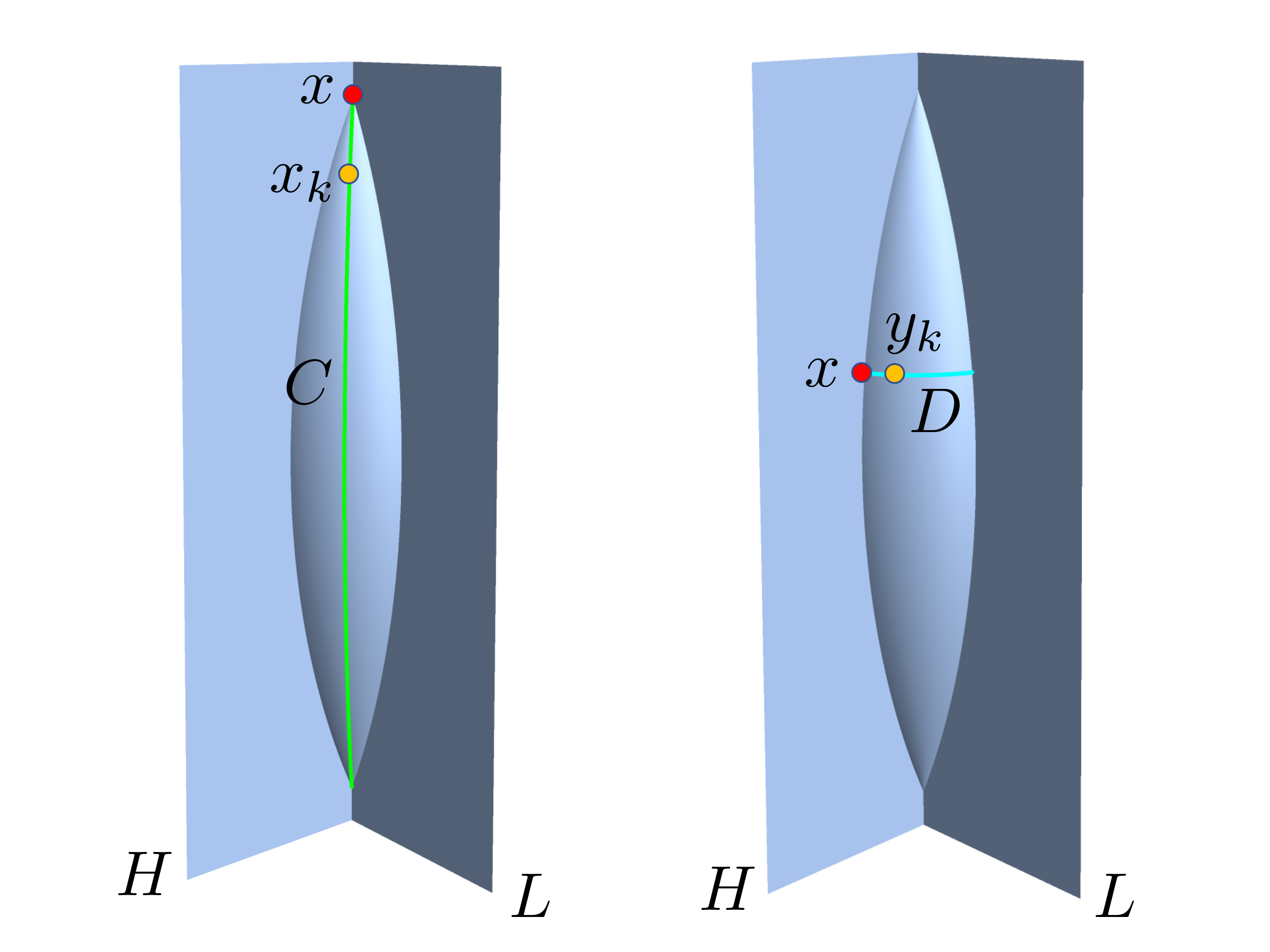}
 \caption{Here we have two images of the intersection $\textup{Sp}(b_1',c_1')\cap \partial (H\cap L)$. This graphic helps explain that if  $x\in  \partial\textup{Sp}(b_1',c_1')\cap \partial (H\cap L)$, there is a sequence in $\partial\textup{Sp}(b_1',c_1')\cap \partial (H\cap L)^0$ which converges to $x$. This is a critical step in our of proof of Proposition \ref{structureofMprop}.  }\label{MeissConstWidthFigFig}
\end{figure}
\par To this end, we suppose $x\in  \partial\textup{Sp}(b_1',c_1')\cap \partial (H\cap L)$. If $x=b'_1$ or $c_1'$, we choose a short arc $C\subset \partial\textup{Sp}(b_1',c_1')$ which joins $b_1'$ and $c_1'$ and for which $C\setminus\{b_1',c_1'\}\subset (H\cap L)^0$. It is not hard to see there is a sequence $x_k\in \partial\textup{Sp}(b_1',c_1')\cap (H\cap L)^0\cap C$ which converges to $x$ as $k\rightarrow\infty$; see Figure \ref{MeissConstWidthFigFig} for example. In view of \eqref{FifthWedgeIdentity}, $x_k\in \partial M$ for all $k\in \N$.  Thus, $x\in \partial M$.  Alternatively, if $x$ belongs to a smooth point on $\partial\textup{Sp}(b_1',c_1')$, we consider the circular arc $D$ in $\partial \textup{Sp}(b_1',c_1')\cap (H\cap L)$ which includes $x$ and belongs to the plane orthogonal to $b_1'-c_1'$. It is evident that there are $y_k\in \partial \textup{Sp}(b_1',c_1')\cap (H\cap L)^0\cap D$ which converges to $x$. By \eqref{FifthWedgeIdentity},  $y_k\in  \partial M$ for all $k\in \N$. Thus, $x\in \partial M$.  We conclude that ``$\supset$"  holds in \eqref{SixWedgeIdentity} in all cases. 

\par Next we suppose $x\in \partial M\cap \partial (H\cap L)$. Then $x\in \textup{Sp}(b_1',c_1')$ by \eqref{structureofM}. If $x$ is an interior point of $\textup{Sp}(b_1',c_1')$, then $B_\delta(x)\subset \textup{Sp}(b_1',c_1')$ for some $\delta>0$. By Corollary \ref{LittleEdgeLemma}, $b_1',c_1'\in M$. As $M$ is spindle convex, we would additionally have $ B_\delta(x)\subset\textup{Sp}(b_1',c_1')\subset M$. Since we assumed $x\in \partial M$, it must be that $x\in \partial \textup{Sp}(b_1',c_1')$. Therefore, ``$\subset$"  holds in \eqref{SixWedgeIdentity}. As a result,  \eqref{SixWedgeIdentity} holds and we conclude \eqref{structureofpartialM}. 
\end{proof}

\begin{figure}[h]
\centering
 \includegraphics[width=.42\textwidth]{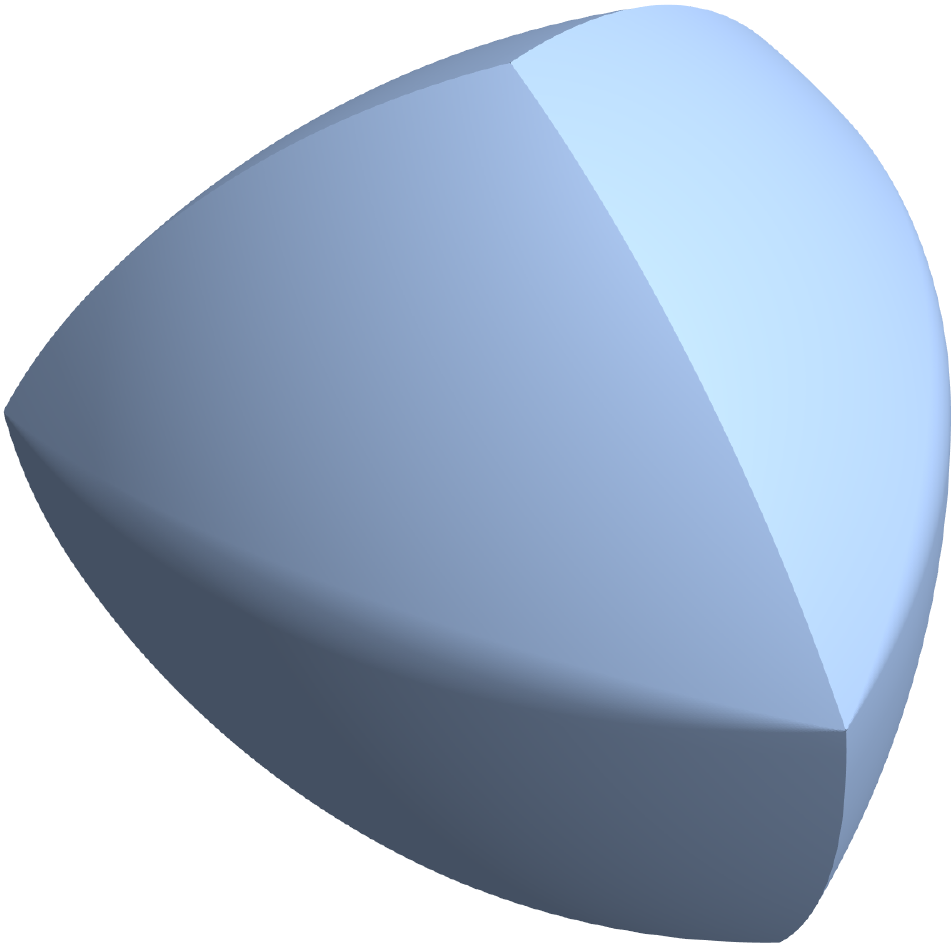}
 \hspace{.2in}
  \includegraphics[width=.42\textwidth]{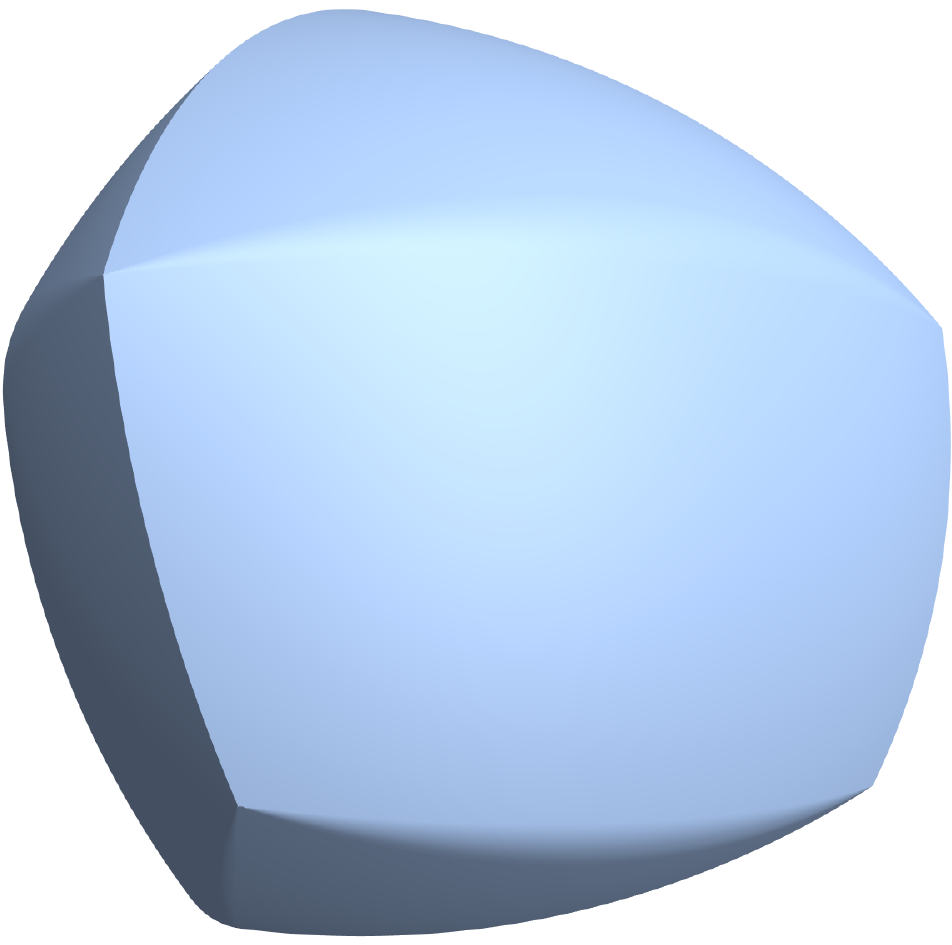}
 \caption{A Meissner polyhedron based on the critical set $\{a_1,\dots,a_7\}$ from Example \ref{ETPexample}.}\label{MeissETP}
\end{figure}

\subsection{Constant width property} 
We now aim to proof Theorem \ref{ConstantWidthThm}. Our first step is to recall a basic fact about convex bodies.

\begin{lem}\label{normalConeLemma}
Suppose $x,y\in K$ with $\textup{diam}(K)=|x-y|.$
Then  $x-y\in N_K(x)$ and $y-x\in N_K(y)$.
\end{lem}
\begin{proof}
Let $w\in K$ and $t\in [0,1]$. Then $y+t(w-y)\in K$ and 
$$
 |x-y|^2\ge  |x-(y+t(w-y)|^2=  |x-y|^2+t^2|w-y|^2-2t(x-y)\cdot (w-y).
$$
This implies $(y-x)\cdot (w-y)\le0$ for all $w\in K$. We conclude $y-x\in N_K(y)$. Likewise, we find $x-y\in N_K(x)$.  
\end{proof}
The following assertion involves the outward unit normal to spindle tori.  Note that the smooth part of the spindle $ \partial\textup{Sp}(ae_3,-ae_3)$ is given by $x\in \R^3$ which satisfies $|x_3|<a$ and 
\be\label{SmoothpartSPeq}
\left(\sqrt{x_1^2+ x_2^2}+\sqrt{1-a^2}\right)^2+x_3^2=1.
\ee
Moreover, the outward unit normal $u$ at such a point $x$ is 
\be\label{spindleNormalForm}
u=\left(\sqrt{x_1^2+x_{2}^2}+\sqrt{1-a^2}\right)\frac{(x_1,x_{2},0)}{\sqrt{x_1^2+x_{2}^2}}+x_3e_3.
\ee
\begin{lem}\label{WedgeLocationLemma}
Suppose $e$ and $e'$ are dual edges of $B(X)$, $b,c$ are the endpoints of $e$, and $b',c'$ are the endpoints of $e'$. Further assume 
$$
x\in \left(\partial \textup{Sp}(b',c')\setminus\{b',c'\}\right)\cap W_{e'}
$$
and $u\in N_{\textup{Sp}(b',c')}(x)$ with $|u|=1$. Then $x-u\in e\cup \{b,c\}$. 
\end{lem}
\begin{proof}
It suffices to establish this claim under the assumption that $b,c,b',c'$ satisfy \eqref{specialcoordinatefourtuple} for some $a\in (0,1/2]$ and $\phi$ which satisfies \eqref{ThetaPhi}. We also recall that in these coordinates, 
$$
e=\left\{\sqrt{1-a^2}(\cos(t),\sin(t),0): t\in (0,\phi)\right\}
$$
and
$$
W_{e'}=\left\{w\in B(X): w_2\le 0, -w_1\sin(\phi)+w_2\cos(\phi)\ge 0 \right\}.
$$
In particular,
\be\label{xWedgeEprimeConstraints}
x_2\le 0, -x_1\sin(\phi)+x_2\cos(\phi)\ge 0
\ee
 as $x\in W_{e'}$.

\par Observe that since $x\in \partial \textup{Sp}(ae_3,-ae_3)\setminus\{ae_3,-ae_3\}$, $x$ belongs to the smooth part of $\textup{Sp}(ae_3,-ae_3)$. In particular, 
$u$ is the unit normal to $\textup{Sp}(ae_3,-ae_3)$ at $x$. In view of \eqref{spindleNormalForm}, 
$$
x-u=-\sqrt{1-a^2}\frac{(x_1,x_{2},0)}{\sqrt{x_1^2+x_{2}^2}}.
$$
Therefore, 
$$
x-u=\sqrt{1-a^2}(\cos(s),\sin(s),0)
$$
for some $s\in [0,2\pi)$.  By \eqref{xWedgeEprimeConstraints}, it must be that $s\in [0,\phi]$. Consequently, $x-u\in e\cup\{b,c\}$.
\end{proof}

\par The subsequent proposition is key to showing Meissner polyhedra have constant width. 
\begin{lem}\label{DiameterMone}
Suppose $M$ is a Meissner polyhedron. Then $\textup{diam}(M)= 1$.
\end{lem}
\begin{proof}
Let $X\subset \R^3$ be the extremal set with diameter one for which $M$ is based on and recall that $X\subset M$ by Corollary \ref{LittleEdgeLemma}. 
There are two vertices $x,y\in X$ with $|x-y|=1$.  As a result, $\textup{diam}(M)\ge 1$.  

\par Choose a pair $x,y\in \partial M$ with $|x-y|=\text{diam}(M)$.   First suppose $x$ is a vertex of $B(X)$; that is $x\in X$.  Since $y\in B(X)$, $|x-y|\le 1$. Next suppose $x$ belongs to an edge $e$ with dual edge $e'$. If $|x-y|>1$, then $y\in W_{e'}$ by Lemma \ref{WedgeLemma}. In view of Proposition \ref{structureofMprop}, $y\in W_{e'}\cap M=W_{e'}\cap \text{Sp}(b',c')\subset \text{Sp}(b',c')$; here $b',c'$ are the endpoints of the edge $e'$.  As $\text{Sp}(b',c')=B(C)$ for a circle $C$ which includes the edge $e$, it must be that $y\in B(e)$. Therefore, we actually must have $|x-y|\le 1$ since $x\in e$.

\par Alternatively, suppose $x$ is an interior point of a portion of a spindle $\partial\text{Sp}(b',c')$. Then $x$ is a smooth point of $\partial\text{Sp}(b',c')$ with outward unit normal $u$. By Lemma \ref{WedgeLocationLemma}, $z=x-u\in e\cup\{b,c\}$. In particular, $z\in \partial M$. Since this normal is unique and $x-y\in N_M(x)$,  $x-y=t u$ for some $t>0$ by Lemma \ref{normalConeLemma}. 
In particular,  
$$
x+t(z-x)=y.
$$
As $y\in \partial M$ and $M$ is strictly convex, it must be that $t\ge1$. If $t>1$, the line segment $\{ x+s(z-x): 0\le s\le 1\}$ from one boundary point $x$ to another $z$ extends nontrivially to $y\in M$.  This  would also contradict the fact that $M$ is strictly convex. Therefore, $t=1$, $y=z$ and $|x-y|=|u|=1$. 

\par Finally suppose that $x$ is an interior point of spherical region of $\partial M$ with corresponding center $c\in X\subset \partial M$. Then $M$ has a unique normal $x-c$ which again is necessarily parallel to $x-y$. Arguing as above, we find $y=c$ and $|x-y|=1$. So in all cases, $\text{diam}(M)=|x-y|\le 1$.
\end{proof}

We now have the necessary ingredients to show that Meissner polyhedra have constant width. 
\begin{proof}[Proof of Theorem \ref{ConstantWidthThm}]
Suppose that $X$ is an extremal set of diameter one with $m\ge 4$ points and that its Reuleaux polyhedron $B(X)$ has dual edge pairs $(e_1,e_1'),\dots, (e_{m-1},e_{m-1}')$.  Let $M=B(Y)$ be the Meissner polyhedron  with $Y=X\cup e_1\cup\dots\cup e_{m-1}.$  We first claim that 
\be\label{DiameterBaseSetClaim}
\textup{diam}(Y)\le 1.
\ee
To verify this claim, suppose $x, y\in Y$.  If $x,y\in X$, then $|x-y|\le 1$ as $X$ has diameter one. Suppose $x\in e_1$ and $z\in B(X)$ with $|x-z|>1$. Lemma \ref{WedgeLemma} implies $z\in W_{e_1'}$ and that $z$ is not a vertex.  It follows that the only edge that $z$ can belong to is $e_1'$.  Therefore, if $y\in X\cup e_1\cup\dots\cup e_{m-1}$, then $|x-y|\le 1$.   We conclude \eqref{DiameterBaseSetClaim}.

\par Inequality \eqref{DiameterBaseSetClaim} implies $Y\subset B(Y)$, which in turn gives
$$
M=B(Y)\supset B(B(Y))=B(M).
$$
According to Lemma \ref{DiameterMone}, we also have $\text{diam}(M)=1$. As a result, $M\subset B(M)$, so $M=B(M)$. It now follows from Theorem \ref{ConstantWidthCharacterization} that $M$ has constant width. 
\end{proof}
We conclude this discussion by verifying that Meissner and Shilling's construction of the Meissner tetrahedra described in the introduction does indeed yield two shapes of constant width. 
\begin{cor}
Meissner tetrahedra have contant width. 
\end{cor}
\begin{proof}
Let $R=B(\{a_1,a_2,a_3,a_4\})$ be a Reuleaux tetrahedron and denote $e_{ij}\subset R$ as the edge which joins vertex $a_i$ to $a_j$.  It is easy to check that $R$ has three dual edge pairs 
$$
(e_{12}, e_{34}),\; (e_{13}, e_{24}), \; \text{and}\;  (e_{14}, e_{23}).
$$
By Theorem \ref{ConstantWidthThm},  
\be\label{SpecificMeissnerTetra}
M=B\left(\{a_1,a_2,a_3,a_4\}\cup e_{12}\cup e_{13}\cup e_{14}\right)
\ee
has constant width.  It suffices to show that $M$ is a Meissner tetrahedron. 

\par In view of \eqref{structureofpartialM}, 
$$
\begin{cases}
 \partial M\cap W_{e_{23}}= \partial \textup{Sp}(a_2,a_3)\cap W_{e_{23}}\\
  \partial M\cap W_{e_{24}}= \partial \textup{Sp}(a_2,a_4)\cap W_{e_{24}}\\
 \partial M\cap W_{e_{34}}= \partial \textup{Sp}(a_3,a_4)\cap W_{e_{34}}
\end{cases}
$$
and 
$$
\partial M\cap\left(\displaystyle W_{e_{23}}\cup W_{e_{24}}\cup W_{e_{34}}\right)^c
= \displaystyle \partial R \cap \left(\displaystyle W_{e_{23}}\cup W_{e_{24}}\cup W_{e_{34}}\right)^c.
$$
Therefore, $\partial M$ is obtained by replacing the portion of $\partial R$ in the wedges 
$W_{e_{23}},W_{e_{24}},W_{e_{34}}$ by pieces of the respective spindle tori $\partial\textup{Sp}(a_2,a_3),\partial\textup{Sp}(a_2,a_4),\partial\textup{Sp}(a_3,a_4)$. This is equivalent to performing surgery as
introduced by Meissner and Schilling \cite{Meissner}. Consequently, $M$ is a Meissner tetrahedron.
\end{proof}
\begin{rem}\label{MeissnerTetraFormula}
The Meissner tetrahedron $M$ \eqref{SpecificMeissnerTetra} is one in which the three smooth edges share the face opposite vertex $a_1$.  Since each $a_i$ is an endpoint of the edges $e_{12}, e_{13}$ or $e_{14}$, 
\be
M=B(e_{12}\cup e_{13}\cup e_{14}).
\ee
\par  Let us also consider the Meissner tetrahedron in which the three smoothed edges share vertex $a_1$. This constant width shape is given by
\be
M=B(\{a_1,a_2,a_3,a_4\}\cup e_{23}\cup e_{24}\cup e_{34}).
\ee
Since $a_2,a_3,a_4$ are endpoints of the edges $e_{23}, e_{24},$ and $e_{34}$, $M$ can be expressed as 
\be\label{SpecificMeissnerTetraTwo}
M=B(\{a_1\}\cup e_{23}\cup e_{24}\cup e_{34}).
\ee
These intersection formulae for Meissner tetrahedra were the ones we mentioned in the introduction of this article. 
\end{rem}

\begin{figure}[h]
\centering
 \includegraphics[width=.42\textwidth]{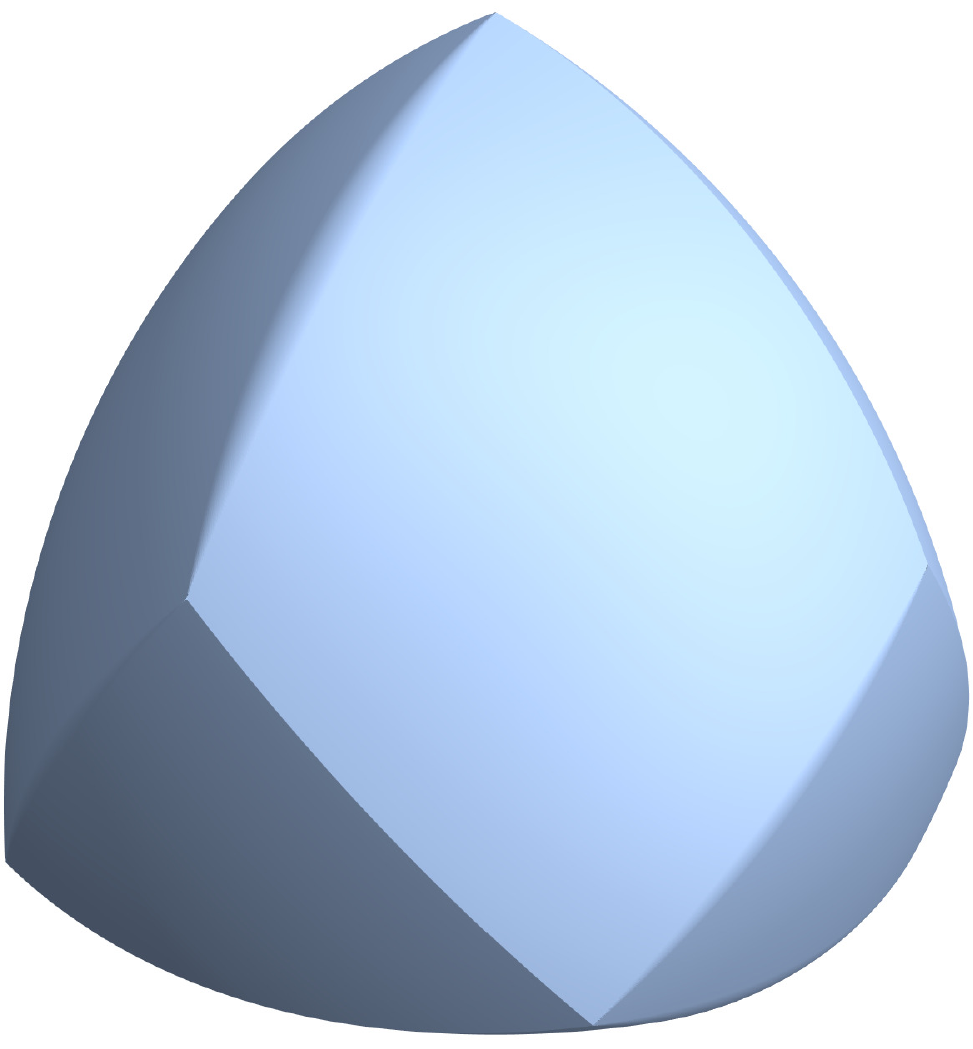}
 \hspace{.2in}
  \includegraphics[width=.42\textwidth]{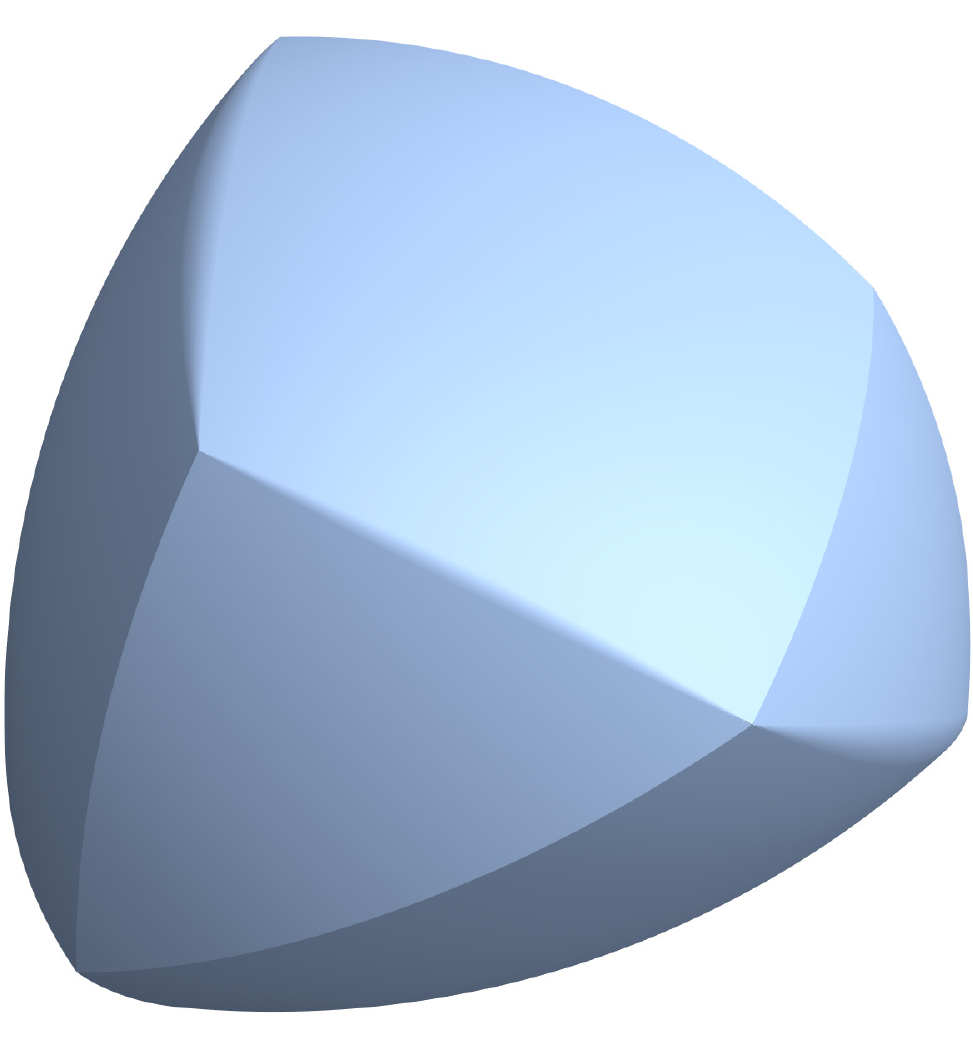}
 \caption{A Meissner polyhedron based on the critical set $\{a_1,\dots,a_9\}$ from Example \ref{DTexample}.}\label{MeissDT}
\end{figure}

\section{Density Theorem }\label{DensitySect}
We will show that Meissner polyhedra are dense within the class of constant width bodies in $\R^3$. 
The proof we give follows Sallee's work \cite{MR296813} in some key aspects. Given a constant width $K$, we will first show that we can find some ball polyhedron $B(X)$ close to $K$ in the Hausdorff metric. Next we will argue that we can choose the centers $X$ so that $X$ is extremal. Then we will explain how to design a Meissner polyhedron based on $X$ and argue this Meissner polyhedra is at least as close to $K$ as $B(X)$ is.

\subsection{First approximation}
First, we will discuss a variant of Theorem \ref{ConstantWidthCharacterization}.
\begin{lem}\label{firstapproxLemma}
Assume $K\subset \R^3$ is a constant width body and $\{x_n\}_{n\in \N}\subset \partial K$ is dense. Then 
\be\label{SecondIntersectionFormula}
K=\bigcap^\infty_{n=1}B(x_n),
\ee
and $\bigcap^N_{n=1}B(x_n)\rightarrow K$ as $N\rightarrow\infty$ in the Hausdorff topology.
\end{lem}
\begin{proof}
Equality \eqref{SecondIntersectionFormula} follows from Theorem \ref{ConstantWidthCharacterization} and Remark \ref{DenseIsEnough}. Now set $K_N:=\cap^\infty_{n=1}B(x_n)$ for $N\in \N$,  and observe
\be\label{DenseBoundaryFormula}
K=\bigcap^\infty_{N=1} K_N.
\ee
Also notice that
\be\label{KNKNplusOne}
K\subset  K_{N+1}\subset K_N\subset K_1.
\ee
Blaschke's selection theorem  (Theorem 2.5.2 of \cite{MR3930585}) implies that there is a convex body $K^*\subset \R^3$ and a subsequence $(K_{N_j})_{j\in \infty}$ such that $N_j\rightarrow\infty$, $N_j<N_{j+1}$, and $K_{N_j}\rightarrow K^*$  in the Hausdorff topology.  We note that $K\subset K^*$. In order to conclude this proof, it suffices to show that $K=K^*$ since $K$ does not depend on the subsequence. 

\par Fix $y\in K^*$. For each $\delta>0$, there is $j_\delta$ for which
$$
K^*\subset K_{N_j}+ B_\delta(0)
$$
for each $j\ge j_\delta$.  Choose  $x_j\in K_{N_j}$ with 
\be\label{whyxndeltaineq}
|y-x_j|\le \delta
\ee
for $j\ge j_\delta$. We may select a subsequence of $(x_{j})_{j\in \N}$ which converges to some $x_\delta$. In view of \eqref{KNKNplusOne}, $x_\delta\in K_N$ for all $N\in \N$. Thus, $x_\delta\in K$ by \eqref{DenseBoundaryFormula}.  In particular, we find $|y-x_\delta|\le \delta$ upon sending $j\rightarrow\infty$ in \eqref{whyxndeltaineq} along an appropriate subsequence. It follows that $y=\lim_{\delta\rightarrow 0}x_\delta\in K$. 
\end{proof}
Since the diameter of a constant width body is equal to one, we may assume that $|x_1-x_2|=1$. It then follows that $\{x_n\}_{1\le n\le N}\subset \partial K$ has diameter one for each $N\ge 2$.  As a result, we have found a first approximation to $K$ in terms of a ball polyhedron. 
\begin{cor}\label{firstApproxBX}
Assume $K\subset \R^3$ is a constant width body and $\epsilon>0$. There is a finite set $X\subset \partial K$ with diameter equal to one such that
$$
d(K, B(X))\le\epsilon. 
$$
\end{cor}

\begin{figure}[h]
\centering
 \includegraphics[width=.42\textwidth]{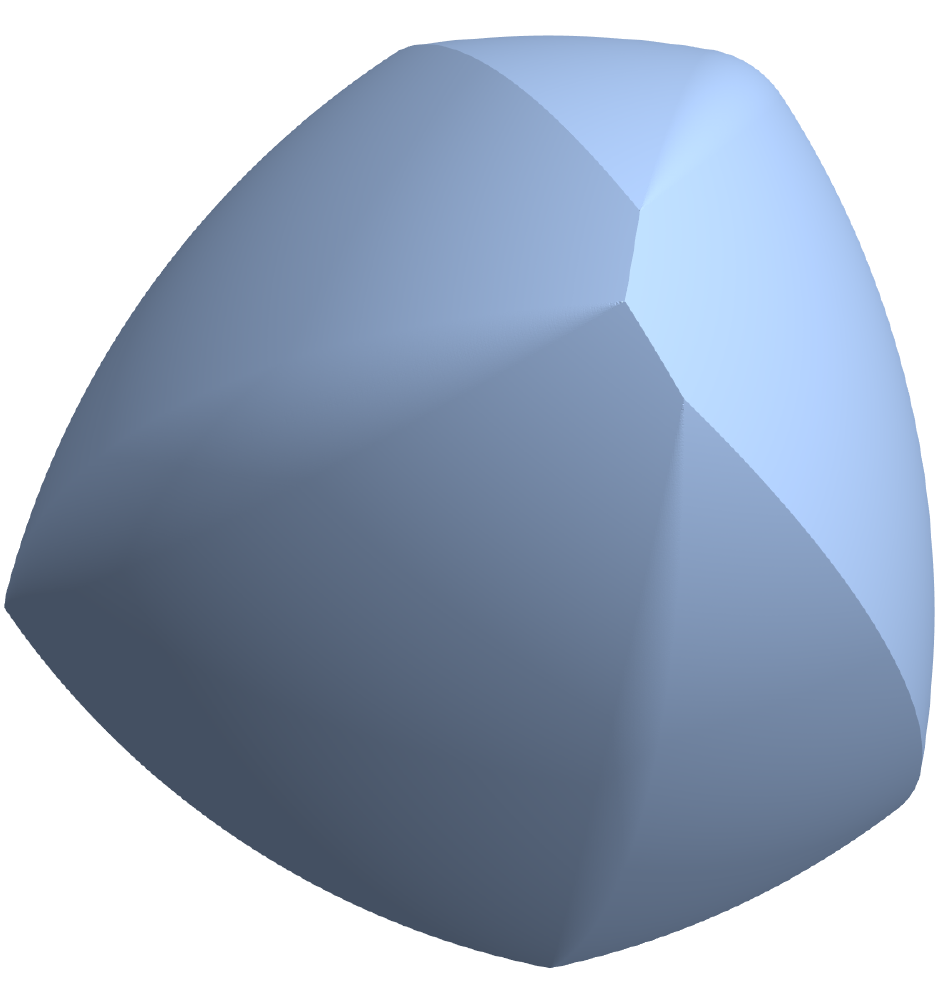}
 \hspace{.2in}
  \includegraphics[width=.42\textwidth]{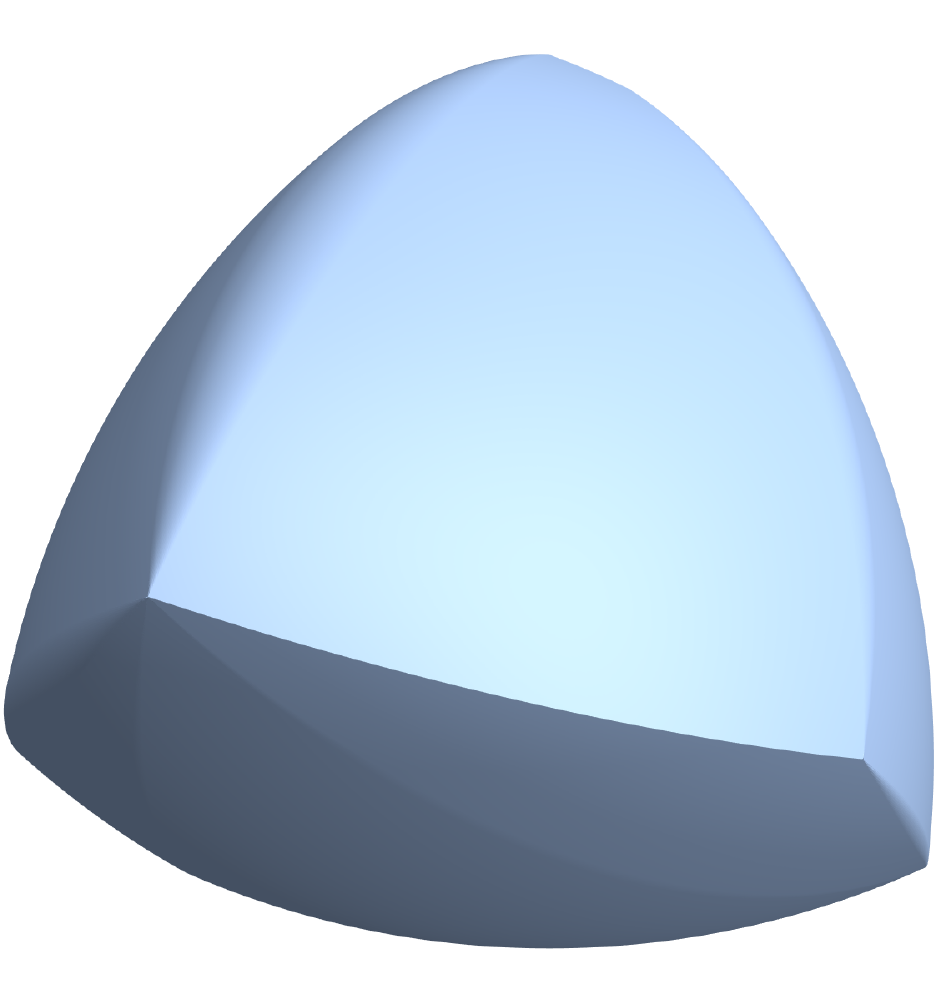}
 \caption{A Meissner polyhedron based on the critical set $\{a_1,\dots,a_8\}$ from Example \ref{BullExample}.}\label{MeissBull}
\end{figure}

\subsection{Two lemmas}
Given a constant width $K\subset \R^3$, we can find an approximating ball polyhedron $B(X)$ to within any tolerance we choose. We will be able to find a Reuleaux polyhedron and then a Meissner polyhedron that are just as good approximations by intersecting $B(X)$ with an appropriate family of balls. To this end, we will need the following two lemmas. The first is due to Sallee (Lemma 2.6 \cite{MR296813}). 
\begin{lem}
Suppose $K,K'\subset \R^3$ are convex bodies with $x\in K'$ and $\textup{diam}(K)\le 1$. Then 
\be\label{distanceInequalitySallee}
d(K,B(x)\cap K')\le d(K,K').
\ee
\end{lem}
\begin{proof}
Let $\delta=d(K,K')$. If $\delta=0$, then $K=K'$; and since $\text{diam}(K)\le 1$, $B(x)\cap K'=B(x)\cap K=K$. Therefore, 
equality holds in \eqref{distanceInequalitySallee}. 

\par Now suppose $\delta>0$.  If \eqref{distanceInequalitySallee} fails to hold, 
then $d(K,B(x)\cap K')>\delta$ and one of the inclusions 
$$
K\subset  B(x)\cap K'+B_\delta(0) \quad \text{and}\quad  B(x)\cap K'\subset K+B_\delta(0)
$$ 
is false.  As $\delta=d(K,K')$, 
\be\label{KprimeinK}
K'\subset K+B_\delta(0).
\ee
Therefore, $B(x)\cap K'\subset K+B_\delta(0)$. It follows that there is $y\in K$ with 
\be\label{yBOnePlusDelta}
y\not\in B(x)\cap K'+B_\delta(0)\subset B(x)+B_\delta(0)=B_{1+\delta}(x).
\ee
By \eqref{KprimeinK}, there is $z\in K$ with $|z-x|\le \delta$.  However, as the diameter of $K$ is at most one 
$$
|y-x|\le |y-z|+|z-x|\le 1+\delta. 
$$
This contradicts \eqref{yBOnePlusDelta}. As a result, we conclude \eqref{distanceInequalitySallee}. 
\end{proof}

\par The subsequent lemma allows us to use a procedure to go from a set of diameter one to an extremal set by taking appropriate intersections. 

\begin{lem}\label{nearFrameLemma}
Let $X\subset \R^3$ be a finite set with $\text{diam}(X)=1$. Assume
\be\label{nearFrameCond}
X\supset \textup{vert}(B(X)).
\ee
If $Y\subset X$ is the collection of essential points of $X$, then $Y\supset \textup{vert}(B(Y))$. And if $Y$ has at three points, then $Y= \textup{vert}(B(Y))$ and $Y$ is extremal.
\end{lem}
\begin{proof}
Assume $y\in \textup{vert}(B(X))$ is a principle vertex. As $\textup{vert}(B(X))\subset X$,  $y\in X$ and $\text{val}(y,X)\ge 3$. It follows from Lemma \ref{essLem} $(iii)$ that $y$ is essential. Likewise, Lemma \ref{essLem} $(iii)$ also gives that each dangling vertex of $X$ is essential.  Therefore, $\textup{vert}(B(X))\subset Y$.  As $B(Y)=B(X)$ by Lemma \ref{essLem} part $(ii)$, we have 
\be\label{BYincludedinY}
\textup{vert}(B(Y))=\textup{vert}(B(X))\subset Y.
\ee

\par Suppose $Y$ has at least three points. Let $z\in Y$ and consider $\partial B(z)\cap B(Y)$, which is the face of $B(Y)$ opposite $z$.  By Proposition 6.1$(i)$ of \cite{MR2593321}, $\partial B(z)\cap B(Y)$ has at least two principle vertices $p,q$; this is where we need to assume that $\#Y\ge 3$.   In view of \eqref{BYincludedinY}, $p,q\in Y$. 
Therefore, $z\in \textup{vert}(B(Y))$. Thus, $Y\subset \textup{vert}(B(Y))$. 

\par If $Y=\textup{vert}(B(Y))$ has exactly three points $Y=\{y_1,y_2,y_3\}$, then $y_1,y_2,y_3$ are necessarily dangling vertices. However, by 
Proposition 6.1$(i)$ of \cite{MR2593321}, each face of $Y$ has at least two principle vertices. It would follow that  $Y$ has at least five points.  Therefore, $Y$ actually has at least four points and is thus extremal by Theorem \ref{ExtendedGHS}. 
\end{proof}

\subsection{Second approximation}
We will now show that we can always find a Reuleaux polyhedron which is arbitrarily close to a given constant width $K$. 

\begin{prop}
Suppose $K\subset \R^3$ has constant width and $\epsilon>0$. There is a finite set $Z\subset \R^3$ which has diameter one, has at least four points, is extreme, and satisfies
$$
d(K,B(Z))\le \epsilon.
$$
\end{prop}
\begin{proof}
Choose a finite set $X\subset \R^3$ with diameter one for which $d(K,B(X))\le\epsilon$.  Such a ball polyhedron exists by Corollary \ref{firstApproxBX}.  In view of our proof of Lemma \ref{firstapproxLemma} and Lemma \ref{distanceInequalitySallee}, we may assume without any loss of generality that $X$ has at least four points. We will first argue that we can choose an approximating $B(Y)$ so that $Y\supset \text{vert}(B(Y))$.

\par If $X$ does not include $\text{vert}(B(X))$. Then $X$ is not extremal by Theorem \ref{ExtendedGHS}, so $e(X)<2(|X|-1)$. 
Choose a principle vertex $y_1\in \text{vert}(B(X))$ with $y_1\not\in X$. Then $|y_1-x|=1$ for at least three points $x\in X$. It follows that 
$e(X\cup\{y_1\})\ge e(X)+3.$ Therefore, 
\begin{align}
0&\le 2(|X\cup\{y_1\}|-1)-e(X\cup\{y_1\})\\
&=2|X|-e(X\cup\{y_1\})\\
&\le 2|X|-e(X)-3\\
&=\left[2(|X|-1)-e(X)\right]-1.
\end{align}
Lemma \ref{distanceInequalitySallee} gives that $d(K,B(X\cup\{y_1\}))\le\epsilon$. Therefore, $B(X\cup\{y_1\})$ is at least as good of an approximation to $K$ as $B(X)$ is and $X\cup\{y_1\}$ is closer to being extremal than $X$ is.

\par Likewise if $X\cup \{y_1\}$ does not include $\text{vert}(B(X\cup \{y_1\}))$, then we can find a principle vertex $y_2\in  B(X\cup \{y_1\})$ for which $y_2\not\in X\cup \{y_1\}$ and
\begin{align}
0&\le 2(|X\cup\{y_1,y_2\}|-1)-e(X\cup\{y_1,y_2\})\\
&\le \left[2(|X\cup\{y_1\}|-1)-e(X\cup\{y_1\})\right]-1\\
&=\left[2(|X|-1)-e(X)\right]-2.
\end{align}
Moreover, Lemma \ref{distanceInequalitySallee} implies that $d(K,B(X\cup\{y_1,y_2\}))\le\epsilon$.
Since $\ell=2(|X|-1)-e(X)$ is  a fixed natural number, if we continue in this manner, we will find a set $Y=X\cup \{y_1,\dots, y_{j}\}$ for which $j\le \ell$ and either $Y\supset \text{vert}(B(Y))$ or $e(Y)=2|Y|-2$. In the latter, case we have $Y= \text{vert}(B(Y))$ by Theorem \ref{ExtendedGHS}.  Thus, there is finite $Y\subset \R^3$ with diameter one such that $d(K,B(Y))\le\epsilon$ and $Y\supset \text{vert}(B(Y))$. 

\par Let $Z\subset Y$ be the essential points of $Y$. By Lemma \ref{nearFrameLemma} and Lemma \ref{essLem}$(ii)$,  $Z\supset \text{vert}(B(Z))$ and $d(K,B(Z))\le\epsilon$. Moreover, if $Z$ has at least four points, $Z$ is extremal, which would allow us to conclude this proof.  As noted in our proof of Lemma \ref{nearFrameLemma}, $Z$ cannot have exactly three points and satisfy $Z\supset \text{vert}(B(Z))$. Alternatively, if $Z=\{z_1,z_2\}$ has only has two points, we may choose any $z_3\in \partial B(z_1)\cap \partial B(z_2)$ and consider the ball polyhedron $B(\{z_1,z_2,z_3\})$. This ball polyhedron has two principle vertices $z_4,z_5$ which are not included in $\{z_1,z_2,z_3\}$. However, is routine to check that $Z'=Z\cup \{z_3,z_4\}=\{z_1,z_2,z_3,z_4\}$ does satisfy $Z'= \text{vert}(B(Z'))$.  Lemma \ref{distanceInequalitySallee} would also give $d(K,B(Z'))\le\epsilon$. As a result, we can take $Z'$ instead of $Z$ to obtain the desired conclusion. Finally, if $Z=\{z_1\}$ is a singleton, we choose any $z_2\in \partial B(z_1)$ and argue as we just did on $Z\cup \{z_2\}=\{z_1,z_2\}$.  As a result, there is an extremal set $Z\subset \R^3$ having at least four points and diameter one such that $d(K,B(Z))\le\epsilon$. 
\end{proof}
We are finally in position to prove the Density Theorem which asserts that Meissner polyhedra are dense within the space of constant width bodies. 
\begin{proof}[Proof of the Density Theorem]
Let $\epsilon>0$ and choose a extremal set $Z\subset \R^3$ of diameter one having at least four points such that  
$$
d(K,B(Z))\le \epsilon.
$$ 
Such a $Z$ exists by the previous lemma. Suppose that $|Z|=m$ and  
$$
 (e_1,e_1'),\dots, (e_{m-1},e_{m-1}')
 $$ 
 are the dual edge pairs of $B(Z)$. Applying Lemma \ref{distanceInequalitySallee} countably many times gives
$$
d(K,B(Z\cup \tilde e_1\cup\dots\cup \tilde e_{m-1} ))\le d(K,B(Z))\le \epsilon
$$
for any countable dense subsets $\tilde e_j\subset e_j$ for $j=1,\dots, m-1$. As noted in Remark \ref{DenseIsEnough}, $B(\tilde e_j)=B(e_j)$ for $j=1,\dots, m-1$.  As a result, 
\begin{align}
B(Z\cup \tilde e_1\cup\dots\cup \tilde e_{m-1} )&=B(Z)\cap B( \tilde e_1)\cap\dots\cap B(\tilde e_{m-1} )\\
&=B(Z)\cap B(  e_1)\cap\dots\cap B( e_{m-1} )\\
&=B( Z\cup e_1\cup \dots \cup e_{m-1}).
\end{align}
Therefore,  
$$
 d(K,B( Z\cup e_1\cup \dots \cup e_{m-1}))\le \epsilon.
 $$ 
We conclude as $M=B(Z\cup e_1\cup \dots \cup e_{m-1})$ is a Meissner polyhedron. 
\end{proof}
\noindent 
{\bf Acknowledgements}: The author wrote this paper while visiting  
CIMAT. He would especially like to thank Gil Bor and H\'ector Chang-Lara for their hospitality.

\appendix

\section{Volume of Meissner tetrahedra}
In this appendix, we'll compute the volume of the two types of Meissner tetrahedra. A key formula for us will be {\it Blaschke's relation}, which asserts that if $M\subset \R^3$ is a constant width body, then 
\be
V(M)=\frac{1}{2}S(\partial M)-\frac{\pi}{3}
\ee
(Theorem 12.1.4 of \cite{MR3930585}). Here $V(M)$ is the volume of $M$ and $S(\partial M)$ is the perimeter of $M$.  Therefore, it suffices to compute the perimeter of each Meissner tetrahedron.  To this end, we will employ the Gauss--Bonnet formula and adapt the computation of the perimeter of a Reuleaux tetrahedron 
made by Harbourne \cite{harbourne}. 

\subsection{Interior angles}
Suppose $a_1,a_2,a_3, a_4\in \R^3$ with 
$$
 |a_i-a_j|=1
 $$
 for $i\neq j$. Let $T$ be a regular tetrahedron with vertices $a_1,a_2,a_3, a_4$  and $R=B(\{a_1,a_2,a_3,a_4\})$ be the associated Reuleaux tetrahedron.  We will first determine the angle between any two planes spanning the faces of $T$. We recall the {\it dihedral angle} $\theta\in [0,\pi/2]$ between two intersecting planes with outer normals $n_1$ and $n_2$ satisfies
$$
\cos(\theta)=\frac{|n_1\cdot n_2|}{|n_1||n_2|}.
$$

\begin{lem}
The dihedral angle between any of two planes which span the faces of $T$ is $\cos^{-1}\left(1/3\right).$
\end{lem}
\begin{proof}
In view of the equality condition for Jung's inequality, we may assume the vertices of $T$ are on a common sphere of radius $\sqrt{3/8}$. In this case, 
$$
1=|a_i-a_j|^2=|a_i|^2+|a_j|^2-2a_i\cdot a_j=\frac{3}{4}-2a_i\cdot a_j
$$
for $i\neq j$. Therefore,  
$$
a_i\cdot a_j=-\frac{1}{8}
$$
for $i\neq j$.    It is also follows from this observation that $T$ is the intersection of the four half spaces 
\be\label{HalfSpacesT}
\{x\in \R^3: x\cdot a_i\ge -1/8\}
\ee
$(i=1,\dots, 4)$.  As a result, the dihedral angle between any of two these boundary planes is 
$$
\cos^{-1}\left(\frac{|a_i\cdot a_j|}{|a_i||a_j|}\right)=\cos^{-1}\left(\frac{1/8}{3/8}\right)=\cos^{-1}\left(\frac{1}{3}\right).
$$
\end{proof}

\par Another important angle to determine is detailed below.  

\begin{prop}\label{interiorAngleProp}
The interior angle made between any two edges at a vertex of $R$ is equal to $\pi-\cos^{-1}(1/3)$.
\end{prop}
\begin{proof}
As the angle of interest is independent of the specific coordinates used, we may assume 
\be
\begin{cases}
a_1=\left(\sqrt{3}/2,0,1/2\right)\\
a_2=\left(\sqrt{3}/2 \cos(\theta), \sqrt{3}/2 \sin(\theta),1/2\right)\\
a_3=\left(0,0,1\right)\\
a_4=\left(0,0,0\right).
\end{cases}
\ee
Here $\theta =\cos^{-1}(1/3)$. Observe the face of $R$ which includes $a_1,a_2,a_3$ is a subset of the unit sphere $\mathbb{S}^2$. Moreover, the edge of the face which joins vertices $a_1$ and $a_2$ is parametrized by  
$$
\gamma(t)=\left(\sqrt{3}/2 \cos(t), \sqrt{3}/2 \sin(t),1/2\right)
$$
for $t\in [0,\theta]$. Also note $\gamma(0)=a_1$ and $\dot\gamma(0)=(0,\sqrt{3}/2,0)$.

\par Let us also consider the geodesic in $\mathbb{S}^2$ joining $a_1$ and $a_2$. This curve belongs to the face in question (by Lemma \ref{SphereConvSubsetLem}) and may be parametrized by 
$$
\xi(s)= \cos(s)a_1+\sin(s)\frac{a_2-(1/2)a_1}{\sqrt{3}/2}
$$
for $0\le s\le \pi/3$.  Note $\xi(0)=a_1$ and 
$$
\dot\xi(0)=\frac{a_2-(1/2)a_1}{\sqrt{3}/2}.
$$
Since 
$$
\frac{\dot\xi(0)}{|\dot\xi(0)|}\cdot \frac{\dot\gamma(0)}{|\dot\gamma(0)|}= \sin(\theta)=\cos(\pi/2-\theta),
$$
the angle between this edge and geodesic at vertex $a_1$ is $\pi/2-\theta$. 

\par  By symmetry, we find that the angle between the edge joining $a_1$ and $a_3$ and the geodesic joining $a_1$ and $a_3$ is also $\pi/2-\theta$.  
The angle between the edge joining $a_1$ and $a_2$ and the edge joining $a_1$ and $a_3$ is then equal to $
2(\pi/2-\theta)+\theta=\pi-\theta. 
$
See Figure \ref{IntAngleFig}. Here we used that the angle between the geodesic joining $a_1$ and $a_2$ and the geodesic joining $a_1$ and $a_3$ is equal to the dihedral angle $\theta$. Moreover, we can add these angles to get the desired interior angle as the tangent vectors of all of the curves considered at the point $a_1$ belong to the tangent space of $\mathbb{S}^2$ at $a_1$. 
\end{proof}
\begin{figure}[h]
     \centering
     \begin{subfigure}[t]{0.45\textwidth}
         \centering
         \includegraphics[width=\textwidth]{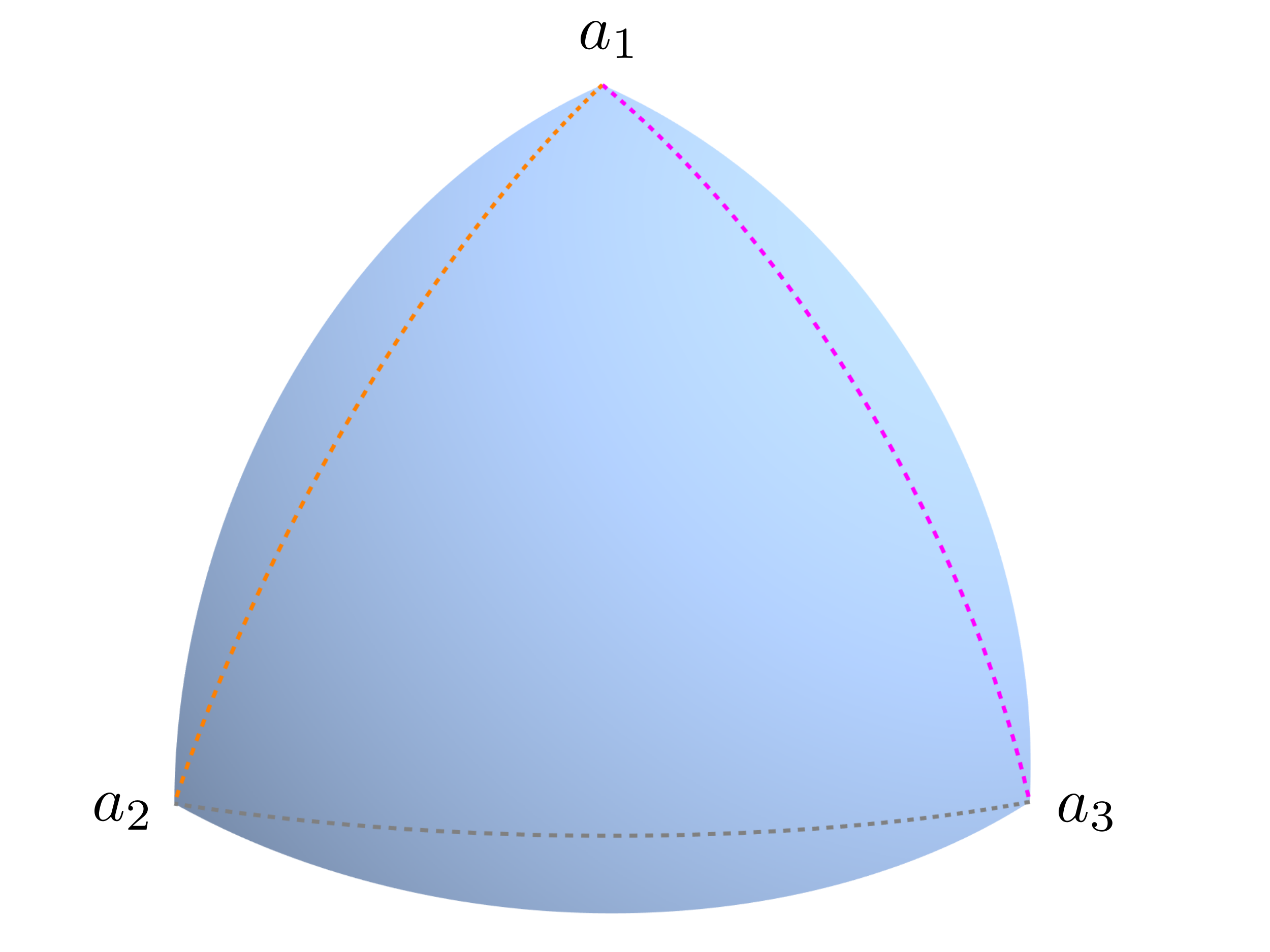}
         \newsubcap{The face $\partial B(a_4)\cap R$ of $R$. The dashed curves are geodesics joining the vertices $a_1,a_2,a_3$. Note that this surface consists of the geodesic triangle with vertices $a_1,a_2,a_3$ and three sliver surfaces. }\label{SliverFig}
     \end{subfigure}
    \hspace{.4in}
     \begin{subfigure}[t]{0.47\textwidth}
         \centering
         \includegraphics[width=\textwidth]{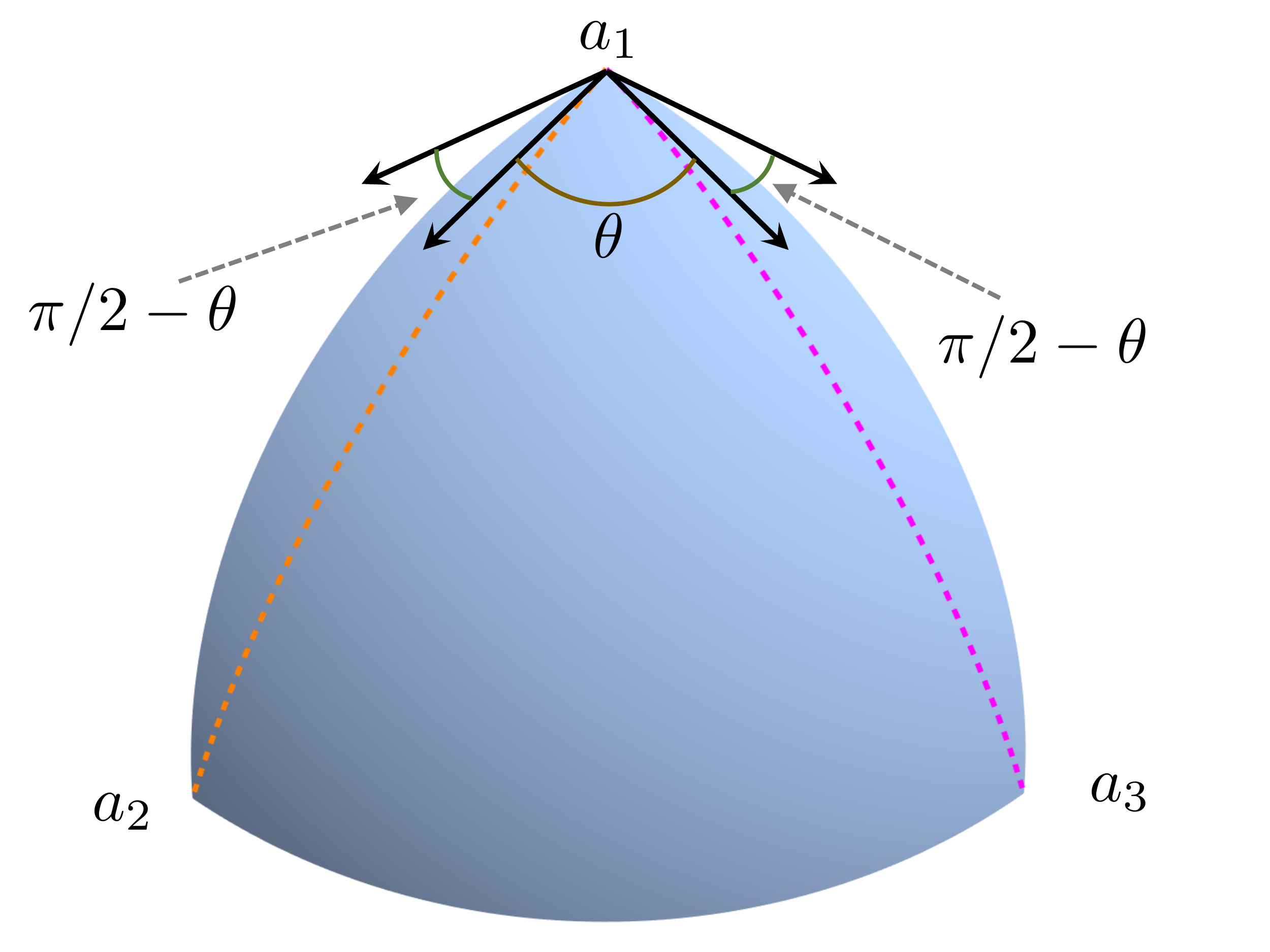}
               \newsubcap{This diagram illustrates the argument made in our proof of Proposition \ref{interiorAngleProp}, which splits the computation of the interior angle at the vertex $a_1$ into three contributions. Here $\theta=\cos^{-1}(1/3)$.}\label{IntAngleFig}
     \end{subfigure}
\end{figure}
\subsection{The Gauss-Bonnet formula}
The Gauss-Bonnet theorem asserts that if $P$ is a compact, two dimensional Riemannian manifold with piecewise smooth boundary made up of curves $C_1,\dots, C_N$ then
\be\label{GBformula}
\int_{P}KdA+\sum^N_{i=1}\left\{\int_{C_i}k_gds+(\pi-\theta_i)\right\}=2\pi \chi(P).
\ee
Here $\pi-\theta_i$ is the angle that the tangent turns from $C_i$ to $C_{i+1}$ for $i=1,\dots, N-1$ and from $C_N$ to $C_1$ when $i=N$; that is, $\theta_i$ is the corresponding interior angle of the curves.  The function $K$ is the Gaussian curvature of $P$. Likewise, $k_g$ denotes the geodesic curvature of $C_i$ in the integral $\int_{C_i}k_gds$. Finally, $\chi(P)$ is the Euler characteristic of $P$.

\par In order to apply this theorem, we will need to know the geodesic curvature of a circle in $\mathbb{S}^2$, which is the intersection of a plane with $\mathbb{S}^2$
\be\label{circleSphere}
\left\{x\in \mathbb{S}^2: x\cdot w=\sqrt{1-r^2}\right\}.
\ee
Here $|w|=1$ and $r\in (0,1]$ is the radius of the circle. If $u,v\in \mathbb{S}^2$ are chosen so that $\{u,v,w\}$ is an orthonormal basis with $u\times v=w$, then 
\be
\gamma(t)=r\left(u \cos (t/r) + v\sin (t/r)\right)+\sqrt{1-r^2}\; w\quad (t\in \R)
\ee
is a unit-speed parametrization of the circle. It turns out that the computation   
$$
\ddot\gamma(t)=-\gamma(t)+\frac{\sqrt{1-r^2}}{r}\gamma(t)\times \dot\gamma(t)
$$
implies the geodesic curvature of the circle of \eqref{circleSphere} is constant and equal to $\sqrt{1-r^2}/r$. 

\par Note that \eqref{circleSphere} is a great circle when $r=1$, which has vanishing geodesic curvature.  We also recall that the Gaussian curvature of $\mathbb{S}^2$ is identically equal to one. It follows that if $P\subset \mathbb{S}^2$ is a geodesic triangle with interior angles $\alpha,\beta,\gamma$, 
$$
S(P)+(\pi-\alpha)+(\pi-\beta)+(\pi-\gamma)=2\pi.
$$
That is, 
\be\label{GeodesicTriangleArea}
S(P)=\alpha+\beta+\gamma-\pi.
\ee
We can now use these observations to compute the surface area of a few regions of interest.
 
\begin{prop}\label{SurfaceAreaProp}
The geodesic triangle in $\partial B(a_4)$ with vertices $a_1,a_2, a_3$ has surface area 
\be\label{GeodesicTriangleAreaTwo}
3\cos^{-1}(1/3)-\pi.
\ee
The surface area of the face $\partial B(a_4)\cap R$ is 
\be\label{NonGeodesicTriangleArea}
2\pi-\frac{9}{2}\cos^{-1}(1/3).
\ee
\end{prop}

\begin{proof}
Without any loss of generality we will assume $a_1,a_2,a_3,a_4$ are given as in the proof of Proposition \ref{interiorAngleProp}. As a result both regions of interest are subsets of $\mathbb{S}^2$. First we recall that the dihedral angle of the regular tetrahedron is $\cos^{-1}(1/3)$. This is also the interior angle for the geodesic triangle $a_1,a_2,a_3$ within the unit sphere $\partial B(a_4)$. In view of formula \eqref{GeodesicTriangleArea}, we conclude the surface area of the geodesic triangle is given by \eqref{GeodesicTriangleAreaTwo}. 

\par Now we consider the face $\partial B(a_4)\cap R$.  Straightforward computations show that each boundary edge is a circular arc in $\mathbb{S}^2$ of radius $r=\sqrt{3}/2$ and arclength $\sqrt{3}/2 \cos^{-1}(1/3)$.  Therefore, the geodesic curvature of these arcs is a constant equal to 
$$
\frac{\sqrt{1-r^2}}{r}=\frac{1}{\sqrt{3}}.
$$ 
And in view of Proposition \ref{interiorAngleProp}, the interior angles at the vertices $a_1,a_2,a_3$ are all equal to $\pi-\cos^{-1}(1/3)$.  Therefore, the Gauss-Bonnet formula gives 
$$
S(\partial B(a_4)\cap R)+3\left[\frac{1}{\sqrt{3}}\frac{\sqrt{3}}{2}\cos^{-1}(1/3)+\cos^{-1}(1/3)\right]=2\pi.
$$
We conclude \eqref{NonGeodesicTriangleArea} by solving for $S(\partial B(a_4)\cap R)$.
\end{proof}
\begin{rem}
Let $\Delta(a_1,a_2,a_3)\subset \partial B(a_4)\cap R$ denote the geodesic triangle with vertices $a_1,a_2,a_3$. Then $\left(B(a_4)\cap R\right)\setminus \Delta(a_1,a_2,a_3)$ is the union of three {\it sliver surfaces} which only overlap at the vertices. See Figure \ref{SliverFig}.  For example, the sliver surface which contains $a_1$ and $a_2$ is equal to $H\cap \left(\partial B(a_4)\cap R\right)$, where $H$ is a half-space which does not contain $a_3$ while $a_1,a_2,a_4\in \partial H$. A direct corollary to Proposition \ref{SurfaceAreaProp} is that the area of each sliver surface is equal to
\be\label{SliverSurface}
S(H\cap \left(\partial B(a_4)\cap R\right))=\pi-\frac{5}{2}\cos^{-1}(1/3).
\ee
\end{rem}

\subsection{Perimeter of a spindle}
Recall that the boundary of a Meissner tetrahedra consists of pieces of spindle tori wedged between intersecting planes.   By rotating and translating a spindle if necessary, we can find the surface area of any such intersection with the following assertion. 

\begin{lem}
Let $\phi\in [0,\pi/2]$ and set 
$$
W_\phi=\left\{x\in \R^3: x_2\ge 0,\; -\sin(\phi)x_1+\cos(\phi)x_2\le 0\right\}.
$$
For $a\in [0,1]$, 
\be\label{SpindleAreaFormula}
S\left(W_\phi\cap \partial \textup{Sp}(ae_3,-ae_3)\right)=2\phi \left(-\sqrt{1-a^2}\sin^{-1}(a)+a\right).
\ee
\end{lem}
\begin{rem}
We will call $W_\phi\cap \partial \textup{Sp}(ae_3,-ae_3)$ a {\it spindle surface with opening angle} $\phi$. 
\end{rem}
\begin{proof}
Recall that $\partial \textup{Sp}(ae_3,-ae_3)$ is the surface of revolution obtained by rotating the curve
$$
\begin{cases}
\left(x_1+\sqrt{1-a^2}\right)^2+x_3^2=1,\\
x_1\ge 0,\\
x_2=0
\end{cases}
$$
about the $x_3$-axis.  It is then routine to check that $W_\phi\cap \partial \textup{Sp}(ae_3,-ae_3)$ can be parametrized via 
$$
X(\theta,t)=\left(\left(-\sqrt{1-a^2}+\cos(t)\right)\cos(\theta),\left(-\sqrt{1-a^2}+\cos(t)\right)\sin(\theta),\sin(t) \right)
$$
for $\theta\in [0,\phi]$ and $t\in [-\sin^{-1}(a),\sin^{-1}(a)]$.  

\par We note that 
$$
\partial_tX(\theta,t)=\left(-\sin(t)\cos(\theta),-\sin(t)\sin(\theta),\cos(t) \right)
$$
and 
$$
\partial_\theta X(\theta,t)=\left(-\left(-\sqrt{1-a^2}+\cos(t)\right)\sin(\theta),\left(-\sqrt{1-a^2}+\cos(t)\right)\cos(\theta),0\right).
$$
As a result, 
\begin{align}
S\left(W_\phi\cap \partial \textup{Sp}(ae_3,-ae_3)\right)&=\int^{\phi}_{0}\int^{\sin^{-1}(a)}_{-\sin^{-1}(a)}|\partial_\theta X(\theta,t)\times \partial_tX(\theta,t)|dtd\theta\\
&=\int^{\phi}_{0}\int^{\sin^{-1}(a)}_{-\sin^{-1}(a)}\left(-\sqrt{1-a^2}+\cos(t)\right)dtd\theta \\ &=\phi \int^{\sin^{-1}(a)}_{-\sin^{-1}(a)}\left(-\sqrt{1-a^2}+\cos(t)\right)dt \\ 
&=2\phi \left(-\sqrt{1-a^2}\sin^{-1}(a)+a\right).
\end{align}
\end{proof}

\subsection{Volume computation}
We are finally ready to compute the volume of the Meissner tetrahedra. 
\begin{prop}
The volume of each Meissner tetrahedron is given by 
$$
\pi\left[\frac{2}{3}-\frac{\sqrt{3}}{4}\cos^{-1}(1/3)\right].
$$
\end{prop}
\begin{proof}
Let $M_1$ be a Meissner tetrahedron in which three edges that share a common vertex are smoothed. In order to compute the perimeter of $M_1$, we note that $\partial M_1$ consists of 3 spindle surfaces with opening angle $\cos^{-1}(1/3)$, one face in common with $R$, and 3 subsets of remaining faces of $R$, each with two sliver surfaces removed. Employing formulae \eqref{GeodesicTriangleAreaTwo},  \eqref{NonGeodesicTriangleArea},\eqref{SliverSurface}, and \eqref{SpindleAreaFormula}, we find
\begin{align}
S(\partial M_1)&=3\cdot \cos^{-1}(1/3)\left(1-\frac{\pi}{2\sqrt{3}}\right)+1\cdot \left(2\pi-\frac{9}{2}\cos^{-1}(1/3)\right)\\
&\hspace{1in} + 3\cdot\left(2\pi-\frac{9}{2}\cos^{-1}(1/3)-2\left[\pi-\frac{5}{2}\cos^{-1}(1/3)\right]\right)\\
&=3\cdot \cos^{-1}(1/3)\left(1-\frac{\pi}{2\sqrt{3}}\right)+1\cdot \left(2\pi-\frac{9}{2}\cos^{-1}(1/3)\right)+3\cdot \frac{1}{2}\cos^{-1}(1/3)\\
&=\pi\left[2-\frac{\sqrt{3}}{2}\cos^{-1}(1/3)\right].
\end{align}

\par Next let $M_2$ be a Meissner tetrahedron having three smoothed edges that share a common face. Observe that $\partial M_1$ consists of 3 spindle surfaces with opening angle $\cos^{-1}(1/3)$, one geodesic triangle within $R$, and 3 subsets of remaining faces of $R$, each with one sliver surface removed. Employing the various formulae above leads to 
\begin{align}
S(\partial M_2)&=3\cdot \cos^{-1}(1/3)\left(1-\frac{\pi}{2\sqrt{3}}\right)+1\cdot \left(3\cos^{-1}(1/3)-\pi\right)\\
&\hspace{1in} + 3\cdot\left(2\pi-\frac{9}{2}\cos^{-1}(1/3)-\left[\pi-\frac{5}{2}\cos^{-1}(1/3)\right]\right)\\
&=3\cdot \cos^{-1}(1/3)\left(1-\frac{\pi}{2\sqrt{3}}\right)+1\cdot \left(3\cos^{-1}(1/3)-\pi\right)+3\cdot\left(\pi-2\cos^{-1}(1/3)\right)\\
&=\pi\left[2-\frac{\sqrt{3}}{2}\cos^{-1}(1/3)\right].
\end{align}
In particular, $S(\partial M_1)=S(\partial M_2)$. And according to Blaschke's relation, 
\begin{align}
V(M_i)&=\frac{1}{2}S(\partial M_i)-\frac{\pi}{3}\\
&=\frac{1}{2}\pi\left[2-\frac{\sqrt{3}}{2}\cos^{-1}(1/3)\right]-\frac{\pi}{3}\\
&=\pi\left[\frac{2}{3}-\frac{\sqrt{3}}{4}\cos^{-1}(1/3)\right]
\end{align}
for $i=1,2$. 
\end{proof}
 
\section{Plotting figures}
We will describe how to plot a Reuleaux and Meissner tetrahedron with \texttt{Mathematica}.  The other Reuleaux and Meissner polyhedra in this article can be plotted with a similar method.  In order to be as explicit as possible, we will show screenshots from a \texttt{Mathematica} notebook we used to plot the Reuleaux tetrahedron in Figure \ref{ReuleauxTetra} and the Meissner tetrahedron in Figure \ref{Meiss1}.   We begin by choosing the vertices $\{a_1,a_2,a_3,a_4\}$ for a specific regular tetrahedron $T$ in $\R^3$.  We chose these vertices because they belong to a sphere of radius $\sqrt{3/8}$ centered at the origin, so the half-spaces which determine $T$ are given by \eqref{HalfSpacesT}. 
\begin{figure}[h]
\begin{centering}
\includegraphics[width=.8\textwidth]{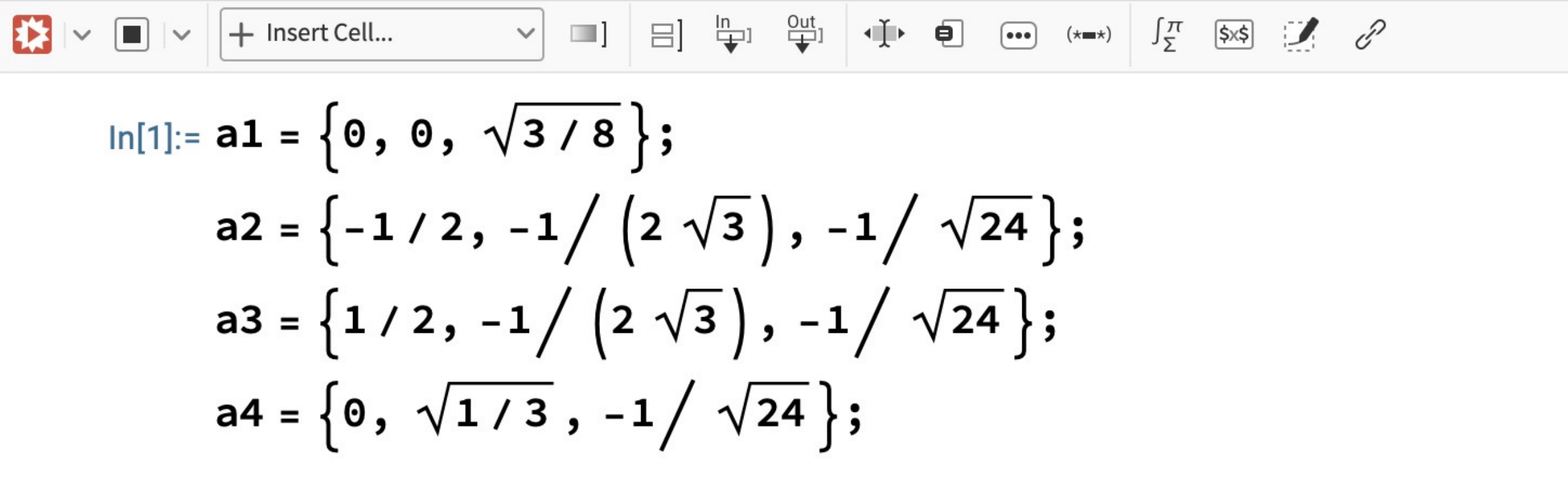}
\end{centering}
\end{figure}

 We recall that $R= B(a_1)\cap  B(a_2)\cap B(a_3)\cap B(a_4)$ is the corresponding Reuleaux tetrahedron, and the faces of $\partial R$ are $\partial B(a_i)\cap R$ for $i=1,\dots, 4$.  We can simplify this description slightly by noting that 
$$
\partial B(a_1)\cap R=\partial B(a_1)\cap  B(a_2)\cap B(a_3)\cap B(a_4),
$$
and likewise for the other faces. This is easy to program into \texttt{Mathematica} using the  built-in \texttt{Sphere}, \texttt{Ball}, \texttt{RegionIntersection},
 \texttt{RegionUnion}, and \texttt{DiscretizeRegion} functions.  The accuracy of our plot of $R$ in the screenshot below is governed by the \texttt{MaxCellMeasure} tolerance. We used $10^{-7}$.  A smaller tolerance would require more computational time but would yield a more accurate plot. 
\\
\begin{figure}[h]
\begin{centering}
\includegraphics[width=.9\textwidth]{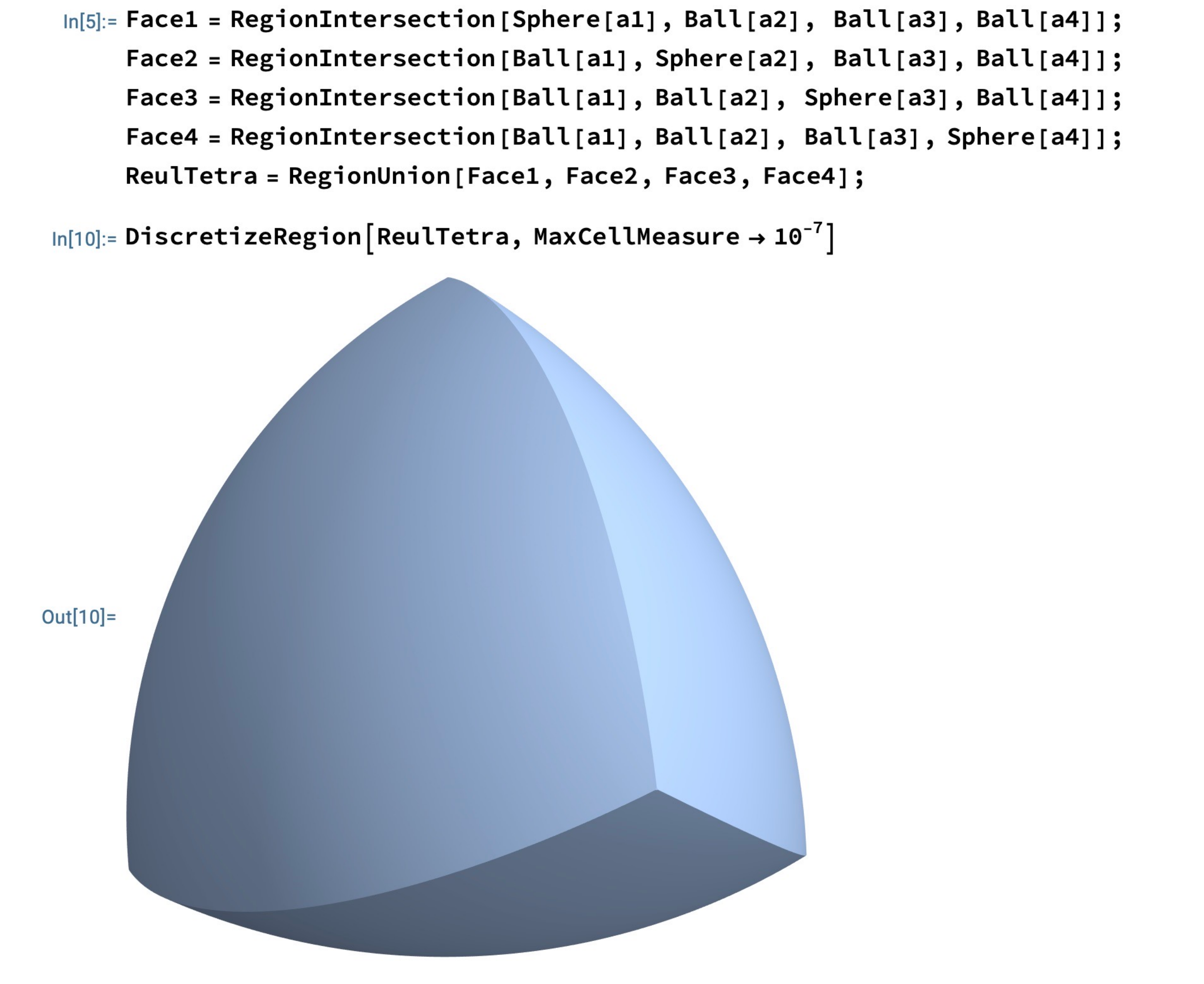}
\end{centering}
\end{figure}

\par Next, we will plot the Meissner polyhedra $M$ whose boundary is obtained by performing surgery on the $\partial R$ near the three edges which have $a_1$ as an endpoint. To this end, we will cut out the relevant portions of $\partial R$.  For example, the portion of $\partial R$ near the edge that joins $a_1$ and $a_4$ which will be cut out lies in the wedge determined by the half-spaces $a_2\cdot x\le -1/8$ and $a_3\cdot x\le -1/8$.  Similar reasoning applies to the other two portions of $\partial R$ which need to be cut out. This can be accomplished with \texttt{Mathematica} using the built-in \texttt{HalfSpace} function as follows. 
\\
\begin{figure}[h]
\begin{centering}
\includegraphics[width=.9\textwidth]{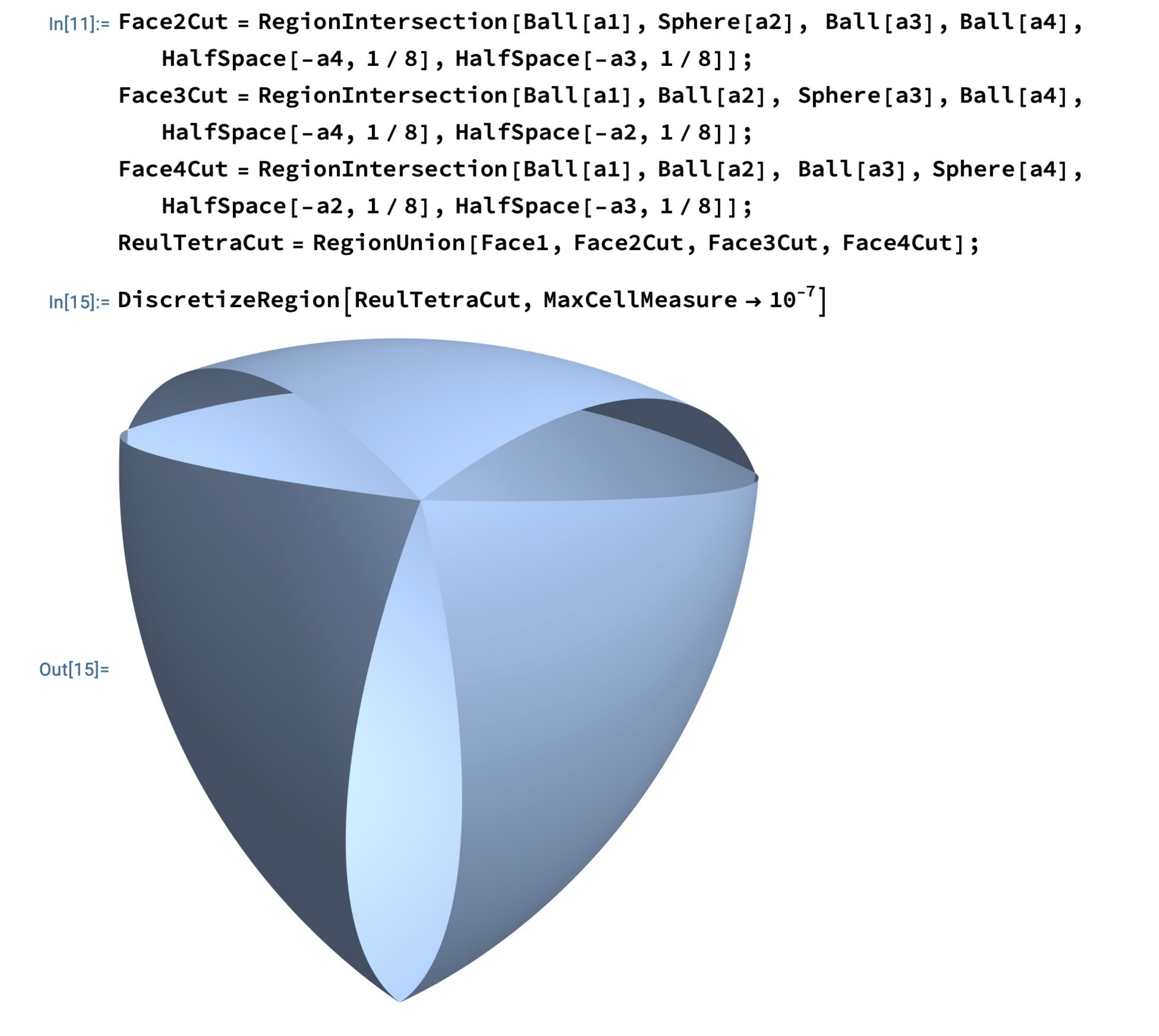}
\end{centering}
\end{figure}

\par Finally, we define a function \texttt{Spindle}, which implicitly specifies the boundary of a spindle as detailed in Corollary \ref{GenSpFormulaCor}.  
This function employs the built-in  \texttt{Norm} and \texttt{Projection} functions. Then we can render the boundary of $M$ by plotting union of the spindles between $a_1$ and $a_j$
for $j=2,3,4$ along with the cut-out Reuleaux tetrahedron displayed above.  We note that the resulting shape contains the entire spindles between $a_1$ and $a_j$
for $j=2,3,4$, while we only need portions of the spindles to fill out the boundary of $M$. It is possible to plot these spindle portions using the built-in \texttt{HalfSpace} function. However, we found this to take considerably more computational time, so we elected to plot the entire spindles. Nevertheless, the two boundaries which are visible by either plotting the entire spindles or only portions of the spindles are the same. 
\\\\
\begin{figure}[h]
\begin{centering}
\includegraphics[width=1\textwidth]{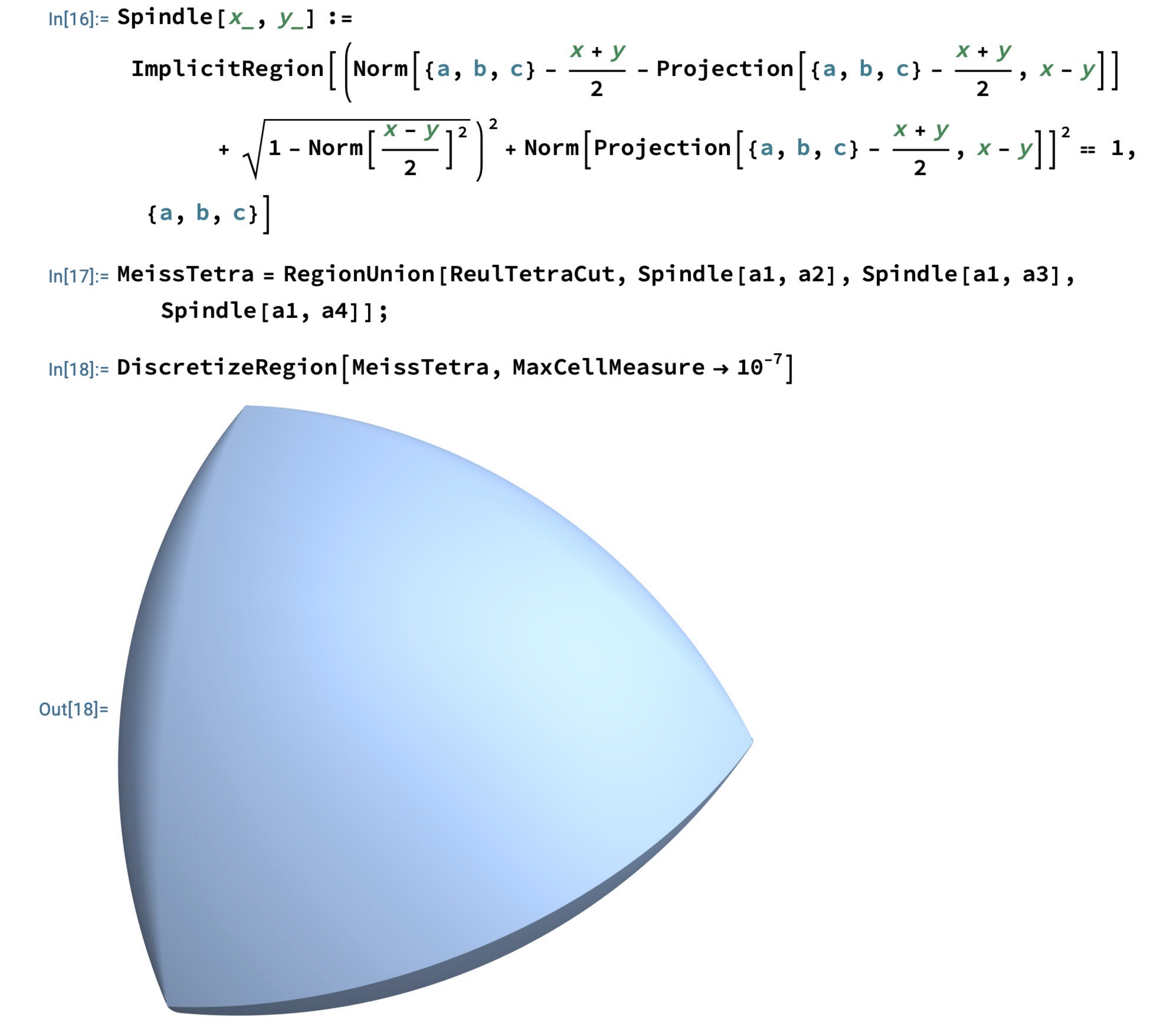}
\end{centering}
\end{figure}


\bibliography{MeissBib}{}

\begin{thebibliography}{10}

\bibitem{MR2343304}
K\'{a}roly Bezdek, Zsolt L\'{a}ngi, M\'{a}rton Nasz\'{o}di, and Peter Papez.
\newblock Ball-polyhedra.
\newblock {\em Discrete Comput. Geom.}, 38(2):201--230, 2007.

\bibitem{MR200695}
H.~G. Eggleston.
\newblock Sets of constant width in finite dimensional {B}anach spaces.
\newblock {\em Israel J. Math.}, 3:163--172, 1965.

\bibitem{MR87115}
B.~Gruenbaum.
\newblock A proof of {V}azonyi's conjecture.
\newblock {\em Bull. Res. Council Israel. Sect. A}, 6:77--78, 1956.

\bibitem{harbourne}
Brian Harbourne.
\newblock Volume and surface area of the spherical tetrahedron (aka reuleaux
  tetrahedron) by geometrical methods.
\newblock
  \url{https://www.math.unl.edu/~bharbourne1/ST/sphericaltetrahedron.html}.
\newblock Accessed on July 30, 2023.

\bibitem{MR87116}
A.~Heppes.
\newblock Beweis einer {V}ermutung von {A}. {V}\'{a}zsonyi.
\newblock {\em Acta Math. Acad. Sci. Hungar.}, 7:463--466, 1956.

\bibitem{MR0123962}
I.~M. Jaglom and V.~G. Boltjanski\u{\i}.
\newblock {\em Convex figures}.
\newblock Holt, Rinehart and Winston, New York, 1960.
\newblock Translated by Paul J. Kelly and Lewis F. Walton.

\bibitem{MR3108700}
B\"{o}rge Jessen.
\newblock \"{U}ber konvexe {P}unktmengen konstanter {B}reite.
\newblock {\em Math. Z.}, 29(1):378--380, 1929.

\bibitem{MR2844102}
Bernd Kawohl and Christof Weber.
\newblock Meissner's mysterious bodies.
\newblock {\em Math. Intelligencer}, 33(3):94--101, 2011.

\bibitem{MR2593321}
Y.~S. Kupitz, H.~Martini, and M.~A. Perles.
\newblock Ball polytopes and the {V}\'{a}zsonyi problem.
\newblock {\em Acta Math. Hungar.}, 126(1-2):99--163, 2010.

\bibitem{MR3930585}
Horst Martini, Luis Montejano, and D\'{e}borah Oliveros.
\newblock {\em Bodies of constant width}.
\newblock Birkh\"{a}user/Springer, Cham, 2019.
\newblock An introduction to convex geometry with applications.

\bibitem{Meissner}
Ernst Meissner and Friedrich Schilling.
\newblock Drei {G}ipsmodelle von {F}lächen konstanter {B}reite.
\newblock {\em Zeitschrift f\"ur angewandte {M}athematik und {P}hysik},
  60:92--94, 1912.

\bibitem{MR3620844}
L.~Montejano and E.~Rold\'{a}n-Pensado.
\newblock Meissner polyhedra.
\newblock {\em Acta Math. Hungar.}, 151(2):482--494, 2017.

\bibitem{MR4156257}
Luis Montejano, Eric Pauli, Miguel Raggi, and Edgardo Rold\'{a}n-Pensado.
\newblock The graphs behind {R}euleaux polyhedra.
\newblock {\em Discrete Comput. Geom.}, 64(3):1013--1022, 2020.

\bibitem{MR1354145}
J\'{a}nos Pach and Pankaj~K. Agarwal.
\newblock {\em Combinatorial geometry}.
\newblock Wiley-Interscience Series in Discrete Mathematics and Optimization.
  John Wiley \& Sons, Inc., New York, 1995.
\newblock A Wiley-Interscience Publication.

\bibitem{MR296813}
G.~T. Sallee.
\newblock Reuleaux polytopes.
\newblock {\em Mathematika}, 17:315--323, 1970.

\bibitem{MR0087117}
S.~Straszewicz.
\newblock Sur un probl\`eme g\'{e}om\'{e}trique de {P}. {E}rd\"{o}s.
\newblock {\em Bull. Acad. Polon. Sci. Cl. III.}, 5:39--40, IV--V, 1957.

\end{thebibliography}

\bibliographystyle{plain}

\typeout{get arXiv to do 4 passes: Label(s) may have changed. Rerun}

\end{document}